\documentclass[a4paper,twoside,11pt,english,intlimits,openany]{article}

\usepackage{latexsym,amsfonts,amssymb,exscale,enumerate}
\usepackage{amsmath}
\usepackage{amsthm}
\usepackage{amssymb}
\usepackage{geometry}
\usepackage{fancyhdr,times,euler,euscript,color}
\usepackage{bbm}
\usepackage{hyperref}
\usepackage{etex}
\usepackage{multicol}
\usepackage{titlesec}
\usepackage{comment}
\usepackage{appendix}
\usepackage{etoolbox}
\makeatletter
\patchcmd{\ttlh@hang}{\parindent\z@}{\parindent\z@\leavevmode}{}{}
\patchcmd{\ttlh@hang}{\noindent}{}{}{}
\makeatother

\usepackage{graphicx}

\usepackage[all]{xy}
\SelectTips{cm}{}

\usepackage{pstricks}

\psset{linewidth=0.3pt,dimen=middle}
\psset{xunit=.70cm,yunit=0.70cm}
\psset{arrowsize=1pt 5,arrowlength=0.6,arrowinset=0.7}
\usepackage{tikz}

\titleformat{\chapter}[frame]
{\normalsize}%
{\filright\sffamily\Large%
\enspace Chapitre \thechapter\enspace}%
{5pt}
{\rule{0pt}{30pt}\sffamily\Huge\bfseries\filcenter}

\titlespacing*{\chapter}{-1cm}{-2cm}{1cm}

\titleformat{\section}[block]{\scshape\filcenter\Large}{\thesection.}{.5em}{}
\titleformat{\subsection}[block]{\bfseries\filcenter\large}{\thesubsection.}{.5em}{}
\titleformat{\subsubsection}[runin]{\bfseries}{\thesubsubsection.}{.5em}{}[.]

\swapnumbers
  	\newtheorem{theorem}[subsubsection]{Theorem}
  	
  	\newtheorem{lemma}[subsubsection]{Lemma}
  	\newtheorem{proposition}[subsubsection]{Proposition}
\theoremstyle{definition}
	\newtheorem{definition}[subsubsection]{Definition}
	
	\newtheorem{remark}[subsubsection]{Remark}
	\newtheorem{example}[subsubsection]{Example}

\input{macros.tex}

\newcommand{\er}[2]{\,_{#1}{#2}}

\newcommand{\ere}[2]{\,_{#1}{#2}_{#1}}
\newcommand{\ER}{\er{E}{R}}

\newcommand{\ERE}{\ere{E}{R}}
\newcommand{\cl}[1]{\overline{#1}}
\newcommand{\Ker}{\text{Ker}}
\newcommand{\Mon}{\text{Mon}}

\newcommand{\deghecke}{\mathcal{AH}^{\text{deg}}}
\newcommand{\aob}{\mathcal{AOB}}

\def\dar[#1,#2,#3]{\ar@<0.8ex>[#1] ^{#3} \ar@<-0.8ex>[#1] _{#2} }

\newcommand{\Saux}{S^{\amalg}}

\newcommand{\Ccontext}{\mathbf{C} \mathcal{C}}

\newcommand{\Eo}{E^{\ell (1)}}
\newcommand{\Ro}{R^{\ell (1)}}
\newcommand{\So}{S^{\ell (1)}}

\newcommand{\tilda}[1]{\widetilde{\widehat{#1}}}

\newcommand{\catego}[1]{\mathbf{#1}}

\begin{document}
\thispagestyle{empty}

\begin{center}

\begin{Large}\begin{uppercase}
{Rewriting modulo isotopies in pivotal linear~$(2,2)$-categories.}
\end{uppercase}\end{Large}

\vskip+8pt

\bigskip\hrule height 1.5pt \bigskip

\vskip+10pt

\begin{large}\begin{uppercase}
{Benjamin Dupont}
\end{uppercase}\end{large}

\vskip+40pt

\begin{small}\begin{minipage}{14cm}
\noindent\textbf{Abstract --}
In this paper, we study rewriting modulo a set of algebraic axioms in categories enriched in linear categories, called linear~$(2,2)$-categories. We introduce the structure of linear~$(3,2)$-polygraph modulo as a presentation of a linear~$(2,2)$-category by a rewriting system modulo algebraic axioms. 
We introduce a symbolic computation method in order to compute linear bases for the vector spaces of $2$-cells of these categories. In particular, we study the case of pivotal $2$-categories using the isotopy relations given by biadjunctions on $1$-cells and cyclicity conditions on $2$-cells as axioms for which we rewrite modulo. By this constructive method, we recover the bases of normally ordered dotted oriented Brauer diagrams in te affine oriented Brauer linear~$(2,2)$-category.

\bigskip


\smallskip\noindent\textbf{M.S.C. 2010 -- 68Q42, 18D05.} 
\end{minipage}\end{small}
\end{center}

\vspace{0.8cm}

\begin{center}
\begin{small}\begin{minipage}{12cm}
\renewcommand{\contentsname}{}
\setcounter{tocdepth}{2}
\tableofcontents
\end{minipage}
\end{small}
\end{center}

\clearpage

\section*{Introduction}

In representation theory, many families of algebras admitting diagrammatic presentations by generators and relations emerged, as for instance Temperley-Lieb algebras \cite{TemperleyLieb71}, Brauer algebras \cite{Brauer37}, Birman-Wenzl algebras \cite{BirmanWenzl89}, Jones' planar algebras \cite{Jones99}, or Khovanov-Lauda-Rouquier algebras \cite{KL1}. These algebras can often be interpreted as categories enriched in linear categories, that is with a structure of $\mathbb{K}$-vector space (for a given field $\mathbb{K}$) on each space of $2$-cells between two $1$-cells. Such a $2$-category will be called a linear~$(2,2)$-category, and its $2$-cells will be depicted using string diagrammatic representation. In general, the presentations of these algebras admit a great number of relations, some of them being given by the algebraic structure, yielding a lack of efficient tools to compute in these structures. In particular, many of these categories have an additionnal pivotal structure, interpreted by duality on $1$-cells yielding unit and counit $2$-cells, diagrammatically represented by caps and cups satisfying isotopy relations. In this structure, two isotopic diagrams represent the same $2$-cell \cite{CKS00}, so that the computations are difficult to implement. Many linear~$(2,2)$-categories arising in representation theory admit a pivotal structure, such as for instance the category of $\mathfrak{gl}_n$-webs encoding the representation theory of the Lie algebra $\mathfrak{gl}_n$ \cite{CKM14,ELI15}, the Khovanov-Lauda-Rouquier $2$-categorification of a quantum group \cite{KL3,ROU08} and the Heisenberg categories categorifying the Heisenberg algebra \cite{KHO14}. The main objective of this paper is to present a symbolic computation method to compute in these linear~$(2,2)$-categories using the theory of rewriting. We use rewriting modulo the isotopy axioms of the pivotal structure to facilitate the analysis of confluence, and to compute linear bases for each space of $2$-cells using suited notions of termination and confluence modulo.

\subsubsection*{Algebraic rewriting and polygraphs}
The main idea beyond algebraic rewriting is to orient any algebraic equation, 
allowing to replace each occurence of the left-hand side by the right hand-side but not in the other way.
The objective is to combine the usual theory of abstract rewriting
with the algebraic axioms of the ambient structure to deduce algebraic properties of this latter. In this paper, we study presentations of higher-dimensional categories presented by generators and relations using rewriting systems introduced independently by Burroni under the name of polygraphs \cite{BUR93} or by Street under the name computads \cite{STR86,STR87}. We refer the reader to \cite{GM18} for more details on polygraphs and their rewriting properties. Rewriting with polygraphs has for instance been used to compute coherent presentations of higher-dimensional categories \cite{GM09}, to obtain homological and homotopical properties using Squier's theorems \cite{GM12,GM18}, to prove Koszulness of algebras \cite{GHM17} or to compute explicit linear bases of vector spaces or algebras \cite{GHM17} or higher-dimensional linear~$(2,2)$-categories \cite{AL16}.
The context of linear rewriting introduced by Guiraud, Hoffbeck and Malbos in \cite{GHM17} for associative algebras has been extended to the context of linear~$(2,2)$-categories by Alleaume \cite{AL16},
where those categories were presented by linear~$(3,2)$-polygraphs,
for which we recall the rewriting properties of termination and confluence that we will use throughout this paper in \ref{SS:LinearRewriting}. One of the main objectives of this paper is to extend Alleaume's results to linear rewriting modulo a part of the relations or axioms, in order to facilitate the analysis of confluence of linear~$(3,2)$-polygraphs, as explained in the sequel.

\subsubsection*{Linear rewriting modulo}
In this work, we use rewriting modulo theory in order to rewrite with a set of oriented rules $R$ modulo the set $E$ of axioms of the inherent 
algebraic structure. Rewriting modulo makes the confluence easier to prove for two reasons.
First of all, the family of critical branchings that should be considered in the analysis of confluence is reduced since one should not consider critical branchings implying a defining rule and an axiomatic rule. Besides, rewriting modulo non-oriented relations allows more flexibility in computations when reaching confluence. 
In particular, we use in this paper rewriting modulo the isotopy axioms of a pivotal linear~$(2,2)$-category, whose structure is recalled in Section \ref{SS:PivotalCategories}. These relations create obstructions to prove confluence, as they lead to a huge number of critical branchings. However, critical branchings implying a relation of the linear~$(2,2)$-category and an isotopy relation should not be taken into account. Following \cite[Section 3.4]{DMpp18}, these branchings does not need to be considered when rewriting with the suited rewriting modulo system. 

In the literature, three different paradigms of rewriting modulo are well-known.
The first one is to consider the rewriting system $\ERE$ consisting in rewriting with $R$ on congruence classes modulo $E$. However, the analysis of confluence of $\ERE$ appears inefficient in general, since the reducibility of an equivalence class needs to explore all the class.
Another approach is using Huet's definition of confluence modulo \cite{Huet80}, in which rewriting sequences involve only oriented rules and no equivalence steps, but sources and targets in confluence diagrams are not required to be equal but congruent modulo $E$. 
Jouannaud and Kirchner enlarged this approach in \cite{JouannaudKirchner84} by defining rewriting properties for any rewriting system modulo~$S$ such that $R\subseteq S \subseteq \ERE$.
In \cite{DMpp18}, a categorical and polygraphic model was introduced to rewrite modulo in higher-dimensional categories, following Huet and  Jouannaud and Kirchner's approaches. An abstract local confluence criteria  and a critical pair lemma were proved in this setting, both requiring a termination hypothesis for $\ERE$. In Section \ref{SS:LinearPolMod}, we extend these constructions to the linear context, in particular we introduce the notion of linear~$(3,2)$-polygraphs modulo. We precise in Section \ref{SS:ConfluenceModulo} the property of confluence modulo in this context, and give a classification of branchings modulo, which is bigger than in \cite{DMpp18} since some branchings come from the linear structure.

There are two main difficulties when rewriting in linear structures: first of all, we have to specify allowed rewritings in order to avoid non-termination due to the linear context, \cite{GHM17}. We consider in this paper a similar notion of rewriting step for the set $R$ of rewriting rules. The second difficulty is that proving local confluence from confluence of critical branchings require a termination assumption, see \cite[Section 4.2]{GHM17}. Indeed, some branchings that would be trivially confluent if all rewritings were allowed may become non-confluent because of the restriction on the set of rewritings. In Section \ref{SS:ConfluenceModulo}, we extend the critical pair lemma modulo of \cite{DMpp18} for linear~$(3,2)$-polygraphs modulo to get local confluence from confluence of critical branchings modulo, under the same termination assumption for $\ERE$.

\subsubsection*{Confluence modulo by decreasingness}
The termination of the polygraph modulo $\ERE$ is in general difficult to prove. In particular, this is the case when considering linear~$(3,2)$-polygraphs modulo presenting pivotal linear~$(2,2)$-categories, due to the existence of $2$-cells with source and target the same identity $1$-cell, called \emph{bubbles}. Indeed, Alleaume enlighted the fact that linear~$(2,2)$-categories with bubbles that can go through strands can in general not be enriched with a monomial order, so that they can not be presented by terminating rewriting systems. Moreover, the cyclicity of a $2$-cell with respect to the biadjunctions of the pivotal structure implies that the dot picturing this $2$-cell can be moved around the cap and cup $2$-cells, eventually creating rewriting cycles and making termination fail. However, even if $\ERE$ is not terminating, in many cases it will be quasi-terminating, that is all infinite rewriting sequences are generated by cycles. Following \cite{DM19}, the termination assumption for $\ERE$ can be weakened to a quasi-termination assumption, in order to prove confluence modulo of a linear~$(3,2)$-polygraph modulo $(R,E,S)$ from confluence of its critical branchings modulo. We introduce in Section \ref{SS:ConfluenceByDecreasingness} a notion of decreasingness modulo for a linear~$(3,2)$-polygraph modulo following Van Oostrom's abstract decreasingness property \cite{VOO94}. We then establish the following result:
\begin{quote}
\noindent{\bf Theorem \ref{T:ConfluenceByDecreasingness}}.
\emph{
Let $(R,E,S)$ be a left-monomial linear~$(3,2)$-polygraph modulo. If $(R,E,S)$ is decreasing modulo $E$, then $S$ is confluent modulo $E$.}
\end{quote}
	
The property of decreasingness modulo is defined by the existence of a well-founded labelling on the rewriting steps of a linear~$(3,2)$-polygraph modulo $(R,E,S)$, for which we require that all labels on the cells of $E$ are trivial, and such that labels are strictly decreasing on confluence modulo diagrams as defined in Section \ref{SS:ConfluenceByDecreasingness}. When $\ERE$ is quasi-terminating, there exists a particular labelling counting the distance between a $2$-cell and a fixed quasi-normal form, that is a $2$-cell from which we can only get rewriting cycles. Proposition \ref{P:DecreasingnessFromCriticalPairs}, proved in \cite{DM19}, shows that we can prove decreasingness by proving that all the critical branchings modulo $E$ are decreasing with respect to this quasi-normal form labelling.

\subsubsection*{Linear bases from confluence modulo}
The objective of this paper is to compute a hom-basis of a linear~$(2,2)$-category $\mathcal{C}$ presented by generators and relations, that is a family of sets $(\mathcal{B}_{p,q})$ indexed by pairs $(p,q)$ of $1$-cells such that $\mathcal{B}_{p,q}$ is a linear basis of the vector space $\mathcal{C}_2(p,q)$ of $2$-cells of $\mathcal{C}$ with $1$-source $p$ and $1$-target $q$. In \cite{AL16}, Alleaume proved following the constructions of \cite{GHM17} that such a basis may be obtained from a finite convergent presentation, considering all the monomials in normal form. In the context of rewriting modulo, there are two different degrees of normal forms. First of all, we require that the linear~$(3,2)$-polygraphs modulo $(R,E,S)$ is either normalizing or quasi-terminating so that one can either speak of normal form or quasi-normal form with respect to $S$. Then, one can also consider normal forms with respect to the polygraph $E$ for which we rewrite modulo, that we require to be convergent. We say that a normal form for $(R,E,S)$ is a $2$-cell appearing in the monomial decomposition of the $E$-normal form of a monomial in normal form with respect to $S$. In Section \ref{SS:LinearBasesModulo}, we give a method to compute a hom-basis of a linear~$(2,2)$-category from an assumption of confluence modulo some relations. Namely, let $\mathcal{C}$ be a pivotal linear~$(2,2)$-category, $P$ a linear~$(3,2)$-polygraph modulo presenting $\mathcal{C}$, and $(E,R)$ a convergent splitting of $P$, given by a couple $(E,R)$ of linear~$(3,2)$-polygraphs where $E$ is convergent and contains all the $3$-cells corresponding to the axioms of pivotality and $R$ contains the remaining relations. This data allows to define polygraphs modulo $(R,E,S)$, and we prove the following result:
\begin{quote}
\noindent{\bf Theorem \ref{T:BasisByConfluenceModulo}}.
\emph{
Let $P$ be a linear~$(3,2)$-polygraph presenting a linear~$(2,2)$-category $\C$, $(E,R)$ a convergent splitting of $P$ and $(R,E,S)$ a linear~$(3,2)$-polygraph modulo such that 
\begin{enumerate}[{\bf i)}]
\item $S$ is normalizing,
\item $S$ is confluent modulo $E$,
\end{enumerate}
then the set of normal forms for $(R,E,S)$ is a hom-basis of $\mathcal{C}$.}
\end{quote}

This result is then extented to the case where $S$ is quasi-terminating, where we define a quasi-normal form for the linear~$(3,2)$-polygraph modulo $(R,E,S)$ as a monomial appearing in the monomial decomposition of the $E$-normal form of a monomial in $\text{Supp}(\cl{u})$, where $\cl{u}$ is the fixed quasi-normal form of a $2$-cell $u$. 

\begin{quote}
\noindent{\bf Theorem \ref{T:BasisByConfluenceModulob}}.
\emph{
With the same assumptions than in Theorem \ref{T:BasisByConfluenceModulo}, if 
\begin{enumerate}[{\bf i)}]
\item $S$ is quasi-terminating,
\item $S$ is confluent modulo $E$,
\end{enumerate}
then the set of quasi-normal forms for $(R,E,S)$ is a hom-basis of $\mathcal{C}$.}
\end{quote}

\subsubsection*{The affine oriented Brauer category}
As an illustration, we apply Theorem \ref{T:BasisByConfluenceModulob} to compute a hom-basis for the affine oriented Brauer category, which is a monoidal category (seen as a $2$-category with only one object) obtained from the free symmetric monoidal category generated by a single $1$-cell and one additional dual, and adjoined with a polynomial generator subject to apropriate relations. This category is already well-understood, and a hom-basis has already been provided by Brundan, Comes, Nash and Reynolds \cite{BCNR14}. This linear~$(2,2)$-category has also already been studied by Alleaume using rewriting methods in \cite{AL16}. In this paper, we improve Alleaume's approach using rewriting modulo the isotopy axioms of the pivotal structure, as there are less critical branchings to consider in the proof of confluence. With this method, we recover Brundan, Comes, Nash and Reynolds' basis, that was also obtained by Alleaume in the non-modulo setting. More precisely, we define a linear~$(3,2)$-polygraph $\cl{\aob}$ presenting the linear~$(2,2)$-category $\aob$, for which we rewrite modulo its convergent subpolygraph of isotopy relations. We prove the following result:
\begin{quote}
\noindent{\bf Theorem \ref{T:QuasiTermConfModAOB}}.
\emph{ Let $(R,E)$ be the convergent splitting of $\aob$ defined in Section \ref{SSS:LinearPolAOB}, then $\ER$ is quasi-terminating and $\ER$ is confluent modulo $E$.}
\end{quote}
We prove that $\ER$ is quasi-terminating in Section \ref{SS:Quasi-TermAOB}, and confluent modulo $E$ in Section \ref{SS:ConfluenceModuloAOB}. As a consequence, we obtain a hom-basis for $\aob$, which coincides for an apropriate choice of quasi-normal forms for $\ER$ with the basis of normally ordered dotted orienter Brauer diagrams with bubbles of \cite{BCNR14}.

\subsubsection*{Organization of the paper}
In the first section of this paper, we recall from \cite{AL16} notions about linear~$(2,2)$-categories and linear~$(3,2)$-polygraphs. We recall the definition of a pivotal category and the isotopies of string diagrams coming from this structure. We define a linear~$(3,2)$-polygraph of isotopies that we use throughout this paper as a prototypical example to rewrite modulo. In Subsection \ref{SS:TerminationDerivation}, we recall from \cite{GUI06,GM09} a method to prove termination for a $3$-dimensional polygraph using a notion of derivation on a $2$-category. This method will be used in all termination and quasi-termination proofs in this paper. In Section \ref{S:LinearRewritingModulo}, we extend the previous rewriting properties and results in the context of rewriting modulo. We introduce a notion of linear~$(3,2)$-polygraph modulo and of rewriting steps taking into account the restriction of rewritings to positive $3$-cells. We introduce following \cite{Huet80,JouannaudKirchner84,DMpp18} a notion of confluence modulo for these polygraphs, and we extend the critical pair lemma from \cite{DMpp18} in the linear setting. We introduce a notion of decreasingness modulo from Van Oostrom's abstract decreasingness property \cite{VOO94}, and we prove that decreasingness modulo implies confluence modulo. In the last part of this section \ref{SS:LinearBasesModulo}, we describe a method to compute a hom-basis of a linear~$(2,2)$-category from terminating or quasi-terminating linear~$(3,2)$-polygraphs modulo satisfying an assumption of confluence modulo. In the last section of this paper, Section \ref{S:AOB}, we use this procedure to compute a hom-basis for the linear~$(2,2)$-category $\aob$. We explicit a presentation of this category, and fix a linear~$(3,2)$-polygraph presenting it. We prove that the linear~$(3,2)$-polygraph $\ER$ obtained by considering in $E$ the convergent subpolygraph of isotopies and in $R$ the remaining relations is quasi-terminating, and that $\ER$ is confluent modulo $E$ by proving that its critical branchings are decreasing with respect to the quasi-normal form labelling and Proposition \ref{P:DecreasingnessFromCriticalPairs}.

\section{Preliminaries}
\label{S:Preliminaries}
In this section, we fix a field $\mathbb{K}$. If $\mathcal{C}$ is an $n$-category, we denote by $\mathcal{C}_n$ the set of $n$-cells in $\mathcal{C}$. For any $0 \leq k < n$ and any $k$-cells $p$ and $q$ in $\mathcal{C}$, we denote by $\mathcal{C}_{k+1}(p,q)$ the set of $(k+1)$-cells in $\C$ with $k$-souce $p$ and $k$-target $q$. If $p$ is a $k$-cell of $\mathcal{C}$, we denote respectively by $s_i(p)$ and $t_i(p)$ the $i$-source and $i$-target of $p$ for $0 \leq i \leq k-1$. The source and target maps satisfy the globular relations
\[ s_i \circ s_{i+1} = s_i \circ t_{i+1} \quad \text{and} \quad t_i \circ s_{i+1} = t_i \circ t_{i+1} \] 
for any $0 \leq i \leq n-2$. Two $k$-cells $p$ and $q$ are \emph{$i$-composable} when $t_i(p) = s_i(q)$. In that case, their $i$-composition is denoted by $p \star_i q$. The compositions of $\mathcal{C}$ satisfy the \emph{exchange relations}:
\[ (p_1 \star_i q_1) \star_j (p_2 \star_i q_2) = (p_1 \star_j p_2) \star_i (q_1 \star_j q_2) \]
for any $i < j$ and for all cells $p_1$,$p_2$,$q_1$,$q_2$ such that both sides are defined. If $p$ is a $k$-cell of $\mathcal{C}$, we denote by $1_{p}$ its identity $(k+1)$-cell.
A $k$-cell $p$ of $\mathcal{C}$ is \emph{invertible with respect to $\star_i$-composition} ($i$-invertible for short) when there exists a (necessarily unique) $k$-cell $q$ in $\mathcal{C}$ with $i$-source $t_i(p)$ and $i$-target $s_i(p)$ such that
\[ p \star_i q = 1_{s_i(p)} \quad \text{and} \quad q \star_i p = 1_{t_i(p)} \]

\noindent Throughout this paper, $2$-cells in $2$-categories are represented using the classical representation by string diagrams, see \cite{LAU12,SAV18} for surveys on the correspondance between $2$-cells and diagrams. The $\star_0$ composition of $2$-cells is depicted by placing two diagrams next to each other, the $\star_1$-composition is vertical concatenation of diagrams. In the sequel, the $0$-cells will be denoted by $x,y, \dots$, the $1$-cells will be denoted by $p,q, \dots$, the $2$-cells will be denoted by $u,v,\dots$ and the $3$-cell will be denoted either by $\alpha$,$\beta$,$\dots$ when rewriting in a non modulo setting and $f$,$g$, $\dots$ when rewriting modulo.

\subsection{Rewriting in linear~$(2,2)$-categories} 
\label{SS:LinearRewriting}
Recall from \cite{AL16} that a \emph{linear~$(n,p)$-category (over $\mathbb{K}$)} is an $n$-category such that for any $p \leq i \leq n$, the set $\mathcal{C}_i(p,q)$ is a $\mathbb{K}$-vector space for any $(i-1)$-cells $p$ and $q$ in $\mathcal{C}$. A linear~$(n,p)$-category is presented by a \emph{linear~$(n+1,p)$-polygraph}. 
In this section, we expand on the notions of linear~$(2,2)$-categories and linear~$(3,2)$-polygraphs over $\mathbb{K}$ and their rewriting properties.

\subsubsection{Linear $(2,2)$-categories}  A linear~$(2,2)$-category is a $2$-category $\C$ such that:
\begin{enumerate}[{\bf i)}]
\item for any $p$ and $q$ in $\C_1$, the set $\C_2(p,q)$ is a $\K$-vector space.  
\item for any $p,q,r$ in $\C_1$, the map $\star_1: \C_2(p,q) \otimes \C_2(q,r) \to \C_2(p,r)$ is $\K$-linear.
\item sources and target maps are compatible with the linear structure. 
\end{enumerate}

When $\mathcal{C}$ admits a presentation by generators and relations, a $2$-cell $\phi$ obtained from $\star_0$ and $\star_1$ compositions of generating $2$-cells is called a \emph{monomial} in $\C$. Any $2$-cell $\phi$ in $\C$ can be uniquely decomposed into a sum of monomials $\phi= \sum \phi_i$, which is called the \emph{monomial decomposition} of $\phi$. We also set the \emph{support} of $\phi$ to be the set $\{  \phi_i \rbrace$ of $2$-cells that appear in this monomial decomposition. It will be denoted by $\text{Supp}(\phi)$. Given a linear~$(2,2)$-category $\mathcal{C}$, a \emph{hom-basis} of $\mathcal{C}$ is a
family of sets $(\mathcal{B}_{p,q})_{p,q \in \mathcal{C}_1}$ indexed by pairs $(p,q)$ of $1$-cells of $\mathcal{C}$ such that for any $1$-cells $p$ and $q$, the set $\mathcal{B}_{p,q}$ is a linear basis of $\mathcal{C}_2(p,q)$.

\subsubsection{Contexts}
Following \cite{GM09}, a \emph{context} of a $2$-category $\mathcal{C}$ is a pair $(x,C)$, where $x= (f,g)$ is a $1$-sphere of $\mathcal{C}$ and $C$ is a $2$-cell in the $2$-category $\mathcal{C}[x]$ which is defined as the $2$-category with an additional $2$-cell filling the $1$-sphere $x$, such that $C$ contains exactly one occurence of $x$. Seeing the $1$-sphere $x$ as a formal $2$-cell denoted by $\square$, such a context in $\mathcal{C}$ has the shape
 \[ C = m_1 \star_1 ( m_2 \star_0 \square \star_0 m_3)\star_1 m_4, \] where the $m_i$ are $2$-cells in $\mathcal{C}$. For any $2$-cell $u$ in $\mathcal{C}_2$, the above $2$-cell where $\square$ is replaced by $u$ will be denoted by $C[u]$. 

A \emph{context} of a linear~$(2,2)$-category $\mathcal{C}$ has the shape 
\[  C = \lambda m_1 \star_1 ( m_2 \star_0 \square \star_0 m_3)\star_1 m_4 + u, \]
where the $m_i$ are monomials in $\mathcal{C}$, $\lambda$ is a scalar in $\mathbb{K}$ and $u$ is a $2$-cell in $\mathcal{C}$. The \emph{category of contexts of $\mathcal{C}$} is the category denoted by $\Ccontext$ whose objects are the $2$-cells of $\mathcal{C}$ and the morphisms from $f$ to $g$ are the contexts $(\partial (f), C)$ of $\mathcal{C}$ such that $C[f] = g$.

\subsubsection{Linear $(3,2)$-polygraphs}
We define the notion of \textit{linear $(3,2)$-polygraph} inductively:
\begin{enumerate}[{\bf i)}]
\item A \emph{$1$-polygraph} $P$ is a graph (with possible loops and multi-edges) with a set of vertices $P_0$ and a set of edges $P_1$ equipped with two applications $s_0^P$, $t_0^P$: $P_1 \to P_0$ (source and target). We denote by $P_1^{\ast}$ the free $1$-category generated by $P$.
\item A \emph{$2$-polygraph} is a triple $( P_0,  P_1 , P_2)$ where:
\begin{enumerate}[{\bf a)}]
\item $(P_0, P_1)$ is a $1$-polygraph,
\item $P_2$ is a globular extension of $P_1^\ast$, that is a set $P_2$ of $2$-cells equipped with two maps \[ s_1^P, \: t_1^P: P_2 \to P_1^* \] called respectively $1$-source and $1$-target such that the globular relations  $s_0^P \circ s_1^P = s_0^P \circ t_1^P $ and $ t_0^P \circ s_1^P = t_0^P \circ t_1^P $ hold.
\end{enumerate}
\item We construct $P_2^\ell$ the free linear~$(2,2)$-category generated by $P = ( P_0, P_1, P_2 )$ as the $2$-category whose $0$-cells are $P_0$, whose $1$-cells are elements of $P_1^\ast$ and for any $p,q$ in $P_1^\ast$, $P_2^\ell(p,q)$ is the free vector space on $P_2^*(p,q)$, where $P_2^\ast$ if the free $2$-category generated by $(P_0 , P_1 , P_2)$.
\item A \emph{linear~$(3,2)$-polygraph} is the data of $P = ( P_0, P_1, P_2, P_3 )$ where $( P_0, P_1, P_2)$ is a $2$-polygraph and $P_3$ is a globular extension of $P_2^\ell$, that is a set equipped with two maps 
\[ s_2^P, \: t_2^P: P_3 \to P_2^\ell \] respectively called $2$-source and $2$-target such that the globular relations hold: $s_0^P \circ s_1^P = s_0^P \circ t_1^P $, $ t_0^P \circ s_1^P = t_0^P \circ t_1^P $, $s_1^P \circ s_2^P = s_1^P \circ t_2^P$ and $t_1^P \circ s_2^P = t_1^P \circ t_2^P$.
Source and target maps are compatible with the linear structure. In particular, $s_2^P$ and $t_2^P$ are $\mathbb{K}$-linear.
\end{enumerate} 

Note that, when there is no ambiguity, the source and target maps will be simply denoted by $s_i$, $t_i$ for $0 \leq i \leq 2$. 
Let $P$ be a linear~$(3,2)$-polygraph, we denote by $P_{\leq k}$ the underlying $k$-polygraph of $P$, for $k = 1,2$. We denote by $P^\ell$ the free linear~$(3,2)$-category generated by $P$, as defined in \cite[Section 3.1]{ALPhD}. Recall from \cite[Proposition 1.2.3]{GHM17} that every $3$-cell $\alpha$ in $P^\ell$ is $2$-invertible, its inverse being given by $1_{s_2(\alpha)} - \alpha + 1_{t_2(\alpha)}$. The \textit{congruence} generated by $P$ is the equivalence relation $\equiv$ on $\Sl$ defined by 
\[ \text{$ u \equiv v$ if there is a $3$-cell $\alpha$ in $P^\ell$ such that $s_2(\alpha) = u$ and $t_2(\alpha)=v$}, \]
 We say that a linear~$(2,2)$-category $\mathcal{C}$ is presented by $P$ 
if $\mathcal{C}$ is isomorphic to the quotient category $\Sl \slash \equiv$. A $3$-cell in $P^\ell$ is said \emph{elementary} if it is of the form $\lambda m_1 \star_1 ( m_2 \star_0 \alpha \star_0 m_3)\star_1 m_4 + h$ where $\lambda$ is a non zero scalar, $\alpha$ is a generating $3$-cell in $P_3$, the $m_i$ are monomials of $P_2^\ell$ and $h$ is a $2$-cell in $P_2^\ell$. For a cellular extension $\Gamma$ of $P_1^\ast$, we will denote by $||f||_\Gamma$ the number of occurences of $2$-cells of $\Gamma$ in the $2$-cell $f$ in $P_2^\ast$. 

\subsubsection{Rewriting steps}
A \textit{rewriting step} of a linear~$(3,2)$-polygraph $P$ is a $3$-cell of the following form:
\[ C[\alpha]: \; C[s_2(\alpha)] \fl C[t_2(\alpha)] \]
where $\alpha$ is a generating $3$-cell in $P_3$, $C = \lambda m_1 \star_1 ( m_2 \star_0 \square \star_0 m_3)\star_1 m_4 + h$ is a context of the free linear~$(2,2)$-category $P_2^\ell$ generated by $P$ such that the monomial $m_1 \star_1 ( m_2 \star_0 s_2(\alpha) \star_0 m_3)\star_1 m_4$ does not appear in the monomial decomposition of $h$. Given a linear~$(3,2)$-polygraph $P$, we denote by $P_{\text{stp}}$ the set of rewriting steps of $P$.
A \textit{rewriting sequence} of $P$ is a finite or infinite sequence of rewriting steps of $P$. We say that a $2$-cell is a \textit{normal form} if it can not be reduced by any rewriting step. \\

A $3$-cell $\alpha$ of $P^\ell$ is called \emph{positive} if it is a $\star_2$-composition $\alpha= \alpha_1 \star_2 \dots \star_2 \alpha_n$ of rewriting steps of $P$. The \emph{length} of a positive $3$-cell $\alpha$ in $P^\ell$ is the number of rewriting steps of $P$ needed to write $\alpha$ as a $\star_2$-composition of these rewriting steps. We denote by $P^{\ell (1)}$ the set of positive $3$-cells of $P$ of length $1$. A proof similar to \cite[Lemma 3.1.3]{GHM17} yields the following result, that we will use in the proof of the critical pair lemma modulo in the sequel:

\begin{lemma} \label{L:CellDecomposition}
Let $\alpha$ be an elementary $3$-cell in $P^{\ell}$, then there exist two rewriting sequences $\beta$ and $\gamma$ of $P$ of length at most $1$ such that $\alpha = \beta \star_2 \gamma^-$.
\end{lemma}

\subsubsection{Rewriting properties}
A \emph{branching} (resp. \emph{local branching}) of a linear~$(3,2)$-polygraph $P$ is a pair of rewriting sequences (resp. rewriting steps) of $P$ which have the same $2$-cell as $2$-source.
Such a branching is \emph{confluent} if it can be completed by rewriting sequences $f'$ and $g'$ of $P$ as follows:
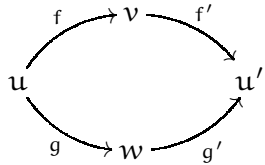
\[ \xymatrix @C=2.6em@R=1.2em{
& v
	\ar @/^1.5ex/ [dr] ^-{f'}
\\
u 
	\ar @/^1.5ex/ [ur] ^-{f}
	\ar @/_1.5ex/ [dr] _-{g}
&& u'
\\
& w
	\ar @/_1.5ex/ [ur] _-{g'}
}	\]

\noindent A linear~$(3,2)$-polygraph $P$ is said:
\begin{enumerate}[{\bf i)}]
\item \emph{terminating} if there is no infinite rewriting sequences in $P$.
\item \emph{confluent} if all the branchings of $P$ are confluent.
\item \emph{convergent} if it is both terminating and confluent.
\item \emph{left (resp. right) monomial}  if every source (resp. target) of a $3$-cell is a monomial. In the sequel, we will only consider left-monomial linear~$(3,2)$-polygraphs.
\end{enumerate}

The local branchings of linear~$(3,2)$-polygraphs can be classified in four different forms, see \cite[Section 4.2]{AL16}. An \emph{aspherical branching} of $P$ is a branching of the form $$ t(\alpha) \leftarrow s(\alpha) \fl t(\alpha).$$
A \emph{Peiffer branching} is a branching formed with two rules which does not overlap: 
$$ t_2(\alpha \star_1 s_2(\beta) + h \leftarrow s_2( \alpha ) \star_1 s_2(\beta) + h \fl s_2(\alpha) \star_1 t_2(\beta) + h. $$
An \textit{additive branching} is a branching of the form
$$ t_2(\alpha) + s_2(\beta)  \leftarrow  s_2(\alpha) + s_2(\beta) \fl s_2(\alpha) + t_2(\beta). $$

\emph{Overlapping} branchings are all the remaining local branchings. We define an order on monomials of $P_2^\ell$ by $f \sqsubseteq g$ if there exists a context $C$ of $P_2^\ast$ such that $g=C[f]$. A \textit{critical branching} of $P$ is an overlapping branching of $P$ which is minimal for $\sqsubseteq$. In \cite[Section 5.1]{GM09}, Guiraud and Malbos classified all the different forms of critical branchings on such diagrams, in a non linear case. In this paper, we only study left-monomial linear~ $(3,2)$-polygraphs so that the classification of overlapping branchings is the same. 
There are \emph{regular} critical branchings corresponding to the overlapping of two reductions on a common part of a diagram, \emph{inclusion} critical branchings corresponding to the inclusion of a source of a rule in the source of another rule, and \emph{right, left, multi-indexed} critical branchings obtained by plugging an aditionnal $2$-cell in a diagram where $2$ reductions already overlap, see \cite{GM09} for more details. Recall from \cite{LAF03,GM09} that it suffices to check the confluence of indexed branchings for the instances of $2$-cells in normal form with respect to $P$.

There exists an adaptation of the linear critical pair lemma of \cite{GHM17} to linear~$(3,2)$-polygraphs, allowing to prove confluence from confluence of the critical branchings: namely if $P$ is a left-monomial and terminating linear~$(3,2)$-polygraph, it is confluent if and only if its critical branchings are confluent, \cite[Theorem 4.2.13]{AL16}. 
Given a convergent presentation of a linear~$(2,2)$-category $\mathcal{C}$ by a linear~$(3,2)$-polygraph $P$, \cite[Proposition 4.2.15]{AL16} states that for any $1$-cells $p$ and $q$ in $\mathcal{C}_1$, the set of monomial $2$-cells with source $p$ and target $q$ in normal form with respect to $P$  form a linear basis of the space $\mathcal{C}_2(p,q)$ . In the sequel, we will extend this basis result to a context of rewriting modulo.

\subsection{Pivotal linear~$(2,2)$-categories}
\label{SS:PivotalCategories}
In this subsection, we recall the structure of pivotal linear~$(2,2)$-category, which is a linear~$(2,2)$-category in which all the diagrammatic $2$-cells are drawn up to isotopy. 

\subsubsection{Adjunctions in a $2$-categories}
\label{SSS:AdjunctionsIn2Cat}
Let $\mathcal{C}$ be a $2$-category with sets of $0$-cells, $1$-cells and $2$-cells respectively denoted by
$\mathcal{C}_0$, $\mathcal{C}_1$ and $\mathcal{C}_2$. For any $1$-cell $p$ in $\mathcal{C}_1$,  a \emph{right adjoint} of $p$ is a $1$-cell $\hat{p}: y \to x$ equipped with two $2$-cells  $\varepsilon$ and $\eta$ in $\mathcal{C}$ defined as follows:
\begin{equation*} 
 \vcenter{\xy
 (0,0)*+{\scs y}="4";
 (-12,0)*+{\scs x}="m";
 (-24,0)*+{\scs y}="6";
 {\ar^{p} "4";"m"}; {\ar^{\hat{p}} "m";"6"};
 {\ar@/_1.95pc/_{1_y} "4";"6"};
 {\ar@{=>}_<<{\scriptstyle \varepsilon} (-12,2)*{};(-12,7)*{}};
\endxy}
\qquad \qquad
 \vcenter{\xy
 (12,0)*+{\scs x}="r";
 (0,0)*+{\scs y}="4";
 (-12,0)*+{\scs x}="m";
 {\ar^{\hat{p}} "4";"m"};
 {\ar@/^1.95pc/^{1_x} "r";"m"};
 {\ar@{=>}_<<{\scriptstyle \eta}(0,-7)*{};(0,-2)*{}};
 {\ar^{p} "r";"4"};
\endxy},
\end{equation*}
called the \emph{counit} and \emph{unit} of the adjunction, such that the equalities
\begin{equation*} \nn
 \scalebox{0.9}{\xy
 (12,0)*+{\scs x}="r";
 (0,0)*+{\scs y}="4";
 (-12,0)*+{\scs x}="m";
 (-24,0)*+{\scs y}="6";
 {\ar^{\hat{p}} "4";"m"}; {\ar^{p} "m";"6"};  {\ar^{p} "r";"4"};
 {\ar@/^1.95pc/^{1_x} "r";"m"};
 {\ar@/_1.95pc/_{1_y} "4";"6"};
 {\ar@{=>}_<<{\scriptstyle \eta}(0,-7)*{};(0,-2)*{}};
 {\ar@{=>}_<<{\scriptstyle \varepsilon} (-12,2)*{};(-12,7)*{}};
\endxy}
\quad = \quad
\scalebox{0.9}{\xy (8,0)*+{\scs x}="4"; (-8,0)*+{\scs y}="6"; {\ar@/^1.65pc/^{p}
"4";"6"}; {\ar@/_1.65pc/_{p} "4";"6"}; {\ar@{=>}_<<<{ 1_{p}} (-.5,-3)*{};(-.5,3)*{}} ;
\endxy}
\qquad \text{and} \qquad 
 \scalebox{0.9}{\xy
 (12,0)*+{\scs y}="r";
 (0,0)*+{\scs x}="4";
 (-12,0)*+{\scs y}="m";
 (-24,0)*+{\scs x}="6";
 {\ar^{p} "4";"m"}; {\ar^{\hat{p}} "m";"6"};  {\ar^{\hat{p}} "r";"4"};
 {\ar@/^1.95pc/^{1_y} "4";"6"};
 {\ar@/_1.95pc/_{1_x} "r";"m"};
 {\ar@{=>}_<<{\scriptstyle \eta}(-12,-7)*{};(-12,-2)*{}};
 {\ar@{=>}_<<{\scriptstyle \varepsilon} (0,2)*{};(0,7)*{}};
\endxy}
 \quad = \quad
\scalebox{0.9}{\xy (8,0)*+{\scs y}="4"; (-8,0)*+{\scs x}="6";
 {\ar@/^1.65pc/^{\hat{p}} "4";"6"};
 {\ar@/_1.65pc/_{\hat{p}} "4";"6"};
 {\ar@{=>}_<<<{ 1_{\hat{p}}} (-.5,-3)*{};(-.5,3)*{}} ; \endxy}
\end{equation*}
hold. We denote the fact that $p$ is a left adjoint of $\hat{p}$ by $p \dashv \hat{p}$. In a string diagrammatic notation, these units and counits are represented by caps and cups as follows:
\[ 
\eta: \:
\cupdb{p}{\hat{p}}{y}{x} , \qquad \varepsilon:  \:\capdb{\hat{p}}{p}{y}{x} \]
 The axioms of an adjunction require that the following equalities hold:
\begin{equation} 
\label{adjunctionref}
\begin{tikzpicture}[baseline = 0,scale=1.311]
  \draw[-,black,thick] (0.3,0) to (0.3,.4);
	\draw[-,black,thick] (0.3,0) to[out=-90, in=0] (0.1,-0.4);
	\draw[-,black,thick] (0.1,-0.4) to[out = 180, in = -90] (-0.1,0);
	\draw[-,black,thick] (-0.1,0) to[out=90, in=0] (-0.3,0.4);
	\draw[-,black,thick] (-0.3,0.4) to[out = 180, in =90] (-0.5,0);
  \draw[-,black,thick] (-0.5,0) to (-0.5,-.4);
  \node at (-0.3,0.65) {$x$};
  \node at (0.1,-0.65) {$y$};
  \node at (-0.5,-0.6) {$p$};
\end{tikzpicture}  = \; \begin{tikzpicture}[baseline=0, scale=1.311]
\draw[-,black,thick] (0,-0.4) to (0,0.4) ;
\node at (0,-0.65) {$p$} ;
\node at (0.3,0) {$y$};
\node at (-0.3,0) {$x$};
\end{tikzpicture} ,
\qquad 
\begin{tikzpicture}[baseline = 0, scale=1.311]
  \draw[-,black, thick] (0.3,0) to (0.3,-.4);
	\draw[-,black, thick] (0.3,0) to[out=90, in=0] (0.1,0.4);
	\draw[-,black, thick] (0.1,0.4) to[out = 180, in = 90] (-0.1,0);
	\draw[-,black, thick] (-0.1,0) to[out=-90, in=0] (-0.3,-0.4);
	\draw[-,black, thick] (-0.3,-0.4) to[out = 180, in =-90] (-0.5,0);
  \draw[-,thick,black, thick] (-0.5,0) to (-0.5,.4);
   \node at (-0.3,-0.65) {$y$};
  \node at (0.1,0.65) {$x$};
  \node at (0.4,-0.65) {$\hat{p}$};
\end{tikzpicture}   = \;  \begin{tikzpicture}[baseline=0, scale=1.311]
\draw[-,thick,black] (0,-0.4) to (0,0.4) ;
\node at (0,-0.65) {$\hat{p}$} ;
\node at (0.3,0) {$x$};
\node at (-0.3,0) {$y$};
\end{tikzpicture} 
\end{equation}
are satisfied. When we are in the situation where $\hat{p}$ is also a left-adjoint of $p$, that is $p$ and $\hat{p}$ are biadjoint, that we denote by $p \dashv \hat{p} \dashv p$, there are other unit and counit $2$-cells $\eta '$ and $\varepsilon '$ for this adjunction and the symmetric zig-zag relations for these $2$-cells hold similarly.

\subsubsection{Cyclic $2$-cells and pivotality}
\label{SSS:Pivotality}
Given a pair of $1$-cells $p,q\maps x \to y$ with chosen biadjoints $(\hat{p},
\eta_p, \hat{\eta}_p, \varepsilon_p, \hat{\varepsilon}_p)$ and $(\hat{q}, \eta_q,
\hat{\eta}_q, \varepsilon_q, \hat{\varepsilon}_q)$, then any 2-cell $\alpha
\maps p \To q$ has two \emph{duals} $ ^*\alpha , \alpha^* \maps \hat{q}\To
\hat{p}$, one constructed using the left adjoint structure, the other
using the right adjoint structure. Diagrammatically the two duals are given by
\[
 ^* \alpha := \begin{tikzpicture}[baseline = 0,scale=1.311]
  \draw[-,black, thick] (0.3,0) to (0.3,.4);
	\draw[-,black, thick] (0.3,0) to[out=-90, in=0] (0.1,-0.4);
	\draw[-,black, thick] (0.1,-0.4) to[out = 180, in = -90] (-0.1,0);
	\draw[-,black, thick] (-0.1,0) to[out=90, in=0] (-0.3,0.4);
	\draw[-,black, thick] (-0.3,0.4) to[out = 180, in =90] (-0.5,0);
  \draw[-,black, thick] (-0.5,0) to (-0.5,-.4);
  \node at (-0.25,0.65) {$\varepsilon_q$};
  \node at (0.13,-0.65) {$\eta_p$};
  \node at (-0.1,0) {$\bullet$};
  \node at (0.05,0) {$\scriptstyle{\alpha}$};
  \node at (0.35,0.65) {$\hat{p}$};
  \node at (-0.5,-0.65) {$\hat{q}$};
\end{tikzpicture}
 \qquad \qquad \qquad
 \alpha ^* := \begin{tikzpicture}[baseline = 0, scale=1.311]
  \draw[-,black, thick] (0.3,0) to (0.3,-.4);
	\draw[-,black, thick] (0.3,0) to[out=90, in=0] (0.1,0.4);
	\draw[-,black, thick] (0.1,0.4) to[out = 180, in = 90] (-0.1,0);
	\draw[-,black, thick] (-0.1,0) to[out=-90, in=0] (-0.3,-0.4);
	\draw[-,black, thick] (-0.3,-0.4) to[out = 180, in =-90] (-0.5,0);
  \draw[-,thick,black, thick] (-0.5,0) to (-0.5,.4);
   \node at (-0.1,0) {$\bullet$};
  \node at (-0.25,0) {$\scriptstyle{\alpha}$};
  \node at (-0.5,0.65) {$\hat{p}$};
   \node at (-0.2,-0.65) {$\hat{\varepsilon}_p$};
  \node at (0.1,0.65) {$\hat{\eta}_q$};
  \node at (0.35,-0.65) {$\hat{q}$};
\end{tikzpicture}
 \]
We will call $\alpha^{\ast}$ the right dual of $\alpha$ because it is
obtained from $\alpha$ as its mate using the right adjoints of $p$ and $q$.
Similarly, $ ^{*}\alpha$ is called the left dual of $\alpha$ because it is
obtained from $\alpha$ as its mate using the left adjoints of $p$ and $q$.

In general there is no reason why $ ^*\alpha$ should be equal to $\alpha^*$, see \cite{LAU12} for a simple
counterexample.

\begin{definition}[\cite{CKS00}]
Given biadjoints $(p,
\hat{p},\eta_p,\hat{\eta}_p,\varepsilon_p,\hat{\varepsilon}_p)$ and $(q,
\hat{q},\eta_q,\hat{\eta}_q,\varepsilon_q,\hat{\varepsilon}_q)$ and a $2$-cell
$\alpha \maps p \To q$ define
$\alpha^* := \hat{p}\hat{\eta}_q \star_1 \hat{\varepsilon}_p\hat{q} \nn $ and $ ^*\alpha := \varepsilon_G \hat{p} \star_1 \hat{}\eta_q$
as above. Then a $2$-cell $\alpha$ is called a {\em cyclic $2$-cell} if the equation $
^*\alpha=\alpha^*$ is satisfied, or either of the equivalent conditions
$^{\ast\ast}\alpha=\alpha$ or $\alpha^{\ast\ast} = \alpha$ are satisfied.
\end{definition}

A $2$-category in which all the $2$-cells are cyclic with respect to some biadjunction is called a \emph{pivotal} 
$2$-category. In this structure, the following theorem states that $2$-cells are represented up to isotopy:
\begin{theorem}[\cite{CKS00}] \label{T:TheoremIsotopy}
Given a string diagram representing a cyclic $2$-cell, between $1$-cells with chosen biadjoints, then any isotopy of the diagram represents the same $2$-cell. 
\end{theorem} 
 
\begin{example}
We consider a $2$-category with only one $0$-cell, two $1$-cells $E$ and $F$ whose identites are respectively represented by upward and downward arrows and such that $E \dashv F \dashv E$, that is $E$ and $F$ are biadjoint. We denote respectively by \:
\raisebox{-1mm}{$\caplsl{}$}, \: $\cuprsl{}$, \: $\caprsl{}$, \: \raisebox{-1mm}{$\cuplsl{}$} the units and counits for these adjunctions. Assume that this category has $2$-morphisms given by \: $\identdotusl{}$, \: $\identdotdsl{}$, \: \raisebox{-7mm}{$\crossingup{}{}$}, \: \raisebox{-7mm}{$\crossingdn{}{}$}
Then, requiring that the $2$-morphisms are cyclic in this $2$-category are made by the following equalities:
\[
\suldsl{} \: = \dpdsl{} \: =  \sdrdsl{} \:, \qquad 
\xy 0;/r.13pc/:
  (0,0)*{\xybox{
    (-4,-4)*{};(4,4)*{} **\crv{(-4,-1) & (4,1)}?(1)*\dir{>};
    (4,-4)*{};(-4,4)*{} **\crv{(4,-1) & (-4,1)};
     (4,4)*{};(-18,4)*{} **\crv{(4,16) & (-18,16)} ?(1)*\dir{};
     (-4,-4)*{};(18,-4)*{} **\crv{(-4,-16) & (18,-16)} ?(1)*\dir{}?(0)*\dir{};
     (18,-4);(18,12) **\dir{-};(12,-4);(12,12) **\dir{-};
     (-18,4);(-18,-12) **\dir{-};(-12,4);(-12,-12) **\dir{-};
     (8,1)*{};
     (-10,0)*{};(10,0)*{};
      (4,-4)*{};(12,-4)*{} **\crv{(4,-10) & (12,-10)}?(1)*\dir{}?(0)*\dir{};
      (-4,4)*{};(-12,4)*{} **\crv{(-4,10) & (-12,10)}?(1)*\dir{}?(0)*\dir{>};
     }};
  \endxy
 \: = \: \raisebox{-8mm}{$\crossingdn{}{}$} \: = \:
 \xy 0;/r.13pc/:
  (0,0)*{\xybox{
    (4,-4)*{};(-4,4)*{} **\crv{(4,-1) & (-4,1)}?(1)*\dir{>};
    (-4,-4)*{};(4,4)*{} **\crv{(-4,-1) & (4,1)};
     (-4,4)*{};(18,4)*{} **\crv{(-4,16) & (18,16)} ?(1)*\dir{};
     (4,-4)*{};(-18,-4)*{} **\crv{(4,-16) & (-18,-16)} ?(1)*\dir{}?(0)*\dir{};
     (-18,-4);(-18,12) **\dir{-};(-12,-4);(-12,12) **\dir{-};
     (18,4);(18,-12) **\dir{-};(12,4);(12,-12) **\dir{-};
     (8,1)*{ };
     (-10,0)*{};(10,0)*{};
     (-4,-4)*{};(-12,-4)*{} **\crv{(-4,-10) & (-12,-10)}?(1)*\dir{}?(0)*\dir{};
      (4,4)*{};(12,4)*{} **\crv{(4,10) & (12,10)}?(1)*\dir{}?(0)*\dir{>};
     }};
  \endxy \quad .
\]
and their mirror image through a reflection by a vertical axis.
\end{example}

\subsubsection{Linear~$(3,2)$-polygraphs of isotopies} 
\label{SSS:PolygraphIsotopy}
We define a $3$-polygraphs whose $3$-cells correspond of the isotopy axioms of a pivotal $2$-category, with respect to a set $I$ labelling the strands of the string diagrams and cyclic $2$-morphisms. Let $\mathcal{C}_I$ be the pivotal $2$-category defined by
\begin{itemize}
\item[-] a set $\mathcal{C}_0$ of $0$-cells denoted by $x,y, \ldots$
\item[-] two families of $1$-cells $E_i: x_i \fl y_i$ and $F_i: y_i \fl x_i$ indexed by $I$ such that $E_i \vdash F_i \vdash E_i$.
Note that the identity $2$-cells on $E_i$ and $F_i$ are respectively diagrammatically depicted by:
\[ 1_{E_i} := \begin{tikzpicture}[baseline = 0]
	\draw[-,thick,black] (0,-0.4) to (0,0.4);
     \draw[->,thick,black] (0,-0.2) to (0,0.2);
     \node at (0,-0.6) {$\scriptstyle{i}$};
     \node at (0.3,0) {$y_i$};
     \node at (-0.3,0) {$x_i$};
\end{tikzpicture}  \qquad \qquad 1_{F_i} := \begin{tikzpicture}[baseline = 0]
	\draw[-,thick,black] (0,-0.4) to (0,0.4);
     \draw[<-,thick,black] (0,-0.2) to (0,0.2);
     \node at (0,-0.6) {$\scriptstyle{i}$};
     \node at (0.3,0) {$x_i$};
     \node at (-0.3,0) {$y_i$};
\end{tikzpicture} \]
\item[-] units and counits $2$-cells $\varepsilon_i^+ : E_i \star_0 F_i \dfl 1$, $\eta_i^+: 1 \dfl E_i \star_0 F_i$, $\varepsilon_i^- : F_i \star_0 E_i \dfl 1$ and $\eta_i^- : 1 \dfl F_i \star_0 E_i$ satisfying the biadjunction relations, where the labels of regions are easily deduced and omitted:
\[  \begin{tikzpicture}[baseline = 0,scale=1.2]
  \draw[->,thick,black] (0.3,0) to (0.3,.4);
	\draw[-,thick,black] (0.3,0) to[out=-90, in=0] (0.1,-0.4);
	\draw[-,thick,black] (0.1,-0.4) to[out = 180, in = -90] (-0.1,0);
	\draw[-,thick,black] (-0.1,0) to[out=90, in=0] (-0.3,0.4);
	\draw[-,thick,black] (-0.3,0.4) to[out = 180, in =90] (-0.5,0);
  \draw[-,thick,black] (-0.5,0) to (-0.5,-.4);
  \node at (-0.2,0.6) {$\varepsilon_i^+$};
  \node at (0.1,-0.6) {$\eta_i^-$};
  \node at (-0.5,-0.6) {$\scriptstyle{i}$};
\end{tikzpicture}  = \; \begin{tikzpicture}[baseline=0, scale=1.2]
\draw[->,thick,black] (0,-0.4) to (0,0.4) ;
\node at (0,-0.6) {$\scriptstyle{i}$} ;
\end{tikzpicture}
= \; \begin{tikzpicture}[baseline = 0, scale=1.2]
  \draw[-,thick,black] (0.3,0) to (0.3,-.4);
	\draw[-,thick,black] (0.3,0) to[out=90, in=0] (0.1,0.4);
	\draw[-,thick,black] (0.1,0.4) to[out = 180, in = 90] (-0.1,0);
	\draw[-,thick,black] (-0.1,0) to[out=-90, in=0] (-0.3,-0.4);
	\draw[-,thick,black] (-0.3,-0.4) to[out = 180, in =-90] (-0.5,0);
  \draw[->,thick,black] (-0.5,0) to (-0.5,.4);
   \node at (-0.2,-0.6) {$\varepsilon_i^-$};
  \node at (0.1,0.6) {$\eta_i^+$};
  \node at (0.5,-0.6) {$\scriptstyle{i}$};
\end{tikzpicture}
\qquad \qquad 
\begin{tikzpicture}[baseline = 0,scale=1.2]
  \draw[-,thick,black] (0.3,0) to (0.3,.4);
	\draw[-,thick,black] (0.3,0) to[out=-90, in=0] (0.1,-0.4);
	\draw[-,thick,black] (0.1,-0.4) to[out = 180, in = -90] (-0.1,0);
	\draw[-,thick,black] (-0.1,0) to[out=90, in=0] (-0.3,0.4);
	\draw[-,thick,black] (-0.3,0.4) to[out = 180, in =90] (-0.5,0);
  \draw[->,thick,black] (-0.5,0) to (-0.5,-.4);
  \node at (-0.2,0.6) {$\varepsilon_i^-$};
  \node at (0.1,-0.6) {$\eta_i^+$};
  \node at (-0.5,-0.6) {$\scriptstyle{i}$};
\end{tikzpicture}  = \; \begin{tikzpicture}[baseline=0, scale=1.2]
\draw[<-,thick,black] (0,-0.4) to (0,0.4) ;
\node at (0,-0.6) {$\scriptstyle{i}$} ;
\end{tikzpicture}
= \; \begin{tikzpicture}[baseline = 0, scale=1.2]
  \draw[->,thick,black] (0.3,0) to (0.3,-.4);
	\draw[-,thick,black] (0.3,0) to[out=90, in=0] (0.1,0.4);
	\draw[-,thick,black] (0.1,0.4) to[out = 180, in = 90] (-0.1,0);
	\draw[-,thick,black] (-0.1,0) to[out=-90, in=0] (-0.3,-0.4);
	\draw[-,thick,black] (-0.3,-0.4) to[out = 180, in =-90] (-0.5,0);
  \draw[-,thick,black] (-0.5,0) to (-0.5,.4);
   \node at (-0.2,-0.6) {$\varepsilon_i^-$};
  \node at (0.1,0.6) {$\eta_i^+$};
  \node at (0.5,-0.6) {$\scriptstyle{i}$};
\end{tikzpicture}   \]

\item[-] cyclic $2$-morphisms $\alpha_i : E_i \dfl E_i$ and $\beta_i : F_i \dfl F_i$ with respect to the biadjunction $E_i \vdash F_i \vdash E_i$, respectively represented by a dot on an upward strand or on a downward strand labelled by $i$. By definition, cyclicity  yields the following relations:
\[  \begin{tikzpicture}[baseline = 0,scale=1.2]
  \draw[->,thick,black] (0.3,0) to (0.3,.4);
	\draw[-,thick,black] (0.3,0) to[out=-90, in=0] (0.1,-0.4);
	\draw[-,thick,black] (0.1,-0.4) to[out = 180, in = -90] (-0.1,0);
	\draw[-,thick,black] (-0.1,0) to[out=90, in=0] (-0.3,0.4);
	\draw[-,thick,black] (-0.3,0.4) to[out = 180, in =90] (-0.5,0);
  \draw[-,thick,black] (-0.5,0) to (-0.5,-.4);
\node at (-0.1,0) {$\bullet$};
\node at (0.1,0.3) {$\beta_i$};
  \node at (-0.5,-0.6) {$\scriptstyle{i}$};
\end{tikzpicture}  = \; \begin{tikzpicture}[baseline=0, scale=1.2]
\draw[->,thick,black] (0,-0.4) to (0,0.4) ;
\node at (0,-0.6) {$\scriptstyle{i}$} ;
\node at (0,0) {$\bullet$} ;
\node at (0.3,0.2) {$\alpha_i$};
\end{tikzpicture}
= \; \begin{tikzpicture}[baseline = 0, scale=1.2]
  \draw[-,thick,black] (0.3,0) to (0.3,-.4);
	\draw[-,thick,black] (0.3,0) to[out=90, in=0] (0.1,0.4);
	\draw[-,thick,black] (0.1,0.4) to[out = 180, in = 90] (-0.1,0);
	\draw[-,thick,black] (-0.1,0) to[out=-90, in=0] (-0.3,-0.4);
	\draw[-,thick,black] (-0.3,-0.4) to[out = 180, in =-90] (-0.5,0);
  \draw[->,thick,black] (-0.5,0) to (-0.5,.4);
\node at (-0.1,0) {$\bullet$};
\node at (-0.25,0.3) {$\beta_i$};  
  \node at (0.5,-0.6) {$\scriptstyle{i}$};
\end{tikzpicture}
\qquad \qquad 
\begin{tikzpicture}[baseline = 0,scale=1.2]
  \draw[-,thick,black] (0.3,0) to (0.3,.4);
	\draw[-,thick,black] (0.3,0) to[out=-90, in=0] (0.1,-0.4);
	\draw[-,thick,black] (0.1,-0.4) to[out = 180, in = -90] (-0.1,0);
	\draw[-,thick,black] (-0.1,0) to[out=90, in=0] (-0.3,0.4);
	\draw[-,thick,black] (-0.3,0.4) to[out = 180, in =90] (-0.5,0);
  \draw[->,thick,black] (-0.5,0) to (-0.5,-.4);
\node at (-0.1,0) {$\bullet$};
\node at (0.1,0.3) {$\alpha_i$};
  \node at (-0.5,-0.6) {$\scriptstyle{i}$};
\end{tikzpicture}  = \; \begin{tikzpicture}[baseline=0, scale=1.2]
\draw[<-,thick,black] (0,-0.4) to (0,0.4) ;
\node at (0,-0.6) {$\scriptstyle{i}$} ;
\node at (0,0) {$\bullet$} ;
\node at (0.3,0.2) {$\beta_i$};
\end{tikzpicture}
= \; \begin{tikzpicture}[baseline = 0, scale=1.2]
  \draw[->,thick,black] (0.3,0) to (0.3,-.4);
	\draw[-,thick,black] (0.3,0) to[out=90, in=0] (0.1,0.4);
	\draw[-,thick,black] (0.1,0.4) to[out = 180, in = 90] (-0.1,0);
	\draw[-,thick,black] (-0.1,0) to[out=-90, in=0] (-0.3,-0.4);
	\draw[-,thick,black] (-0.3,-0.4) to[out = 180, in =-90] (-0.5,0);
  \draw[-,thick,black] (-0.5,0) to (-0.5,.4);
\node at (-0.1,0) {$\bullet$};
\node at (-0.25,0.3) {$\alpha_i$};  
  \node at (0.5,-0.6) {$\scriptstyle{i}$};
\end{tikzpicture}   \]
\end{itemize}

We define the \emph{$3$-polygraph of isotopies} $E_I$ presenting the category $\mathcal{C}_I$ as follows:
\begin{itemize}
\item[-] the $0$-cells of $E_I$ are the $0$-cells of $\mathcal{C}_0$.
\item[-] the generating $1$-cells of $E_I$ are the $E_i$ and $F_i$ for $i \in I$, and the $1$-cells of $E_I$ are given by sequences $(E^\pm_i , E^\pm_j , E^\pm_k , \dots)$ with $E^+ = E$ and $E^- = F$.
\item[-] the generating $2$-cells of $E_I$ are given by cup and cap $2$-cells $\varepsilon_i^+, \eta_i^+, \varepsilon_i^-, \eta_i^-$, and cyclic $2$-cells $\alpha_i$ depicted by an upward strand decorated by a dot and labelled by $i$, and its bidual $\beta_i$ represented by a downward strand decorated by a dot and labelled by $i$.
\item[-] the $3$-cells of $E_I$ are given by:
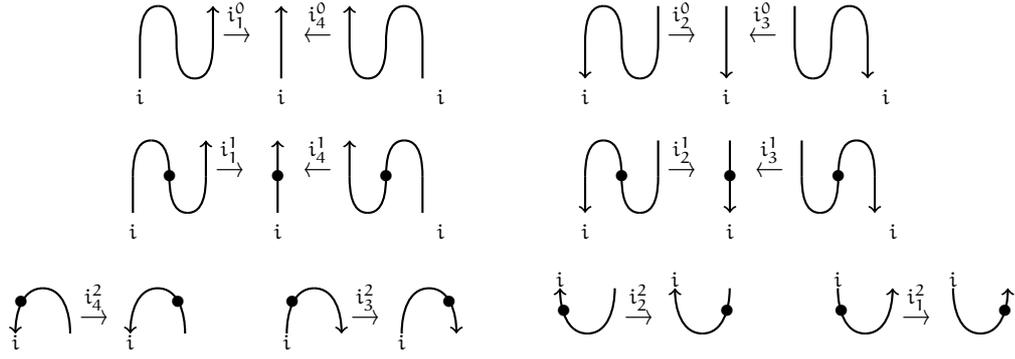
\[  \begin{tikzpicture}[baseline = 0,scale=1.2]
  \draw[->,thick,black] (0.3,0) to (0.3,.4);
	\draw[-,thick,black] (0.3,0) to[out=-90, in=0] (0.1,-0.4);
	\draw[-,thick,black] (0.1,-0.4) to[out = 180, in = -90] (-0.1,0);
	\draw[-,thick,black] (-0.1,0) to[out=90, in=0] (-0.3,0.4);
	\draw[-,thick,black] (-0.3,0.4) to[out = 180, in =90] (-0.5,0);
  \draw[-,thick,black] (-0.5,0) to (-0.5,-.4);
  \node at (-0.5,-0.6) {$\scriptstyle{i}$};
\end{tikzpicture}  \overset{i_1^0}{\fl} \; \begin{tikzpicture}[baseline=0, scale=1.2]
\draw[->,thick,black] (0,-0.4) to (0,0.4) ;
\node at (0,-0.6) {$\scriptstyle{i}$} ;
\end{tikzpicture}
\overset{i_4^0}{\leftarrow} \; \begin{tikzpicture}[baseline = 0, scale=1.2]
  \draw[-,thick,black] (0.3,0) to (0.3,-.4);
	\draw[-,thick,black] (0.3,0) to[out=90, in=0] (0.1,0.4);
	\draw[-,thick,black] (0.1,0.4) to[out = 180, in = 90] (-0.1,0);
	\draw[-,thick,black] (-0.1,0) to[out=-90, in=0] (-0.3,-0.4);
	\draw[-,thick,black] (-0.3,-0.4) to[out = 180, in =-90] (-0.5,0);
  \draw[->,thick,black] (-0.5,0) to (-0.5,.4);
  \node at (0.5,-0.6) {$\scriptstyle{i}$};
\end{tikzpicture}
\qquad \qquad 
\begin{tikzpicture}[baseline = 0,scale=1.2]
  \draw[-,thick,black] (0.3,0) to (0.3,.4);
	\draw[-,thick,black] (0.3,0) to[out=-90, in=0] (0.1,-0.4);
	\draw[-,thick,black] (0.1,-0.4) to[out = 180, in = -90] (-0.1,0);
	\draw[-,thick,black] (-0.1,0) to[out=90, in=0] (-0.3,0.4);
	\draw[-,thick,black] (-0.3,0.4) to[out = 180, in =90] (-0.5,0);
  \draw[->,thick,black] (-0.5,0) to (-0.5,-.4);
  \node at (-0.5,-0.6) {$\scriptstyle{i}$};
\end{tikzpicture}  \overset{i_2^0}{\fl} \; \begin{tikzpicture}[baseline=0, scale=1.2]
\draw[<-,thick,black] (0,-0.4) to (0,0.4) ;
\node at (0,-0.6) {$\scriptstyle{i}$} ;
\end{tikzpicture}
\overset{i_3^0}{\leftarrow} \; \begin{tikzpicture}[baseline = 0, scale=1.2]
  \draw[->,thick,black] (0.3,0) to (0.3,-.4);
	\draw[-,thick,black] (0.3,0) to[out=90, in=0] (0.1,0.4);
	\draw[-,thick,black] (0.1,0.4) to[out = 180, in = 90] (-0.1,0);
	\draw[-,thick,black] (-0.1,0) to[out=-90, in=0] (-0.3,-0.4);
	\draw[-,thick,black] (-0.3,-0.4) to[out = 180, in =-90] (-0.5,0);
  \draw[-,thick,black] (-0.5,0) to (-0.5,.4);
  \node at (0.5,-0.6) {$\scriptstyle{i}$};
\end{tikzpicture}   \]
\[  \begin{tikzpicture}[baseline = 0,scale=1.2]
  \draw[->,thick,black] (0.3,0) to (0.3,.4);
	\draw[-,thick,black] (0.3,0) to[out=-90, in=0] (0.1,-0.4);
	\draw[-,thick,black] (0.1,-0.4) to[out = 180, in = -90] (-0.1,0);
	\draw[-,thick,black] (-0.1,0) to[out=90, in=0] (-0.3,0.4);
	\draw[-,thick,black] (-0.3,0.4) to[out = 180, in =90] (-0.5,0);
  \draw[-,thick,black] (-0.5,0) to (-0.5,-.4);
\node at (-0.1,0) {$\bullet$};
  \node at (-0.5,-0.6) {$\scriptstyle{i}$};
\end{tikzpicture}  \overset{i_1^1}{\fl} \; \begin{tikzpicture}[baseline=0, scale=1.2]
\draw[->,thick,black] (0,-0.4) to (0,0.4) ;
\node at (0,-0.6) {$\scriptstyle{i}$} ;
\node at (0,0) {$\bullet$} ;
\end{tikzpicture}
\overset{i_4^1}{\leftarrow} \; \begin{tikzpicture}[baseline = 0, scale=1.2]
  \draw[-,thick,black] (0.3,0) to (0.3,-.4);
	\draw[-,thick,black] (0.3,0) to[out=90, in=0] (0.1,0.4);
	\draw[-,thick,black] (0.1,0.4) to[out = 180, in = 90] (-0.1,0);
	\draw[-,thick,black] (-0.1,0) to[out=-90, in=0] (-0.3,-0.4);
	\draw[-,thick,black] (-0.3,-0.4) to[out = 180, in =-90] (-0.5,0);
  \draw[->,thick,black] (-0.5,0) to (-0.5,.4);
\node at (-0.1,0) {$\bullet$};
  \node at (0.5,-0.6) {$\scriptstyle{i}$};
\end{tikzpicture}
\qquad \qquad 
\begin{tikzpicture}[baseline = 0,scale=1.2]
  \draw[-,thick,black] (0.3,0) to (0.3,.4);
	\draw[-,thick,black] (0.3,0) to[out=-90, in=0] (0.1,-0.4);
	\draw[-,thick,black] (0.1,-0.4) to[out = 180, in = -90] (-0.1,0);
	\draw[-,thick,black] (-0.1,0) to[out=90, in=0] (-0.3,0.4);
	\draw[-,thick,black] (-0.3,0.4) to[out = 180, in =90] (-0.5,0);
  \draw[->,thick,black] (-0.5,0) to (-0.5,-.4);
\node at (-0.1,0) {$\bullet$};
  \node at (-0.5,-0.6) {$\scriptstyle{i}$};
\end{tikzpicture}  \overset{i_2^1}{\fl} \; \begin{tikzpicture}[baseline=0, scale=1.2]
\draw[<-,thick,black] (0,-0.4) to (0,0.4) ;
\node at (0,-0.6) {$\scriptstyle{i}$} ;
\node at (0,0) {$\bullet$} ;
\end{tikzpicture}
\overset{i_3^1}{\leftarrow} \; \begin{tikzpicture}[baseline = 0, scale=1.2]
  \draw[->,thick,black] (0.3,0) to (0.3,-.4);
	\draw[-,thick,black] (0.3,0) to[out=90, in=0] (0.1,0.4);
	\draw[-,thick,black] (0.1,0.4) to[out = 180, in = 90] (-0.1,0);
	\draw[-,thick,black] (-0.1,0) to[out=-90, in=0] (-0.3,-0.4);
	\draw[-,thick,black] (-0.3,-0.4) to[out = 180, in =-90] (-0.5,0);
  \draw[-,thick,black] (-0.5,0) to (-0.5,.4);
\node at (-0.1,0) {$\bullet$};
  \node at (0.5,-0.6) {$\scriptstyle{i}$};
\end{tikzpicture}   \]
\[ \capldlsl{i}{} \overset{i_4^2}{\fl} \capldrsl{i}{} \qquad \caprdlsl{i}{} \overset{i_3^2}{\fl} \caprdrsl{i}{} \qquad \cupldlsl{i}{} \overset{i_2^2}{\fl} \cupldrsl{i}{} \qquad \cuprdlsl{i}{} \overset{i_1^2}{\fl} \cuprdrsl{i}{} \]
\end{itemize}
Note that the last family of relations (dot moves on caps and cups) are direct consequences of the first families of relations. However, without these $3$-cells the linear~$(3,2)$-polygraph would not be convergent. With these $3$-cells, the linear~$(3,2)$-polygraph $E_I$ is confluent, the proof being similar to the proof of confluence of the $3$-polygraph of pearls in \cite{GM09}.

\subsection{Termination of linear~$(3,2)$-polygraphs by derivation}
\label{SS:TerminationDerivation}
In this subsection, we recall from \cite{GUI06,GM09} a method to prove termination for a $3$-polygraph using the notion of derivation of a $2$-category. This method is extended to the linear setting.

\subsubsection{Modules of $2$-categories}
Let $\mathcal{C}$ be a linear~$(2,2)$-category.
A \emph{$\mathcal{C}$-module} is a functor from the category of contexts $\Ccontext$ to the category $\catego{Ab}$ of abelian groups. Hence, a $\mathcal{C}$-module is specified by an abelian group $M(f)$ for every $2$-cell $f$ in $\mathcal{C}$, and a morphism of groups $M(C): M(f) \fl M(g)$ for every context $C: f \fl g$ in $\mathcal{C}$.

\subsubsection{Example}
Recall from \cite{GM09} a prototypical example of module over a $2$-category, that we will use to prove termination of linear~$(3,2)$-polygraphs using derivations. Let $\catego{Ord}$ be the category of partially ordered sets and monotone maps, viewed as a $2$-category with one $0$-cell, ordered sets as $1$-cells and monotone maps as $2$-cells. We recall that an internal abelian group in $\textbf{Ord}$ is a partially ordered set equipped with a structure of abelian group whose addition is monotone in both arguments.
Let us fix such an internal abelian group $G$, a $2$-category $\mathcal{C}$ and $2$-functors $X: \mathcal{C} \fl \catego{Ord}$ and $Y: \mathcal{C}^{\text{op}} \fl \catego{Ord}$. We consider the $\mathcal{C}$-module $M_{X,Y,G}$ as follows:
\begin{enumerate}[{\bf i)}]
\item Every $2$-cell $f: u \dfl v$ in $\mathcal{C}$ is sent to the abelian group of morphisms $$ M_{X,Y,G}(f) = \textrm{Hom}_{\catego{Ord}} (X(u) \times Y(v) , G ) $$
\item If $w$ and $w'$ are $1$-cells of $\mathcal{C}$ and $C=w \star_0 x \star_0 w'$ is a context from $f: u \Rightarrow v$ to $w\star_0 f \star_0 w'$, then $M_{X,Y,G}(C)$ sends a morphism $a:X(u) \times Y(v) \to G$ in $\catego{Ord}$ to:
$$
\begin{array}{r c l}
X(w)\times X(u) \times X(w') \times Y(w) \times Y(v) \times Y(w')  &\:\longrightarrow\:& G \\
(x',x,x'',y',y,y'') &\:\longmapsto\:& a(x,y).
\end{array}
$$
\item If $g:u'\Rightarrow u$ and $h:v \Rightarrow v'$ are $2$-cells of $\mathcal{C}$ and $C=g\star_1 x \star_1 h$ is a context from $f: u\Rightarrow v$ to $g\star_1 f \star_1 h$, then $M_{X,Y,G}(C)$ sends a morphism $a:X(u \times Y(v)) \to G$ in $\textbf{Ord}$ to
\[
\begin{array}{r c l}
X(u') &\:\longrightarrow\:& G \\
x &\:\longmapsto\:& a \left(  X(g)(x) , Y(h)(y) \right).
\end{array}
\]
\end{enumerate}
By construction, when $\mathcal{C} = P_2^\ast$ is freely generated by a $2$-polygraph $P$, such a $\mathcal{C}$-module is uniquely and entirely determined by the values $X(u)$ and $Y(u)$ for every generating $1$-cell $u \in \Sigma_1$ and the morphisms $X(f) : X(u) \fl X(v) $ for every generating $2$-cell $f: u \dfl v$ in $P_2$.

\subsubsection{Derivations of linear~$(2,2)$-categories}
\label{SSS:TerminationByDerivation}
Let $\mathcal{C}$ be a linear~$(2,2)$-category, and let $M$ be a $\mathcal{U}(\mathcal{C})$-module. A \textit{derivation of $\C$ into $M$} is a map sending every $2$-cell $f$ of $\mathcal{C}$ to an element $d(f) \in M(f)$ such that the following relation holds, for every $i$-composable pair $(f,g)$ of $2$-cells of $\mathcal{C}$: 
$$
d(f\star_i g) \:=\: f\star_i d(g) + d(f) \star_i g.
$$
where $f \star_i d(g)$ (resp. $d(f) \star_i g)$ denotes the value $M(C)(d(g))$ (resp. $M(C')(d(f))$) where $C= f \star_i \square $ (resp. $C' = \square \star_i g$) for any $0 \leq i \leq 1$.
Following \cite[Theorem 4.2.1]{GM09}, we obtain that if 
$P$ is a linear~$(3,2)$-polygraph such that there exist:
\begin{enumerate}[{\bf i)}]
\item Two $2$-functors $X: P_2^\ast \to \catego{Ord}$ and $Y:(P_2^\ast)^{\text{co}} \to \catego{Ord}$ defined on monomials of $P_2^\ell$ such that, for every $1$-cell $a$ in $P_1$, the sets~$X(a)$ and $Y(a)$ are non-empty and, for every generating $3$-cell $\alpha$ in $P_3$, the inequalities $X(s(\alpha))\geq X(h)$ and $Y(s(\alpha))\geq Y(h)$ hold for any $h$ in $\text{Supp}(t(\alpha))$.
\item An abelian group $G$ in $\mathbf{Ord}$ whose addition is strictly monotone in both arguments and such that every decreasing sequence of non-negative elements of $G$ is stationary.
\item A derivation $d$ of $P_2^\ell$ into the module $M_{X,Y,G}$ such that, for every monomial $f$ in $P_2^\ell$, we have $d(f) \geq 0$ and, for every $3$-cell $\alpha$ in $P_3$, the strict inequality $d(s(\alpha))> d(h)$ holds for any $h$ in $\text{Supp}(t(\alpha))$.
\end{enumerate}
 Then the linear~$(3,2)$-polygraph $P$ terminates.

\begin{example}
	For instance, following the proof of termination for the $3$-polygraphs of pearls in \cite[Section 5.5.1]{GM09}, one proves that the linear~$(3,2)$-polygraph $E_I$ of isotopies defined in \ref{SSS:PolygraphIsotopy} is terminating, in two steps:
	\begin{enumerate}[{\bf i)}]
		\item At first, if we consider the derivation \[ d ( \cdot ) = || \cdot ||_{ \{ \varepsilon_i^-, \varepsilon_i^+, \eta_i^-, \eta_i^+ \} } \]  into the trivial module $M_{\ast, \ast, \Z}$ counting the number of oriented caps and cups of a diagram. This enables to reduce the termination of $E_I$ to the termination of the linear $(3,2)$-polygraph $E_I^{'}$ having for $3$-cells the $i_k^2$ for $1 \leq k \leq 4$.
		\item The polygraph $E_I^{'}$ terminates, using the $2$-functors $X$ and $Y$ and the derivation $d$ into the $(E_I)_2^\ast$-module $M_{X,Y,\Z}$ given by:
		\[ X \left( \ident{} \right) = \N, \qquad X \left( \capnonorientedsl{} \right) (i,j) = (0,0), \qquad X \left( \identdot{} \right) (i) = i+1 \] 
		\[ Y \left( \ident{} \right) = \N, \qquad Y \left( \cupnonorientedsl{} \right) (i,j) = (0,0), \qquad Y \left( \identdot{} \right) (i) = i+1 \] 
			\[ d \left( \capnonorientedsl{}  \right) (i,j) = i, \qquad d \left( \cupnonorientedsl{} \right) (i,j) = i, \qquad d \left( \identdot{} \right) (i,j) = 0 \]
for any orientation of the strands and any label on it. The required inequalities are proved in \cite{GM09}.
	\end{enumerate}
\end{example}

\section{Linear rewriting modulo}
\label{S:LinearRewritingModulo}
In this section, we introduce the notion of linear~$(3,2)$-polygraph modulo, for which we define the property of confluence modulo, and give several ways to prove confluence modulo from confluence of critical branchings in the case of termination or decreasingness in the case of quasi-termination. We give a method to compute a hom-basis for a linear~$(2,2)$-category $\mathcal{C}$ using rewriting modulo.

\subsection{Linear $(3,2)$-polygraphs modulo}
\label{SS:LinearPolMod}
The notion of $n$-dimensional polygraph modulo has been developed in \cite{DMpp18} in a non-linear context. In this subsection, we extend this construction by defining linear $(3,2)$-polygraphs modulo. We refer the reader to \cite{DMpp18} for more details about polygraphs modulo.

\begin{definition}
A \emph{linear $(3,2)$-polygraph modulo} is a data $(R,E,S)$ made of
\begin{enumerate}[{\bf i)}]
\item a linear $(3,2)$-polygraph $R = (R_0 , R_1 , R_2 ,R_3)$, whose generating $3$-cells are called \emph{primary rules};
\item a linear $(3,2)$-polygraph $E = (E_0,E_1 ,E_2,E_3)$ such that $E_k=R_k$ for $k= 0,1$ and $E_{2} \subseteq R_{2}$, whose generating $3$-cells are called \emph{modulo rules};
\item $S$ is a cellular extension of $R_{2}^\ell$ the linear free $2$-category generated by $R$, such that the following inclusions of cellular extensions $R \subseteq S \subseteq \ERE$ holds, where the cellular extension \[ \ERE \overset{\gamma^{\ERE}}{\fl} \text{Sph}(R_2^\ell) \] is defined in a similar way than in \cite[Subsubsection 3.1.2]{DMpp18}. The elements of $\ERE$ correspond to $2$-spheres $(u,v) \in R_{2}^\ell$ such that $(u,v)$ is the boundary of a $3$-cell $f$ in $R^\ell_{2}[R_3,E_3,E_3^-]/\text{Inv}(E_3,E_3^-)$, the free linear~$(2,2)$-category generated by $(R_0,R_1,R_2)$ augmented by the cellular extensions $R$, $E$ and the formal inverses $E^-$ of $E$ modulo the corresponding inverse relations, with shape $f = e_1 \star_{2} f_1 \star_{2} e_2$ with $e_1,e_2$ in $E^\ell$ and $f_1$ a rewriting step of $R$.
\end{enumerate}
\end{definition}

 Note that this data defines a linear~$(3,2)$-polygraph $(R_0,R_1,R_2,S)$ that we will denote by $S$ in the sequel.

\subsubsection{Examples}
In the sequel, we will consider the $3$-polygraphs modulo $\ERE$ and $\ER$, whose underlying $1$-polygraph is $(R_0,R_1)$, with $E_2 \subseteq R_2$ and whose set of $3$-cells are respectively defined as follows:
\begin{enumerate}[{\bf i)}]
\item $\ERE$ has a $3$-cell $u \fl v$ whenever there exists a rewriting step $g$ of $R$ and $3$-cells $e,e'$ in $E^\ell$ as in the following diagram:
\[
\xymatrix{
u
  \ar@{-->} [r]
  \ar [d] _-{e}
&
v
\\
u'
  \ar [r] _-{g}
&
v'
  \ar [u] _-{e'}
}
\]
\item $\ER$ has a $3$-cell $u \fl v$ whenever there exists a rewriting step $g$ of $R$ and $3$-cells $e$ in $E^\ell$ as in the following diagram:
\[
\xymatrix{
u
  \ar@{-->} [r]
  \ar[d] _-{e}
&
v
\\
u'
  \ar[r] _-{g}
&
v
  \ar@{->} [u] _-{\fleq}
}
\]
\end{enumerate}

\subsection{Branchings modulo and confluence modulo}
\label{SS:ConfluenceModulo}
In this subsection, we introduce the notion of confluence modulo for a linear $(3,2)$-polygraph modulo $(R,E,S)$ following \cite{DMpp18}. We give a classification of the local branchings modulo, and prove that if $\ERE$ is terminating, confluence of $S$ modulo $E$ is equivalent to the confluence modulo of critical branchings modulo.

\subsubsection{Branchings modulo}
A \emph{branching} of the linear $(3,2)$-polygraph $S$ is a pair~$(f,g)$ of positive $3$-cells of $S^\ell$ such that $s_2^S(f)=s_2^S(g)$. Such a branching is depicted by 
\begin{equation}
\label{E:branching}
\xymatrix @R=1.5em @C=2em {
u 
  \ar [r] ^-{f} 
  \ar [d] _-{\rotatebox{90}{=}}
& 
u'
\\
u
  \ar [r] _-{g} 
&
v'
}
\end{equation}
and will be denoted by $(f,g) : u \fl (u',v')$.
We do not distinguish the branchings~$(f,g)$ and~$(g,f)$.

A \emph{branching modulo $E$} of the linear $(3,2)$-polygraph modulo~$S$ is a triple~$(f,e,g)$ where $f$ is a positive $3$-cell of $S^\ell$, $g$ is aiehter a positive $3$-cell of $S^\ell$ or an identity $3$-cell, and $e$ is a $3$-cell of $E^\ell$. Such a branching is depicted by
\begin{equation}
\label{E:branchingModulo}
\raisebox{0.55cm}{
\xymatrix @R=1.5em @C=2em {
u 
  \ar[r] ^-{f} 
  \ar[d] _-{e}
& 
u'
\\
v
  \ar[r] _-{g} 
&
v'
}}
\qquad\qquad
\big(\text{resp.}
\quad
\raisebox{0.55cm}{
\xymatrix @R=1.5em @C=2em {
u 
  \ar[r] ^-{f} 
  \ar[d] _-{e}
& 
u'
\\
v
}}
\quad
\big)
\end{equation}
when $g$ is non trivial (resp. trivial) and denoted by $(f,e,g) : (u,v) \fl (u',v')$ (resp. $(f,e) : u \fl (u',v)$. Note that any branching $(f,g)$ is a branching $(f,e,g)$ modulo $E$ where $e$ is the identity $3$-cell on $s_2(f)=s_2(g)$.

\subsubsection{Confluence modulo}
A branching modulo $E$ as in (\ref{E:branchingModulo}) is \emph{confluent modulo $E$} if there exists positive $3$-cells~$f',g'$ in~$S^\ell$ and a $3$-cell $e'$ in $E^\ell$ as in the following diagram:
\[
\raisebox{0.55cm}{
\xymatrix @R=1.5em @C=2em {
u
  \ar[r] ^-{f}
  \ar[d] _-{e}
&
u' 
  \ar@{.>}[r] ^-{f'} 
& 
w
  \ar@{.>}[d] ^-{e'}
\\
v
  \ar [r] _-{g}
&
v'
  \ar@{.>}[r] _-{g'} 
&
w'
}}
\]
We then say that the triple $(f',e',g')$ is a confluence modulo $E$ of the branching $(f,e,g)$ modulo $E$. The linear~$(3,2)$-polygraph $S$ is \emph{confluent modulo $E$} if all its branchings modulo $E$ are confluent modulo~$E$.  When $S$ is confluent modulo $E$, a $2$-cell may admit several $S$-normal forms, which are all equivalent modulo~$E$. 

\subsubsection{Church-Rosser property modulo}
\label{SSS:ChurchRosserModulo}
We say that a linear~$(3,2)$-polygraph modulo $(R,E,S)$ is \emph{Church-Rosser modulo} the linear~$(3,2)$-polygraph $E$ if for any $2$-cells $u$,$v$ in $R_2^\ell$ such that there exist a zig-zag sequence 
\[ \xymatrix{ u \ar[r] ^-{f_1} & u_1 & u_2 \ar [l] _-{f_2} \ar[r] ^-{f_3} & \dots \ar [r] ^-{f_{n-2}} & u_{n-1} & u_n \ar [r] ^-{f_n} \ar [l] _-{f_{n-1}} & v } \]
where the $f_i$ are $3$-cells of $E^\ell$ or $R^\ell$, there exist positive $3$-cells $f': u \fl u'$ and $g': v \fl v$ in $S^\ell$ and a $3$-cell $e: u' \fl v'$ in $E^\ell$. In particular, when $S$ is normalizing this implies that for any $2$-cells $u$ and $v$ such that $\cl{u} = \cl{v}$ in $\C$ the category presented by $R \coprod E$, two normal forms $\hat{u}$ and $\hat{v}$ of $u$ and $v$ respectively with respect to $S$ are equivalent modulo $E$. 

\subsubsection{Jouannaud-Kirchner confluence modulo}
\label{SSS:JKConfluence}
In \cite{JouannaudKirchner84}, Jouannaud and Kirchner introduced another notion of confluence modulo $E$, given by two properties that they call confluence modulo $E$ and coherence modulo $E$. Following \cite{JouannaudKirchner84}, we say that 
\begin{enumerate}[{\bf i)}]
\item \emph{JK confluent modulo $E$}, if any branching is confluent modulo $E$,
\item \emph{JK coherent modulo $E$}, if any branching $(f,e) : u \fl (u',v)$ modulo $E$  is confluent modulo~$E$:
\[
\xymatrix @R=1.5em @C=2em {
u
  \ar[r] ^-{f} 
  \ar[d] _-{e}
& 
v \ar@1@{.>} [r] ^-{f'} & v' \ar@{.>}@1 [d] ^-{e'} \\
u'
  \ar@{.>} [rr] _-{g'} 
&  & w  }
\]
in such a way that $g'$ is a positive $3$-cell in $S^\ell$.
\end{enumerate}
However, we prove that this notion of confluence modulo is equivalent to that defined in subsection \ref{SS:ConfluenceModulo}.
\begin{lemma} \label{L:EquivalenceConfluence}
For any linear~$(3,2)$-polygraph  $(R,E,S)$ such that $S$ is terminating, the following assertions are equivalent: 
\begin{enumerate}[{\bf i)}]
\item $S$ is confluent modulo $E$.
\item $S$ is JK confluent modulo $E$ and JK coherent modulo $E$.
\end{enumerate}
\end{lemma}

\begin{proof}
By definition, the property of confluence modulo $E$ trivially implies both JK confluence modulo $E$ and JK coherence modulo $E$.
Conversely, suppose that the linear~$(3,2)$-polygraph $S$ is JK confluent and JK coherent modulo $E$ and let us consider a branching $(f,e,g)$ modulo $E$ in $S$. If $\ell(e) = 0$, then it is clearly confluent modulo $E$ by JK confluence modulo $E$ so let us assume that $\ell(f) \geq 1$. If $g$ is an identity $3$-cell, then the confluence of the branching $(f,e)$ modulo $E$ is given by JK coherence modulo $E$. Otherwise, by JK coherence modulo $E$ on the branching $(f,e)$, there rewriting sequences $f'$ and $h$ in $S^\ell$ with $h$ non trivial and a $3$-cell $e': t_{2}(f') \fl t_{2}(h)$ in $E^\ell$. Applying JK confluence modulo on the branching $(h,g)$ of $S$, there exists positive $3$-cells $g'$ and $h'$ in $S^\ell$ and a $3$-cell $e'': t_{2}(h') \fl t_{2}(g')$ in $E^\ell$. By JK coherence modulo $E$ on the branching $(e'^-,h')$ modulo $E$, we get the existence of positive $3$-cells $f''$ and $h''$ in $S^\ell$ and a $3$-cell $e''': t_{2}(f'') \fl t_{2}(h'')$ in $E^\ell$. This situation can be depicted by:
\[
\xymatrix @R=2.5em @C=2.5em {
u
  \ar[r] ^-{f} 
  \ar[d] _-{e}
& 
u' \ar@{.>} [r] ^-{f'} ^-{}="1"  & u'' \ar@{.>} [rr] ^-{f''} ^-{}="3" \ar@{.>} [d] ^-{e'} & & u''' \ar@{.>}[d] _-{e'''}  \\
v \ar@{.>} [rr] |-{h} _-{}="2" \ar[d] _-{\rotatebox{90}{=}} & &  w \ar@{.>} [r] |-{h'} _-{}="2"  & w' \ar@{.>} [d] _-{e''} \ar@{.>} [r] |-{h''} & w'' \\
v \ar@1@{.>} [r] |-{g}  & v' \ar@1@{.>} [rr] |-{g'} & & v'' & 
\ar@2{} "1,2";"2"!<-55pt,0pt> |{\text{JK coh.}}
\ar@2{} "3,2";"2"!<-25pt,0pt> |{\text{JK confl.}}
\ar@2{} "3";"2,4"!<0pt,0pt> |{\text{JK coh.}}}
\]
At this point, either $h''$ is trivial and thus $e''': u''' \fl w'$ so that the branching $(f,e,g)$ is confluent modulo, or it is non-trivial and we can apply JK coherence on the branching $(h'',e'')$. Since $S$ is terminating, this process can not apply infinitely many times, and thus in finitely many steps we prove the confluence modulo of the branching $(f,e,g)$.
\end{proof}

Now, following \cite[Theorem 5]{JouannaudKirchner84} and Lemma \ref{L:EquivalenceConfluence}, given a a linear~$(3,2)$-polygraph modulo $(R,E,S)$ such that $S$ is terminating, the following properties are equivalent:
\begin{enumerate}[{\bf i)}]
\begin{multicols}{2}
\item $S$ is confluent modulo $E$.
\item $S$ is Church-Rosser modulo $E$.
\end{multicols}
\end{enumerate}

\subsubsection{Local branchings modulo}
A branching $(f,g)$ of the linear~$(3,2)$-polygraph $S$ is \emph{local} if $f,g$ are $3$-cells of~$S^\ast$ of length $1$. 
A branching $(f,e,g)$ modulo $E$ is \emph{local} if $f$ is a $3$-cell of $\So$, $g$ is either a positive $3$-cell of $S^\ell$ or an identity and $e$ a $3$-cell of$E^\ell$ such that $\ell(g) + \ell(e) = 1$.
Local branchings belong to one of the following families:
\begin{enumerate}[{\bf i)}]
\item \emph{local aspherical} branchings of the form:
\[
\xymatrix @R=1.5em @C=2em{
u 
  \ar[r] ^-{f} 
  \ar[d] _-{\rotatebox{90}{=}}
& 
v
  \ar[d] ^-{\rotatebox{90}{=}}
\\
u
  \ar[r] _-{f} 
&
v
}
\]
where $f$ is a $3$-cell of $\So$.
\item \emph{local Peiffer} branchings of the form:
\[
\xymatrix @R=1.5em @C=2em{
u\star_i v + w
  \ar[r] ^-{f\star_i v} 
  \ar[d] _-{\rotatebox{90}{=}}
& 
u'\star_i v +w
\\
u\star_i v + w
  \ar[r] _-{u\star_i g} 
&
u\star_i v' + w
}
\]
where $0\leq i \leq n-2$, $w$ is a $2$-cell of $R_2^\ell$ and $f$ and $g$ are positive $3$-cells in $\So$.
\item \emph{local additive} branchings of the form:
\[
\xymatrix @R=1.5em @C=2em{
u + v 
  \ar[r] ^-{f + v} 
  \ar[d] _-{\rotatebox{90}{=}}
& 
u' + v
\\
u + v
  \ar[r] _-{u + g} 
&
u + v'
}
\]
where $f$ and $g$ are positive $3$-cells in $\So$.
\item \emph{local Peiffer branchings modulo} of the form:
\begin{equation*}
\label{E:LocalPeifferModulo}
\xymatrix @R=1.5em @C=2em {
u\star_i v + w
  \ar[r] ^-{f\star_i v} 
  \ar[d] _-{u\star_i e}
&  
u'\star_i v + w
\\
u\star_i v' + w  &
}
\qquad\qquad
\xymatrix @R=1.5em @C=2em {
v \star_i u + w 
  \ar[r] ^-{v\star_i f} 
  \ar[d] _-{e'\star_i u}
& 
v \star_i u' + w
\\
v '\star_i u + w &
}
\end{equation*}
where $0\leq i \leq n-2$, $w$ is a $2$-cell of $R_2^\ell$, $f$ is a $3$-cell in $\So$ and $e,e'$ are $3$-cells of $E^\ell$ of length $1$.
\item \emph{local additive branchings modulo} of the form:
\[
\xymatrix @R=1.5em @C=2em {
u + v 
  \ar[r] ^-{f + v} 
  \ar[d] _-{u + e}
& 
u' + v
\\
u + v' &
}
\qquad\qquad
\xymatrix @R=1.5em @C=2em {
w + u 
  \ar[r] ^-{w + f} 
  \ar[d] _-{e' + u}
& 
w + u'
\\
w' + u &
}
\]
where $f$ is a positive $3$-cell in $\So$ and $e,e'$ are $3$-cell in $\Eo$ of length $1$.
\item \emph{overlapping branchings} are the remaining local branchings:
\[
\xymatrix @R=1.5em @C=2em {
u 
  \ar[r] ^-{f} 
  \ar[d] _-{\rotatebox{90}{=}}
& 
v
\\
u
  \ar[r] _-{g} 
&
v'
}
\]
where $f$ and $g$ are positive $3$-cells in $\So$.
\item \emph{overlapping branchings modulo} are the remaining local branchings modulo:
\[
\xymatrix @R=1.5em @C=2em {
u 
  \ar[r] ^-{f} 
  \ar[d] _-{e}
& 
v
\\
v' & }
\]
where $f$ is a positive $3$-cell in $\So$ and $e$ is a $3$-cell in $E^\ell$ of length $1$.
\end{enumerate}

We say that $S$ is \emph{locally confluent modulo~$E$} if any of its local branchings modulo $E$ is confluent modulo~$E$. From \cite[Theorem 4.1.4, Proposition 4.2.1]{DMpp18}, we have the following result which is an abstract rewriting result and does not depend on the linear context:

\begin{theorem} \label{T:ConfluenceTheorem}
Let $(R,E,S)$ be a linear~$(3,2)$-polygraph modulo such that $\ERE$ is terminating. The following assertions are equivalent:
\begin{enumerate}[{\bf i)}]
\item $S$ is confluent modulo $E$,
\item $S$ local confluent modulo $E$,
\item $S$ satisfies the following properties \textbf{a)} and \textbf{b)}:
\begin{enumerate}[{\bf a)}]
\item any local branching $(f,g) : u \fl (v,w)$ with $f$ in $\So$ and $g$ in $\Ro$ is confluent modulo $E$:
\[
\xymatrix @R=2em @C=2em {
u
  \ar[r] ^-{f} 
  \ar[d] _-{\rotatebox{90}{=}}
& 
v \ar@1@{.>} [r] ^-{f'} & v' \ar@1@{.>} [d] ^-{e'} \\
u
  \ar[r] _-{g} 
& w \ar@1@{.>} [r] & w' 
}
\]			
\item any local branching $(f,e):u \fl (v,u')$ modulo $E$ with $f$ in $\So$ and $e$ in $E^\ell$ of length $1$ is confluent modulo $E$:
\[
\xymatrix @R=2em @C=2em {
u
  \ar[r] ^-{f} 
  \ar[d] _-{e}
& 
v \ar@1@{.>} [r] ^-{f'} & v' \ar@1@{.>} [d] ^-{e'} \\
u'
  \ar@{.>} [rr] _-{g'} 
&  & w }
\]
\end{enumerate}
\end{enumerate}
\end{theorem}

\subsubsection{Critical branchings}
\label{SSS:CriticalBranchings}
Let $\sqsubseteq$ be the order on monomials of the linear~$(3,2)$-polygraph $S$ defined by $f \sqsubseteq g$ if there exists a context $C$ of $R_2^\ast$ such that $g = C[f]$,
a \emph{critical branching modulo $E$} is an overlapping local branching modulo $E$ that is minimal for the order $\sqsubseteq$.

\begin{theorem}[Linear critical branching lemma modulo]
\label{T:CriticalBranchingLemmaModulo}
Let $(R,E,S)$ be a linear $(3,2)$-polygraph modulo such that $\ERE$ is terminating. Then $S$ is locally confluent modulo $E$ if and only if the two following conditions hold
\begin{itemize}
\item[$\mathbf{a_0)}$] any critical branching $(f,g)$ with $f$ positive $3$-cell in
$\So$ and $g$ positive $3$-cell in $\Ro$ is confluent modulo $E$:
\[
\xymatrix @R=2em @C=2em {
u
   \ar[r] ^-{f}
   \ar[d] _-{\rotatebox{90}{=}}
&
v \ar@1@{.>} [r] ^-{f'} & v' \ar@1@{.>} [d] ^-{e'} \\
u
   \ar[r] _-{g}
& w \ar@1@{.>} [r] & w'
}
\]
\item[$\mathbf{b_0)}$] any critical branching $(f,e)$ modulo $E$ with $f$ in
$\So$ and $e$ in $E^\ell$ of length $1$ is confluent modulo $E$:
\[ \xymatrix @R=2em @C=2em {
u
   \ar[r] ^-{f}
   \ar[d] _-{e}
&
v \ar@1@{.>} [r] ^-{f'} & v' \ar@1@{.>} [d] ^-{e'} \\
u'
   \ar@{.>} [rr] _-{g'} _-{}="1"
&  & w }
\]
\end{itemize}
\end{theorem}

\begin{proof}
By Theorem~\ref{T:ConfluenceTheorem}, the local confluence of $S$ modulo $E$ is equivalent to both conditions {\bf a)} and {\bf b)}. Let us prove that the condition {\bf a)} (resp. {\bf b)}) holds if and only if the condition $\mathbf{a_0)}$ (resp. $\mathbf{b_0)}$) holds. One implication is trivial, let us prove the converse implication. To do so, let us proceed by Huet's double noetherian induction as introduced in \cite{Huet80} on the polygraph modulo $S^{\amalg}$ defined in \cite{DMpp18} which is terminating since $\ERE$ is assumed terminating. We refer to \cite{DMpp18} for further details on this double induction.  

Following the proof of the linear critical pair lemma in \cite{GHM17}, we assume that condition $\mathbf{a_0)}$ holds and prove condition {\bf a)}.
Let us consider a local branching $(f,g)$ of $S$ modulo $E$ of source $(u,v)$ with $f$ and $g$ positive $3$-cells in $\So$ and $\Ro$ respectively. Let us assume that any local branching of source $(u',v')$ such that there is a $3$-cell $(u,v) \fl (u',v')$ in $\Saux$ is confluent modulo $E$. The local branching $(f,g)$ is either a local Peiffer branching, an additive branching or an ovelapping branching. We prove that for each case, $(f,g)$ is confluent modulo $E$.
\begin{enumerate}[{\bf i)}]
\item If $(f,g)$ is a Peiffer branching of the form \[
\xymatrix @R=1.5em @C=2em{
u\star_i v + w
  \ar[r] ^-{f\star_i v} 
  \ar[d] _-{\rotatebox{90}{=}}
& 
u'\star_i v +w
\\
u\star_i v + w
  \ar[r] _-{u\star_i g} 
&
u\star_i v' + w
}
\]
where $0\leq i \leq n-2$, $w$ is a $2$-cell of $R_2^\ell$, $f$ is a positive $3$-cell in $\So$ and $g$ is a positive $3$-cell in $\Ro$, there exist elementary $3$-cells in $S^\ell$ as follows:
\[
\xymatrix @R=1.5em @C=2em{
u\star_i v + w
  \ar[r] ^-{f\star_i v} 
  \ar[d] _-{\rotatebox{90}{=}}
& 
u'\star_i v +w \ar@{.>} [r] ^-{u' \star_i g + w} & u' \star_i v' + w \ar [d] ^-{\rotatebox{90}{=}}  
\\
u\star_i v + w
  \ar[r] _-{u\star_i g} 
&
u\star_i v' + w \ar@{.>} [r] _-{f \star_i v' + w} & u' \star_i v' + w 
}
\]
However, these $3$-cells are not necessarily positive, for instance if $u'v \in \text{Supp}(w)$ or \linebreak $uv' \in \text{Supp}(w)$. By Lemma \ref{L:CellDecomposition}, there exist positive $3$-cells $f_1,f_2,g_1,g_2$ in $S^\ell$ of length at most $1$ such that $f \star_i v' + w = f_1 \star_2 f_2^-$ and $u' \star_i g + w = g_1 \star_2 g_2^-$. Then, the $3$-cells $f_2$ and $g_2$ of $S^\ell$ have the same $2$-source and by assumption, the branching $(f_2,g_2)$ is confluent modulo $E$, so there exist positive $3$-cells $f'$ and $g'$ in $S^\ell$ and a $3$-cell $e$ in $E^\ell$ as follows:
\[
\xymatrix @R=2em @C=3em{
u\star_i v + w
  \ar[r] ^-{f\star_i v + w} 
  \ar[d] _-{\rotatebox{90}{=}} & u'\star_i v +w \ar[d] ^-{\rotatebox{90}{=}} \ar@1 [rr] ^-{f_1} & & \ar [d] ^-{\rotatebox{90}{=}} \ar [r]^-{f'} & \ar [ddd] ^-{e'} \\
u\star_i v + w
  \ar[r] ^-{f\star_i v + w} 
  \ar[d] _-{\rotatebox{90}{=}}
& 
u'\star_i v +w \ar@{.>} [r] ^-{u' \star_i g + w} & u' \star_i v' + w \ar [d] ^-{\rotatebox{90}{=}} \ar [r] ^-{f_2} &   
\\
u\star_i v + w 
  \ar[r] _-{u\star_i g + w} 
  \ar[d] _-{\rotatebox{90}{=}}
&
u\star_i v' + w \ar@{.>} [r] _-{f \star_i v' + w} \ar [d] _-{\rotatebox{90}{=}} & u' \star_i v' + w  \ar [r] _-{g_2} & \ar [d] ^-{\rotatebox{90}{=}} & \\
u\star_i v + w 
  \ar[r] _-{u\star_i g + w} &  u\star_i v' + w \ar [rr] _-{g_1} & & \ar [r] _-{g'} & 
}
\]
which proves the confluence modulo of the branching $(f,g)$.
\item If $(f,g)$ is an additive branching of the form 
\[
\xymatrix @R=1.5em @C=2em{
u + v 
  \ar[r] ^-{f + v} 
  \ar[d] _-{\rotatebox{90}{=}}
& 
u' + v
\\
u + v
  \ar[r] _-{u + g} 
&
u + v'
}
\]
where $f$ is positive $3$-cells of $\So$ and $g$ is a positive $3$-cell of $\Ro$, there exist elementary $3$-cells in $S^\ell$ as follows:
\[
\xymatrix @R=1.5em @C=2em{
u + v 
  \ar[r] ^-{f + v} 
  \ar[d] _-{\rotatebox{90}{=}}
& 
u' + v \ar@{.>} [r] ^-{u' + g} & u' + v' \ar [d] ^-{\rotatebox{90}{=}}
\\
u + v
  \ar[r] _-{u + g} 
&
u + v' \ar@{.>} [r] _-{f + v'} & u' + v'
}
\]
However, these $3$-cells are not necessarily positive, for instance if $u \in \text{Supp}(v)$ or $u \in \text{Supp}(v')$. By Lemma \ref{L:CellDecomposition}, there exist positive $3$-cells $f_1,f_2,g_1,g_2$ in $S^\ell$ of length at most $1$ such that $f \star_i v' + w = f_1 \star_2 f_2^-$ and $u' \star_i g + w = g_1 \star_2 g_2^-$. We then prove the confluence modulo of $(f,g)$ in a same fashion as for case \textbf{i)}.
\item If $(f,g)$ is an overlapping branching of $S$ with $f$ in $\So$ and $g$ in $\Ro$ that is not critical, then by definition there exists a context $C= m_1 \star_1 (m_2 \star_0 \square \star_0 m_3 ) \star_1 m_4$ of $R_2^\ast$  and positive $3$-cells $f'$ and $g'$ in $S^\ell$ and $R^\ell$ respectively such that $f = C[f']$ and $g = C[g']$, and the branching $(f',g')$ is critical. By property ${\bf a_0)}$, the branching $(f',g')$ is confluent modulo $E$, so that there exist positive $3$-cells $f_1$ and $g_1$ in $S^\ell$ and a $3$-cell $e$ in $E^\ell$ as follows:
\[
\raisebox{0.55cm}{
\xymatrix @R=1.5em @C=2em {
u
  \ar [r] ^-{f'}
  \ar [d] _-{\rotatebox{90}{=}}
&
u' 
  \ar@{.>}[r] ^-{f_1} 
& 
w
  \ar@{.>}[d] ^-{e}
\\
u
  \ar[r] _-{g'}
&
v'
  \ar@{.>}[r] _-{g_1} 
&
w'
}}
\]
inducing a confluence modulo of the branching $(f,g)$:
\[
\raisebox{0.55cm}{
\xymatrix @R=1.5em @C=2em {
C[u]
  \ar [r] ^-{f}
  \ar [d] _-{\rotatebox{90}{=}}
&
C[u'] 
  \ar@{.>}[r] ^-{C[f_1]} 
& 
C[w]
  \ar@{.>}[d] ^-{C[e]}
\\
C[u]
  \ar[r] _-{g}
&
C[v']
  \ar@{.>}[r] _-{C[g_1]} 
&
C[w']
}}
\]
\end{enumerate}

Now, suppose that conditions $\mathbf{b_0)}$ holds and prove condition {\bf b)}. Let us consider a local branching $(f,e)$ of $S$ modulo $E$ of source $(u,v)$, with $f$ in $\So$ and $e$ in $E^\ell$ of length $1$. We still assume that any local branching of source $(u',v')$ such that there is a $3$-cell $(u,v) \fl (u',v')$ in $\Saux$ is confluent modulo $E$. The branching $(f,e)$ is either a local Peiffer branching modulo $E$, an additive branching modulo $E$ or an ovelapping modulo $E$. Let us prove that it is confluent modulo $E$ for each case.
\begin{enumerate}[{\bf i')}]
\item If $(f,e)$ is a local Peiffer branching modulo of the form
\[
\xymatrix @R=1.5em @C=2em {
u\star_i v + w
  \ar[r] ^-{f\star_i v} 
  \ar[d] _-{u\star_i e}
& 
u'\star_i v + w
\\
u\star_i v' + w &
}\] with $w$ in $R_2^\ell$, $f$ a positive $3$-cell in $\So$ and $e$ a $3$-cell in $E^\ell$ (the other form of such branching being treated similarly), there exist $3$-cells $f \star_i v'$ and $u' \star_i e$ in $S^\ell$ and $E^\ell$ respectively as in the following diagram
\[
\xymatrix @R=1.5em @C=2em {
u\star_i v + w 
  \ar[r] ^-{f\star_i v} 
  \ar[d] _-{u\star_i e}
& 
u'\star_i v + w \ar@{.>} [d] ^-{u' \star_i e}
\\
u\star_i v' + w \ar@{.>} [r] _-{f \star_i v'} & u' \star_i v' + w 
} \]
However, the dotted horizontal $3$-cell is not necessarily positive, for instance if $uv' \in \text{Supp}(w)$. By Lemma \ref{L:CellDecomposition}, there exist positive $3$-cells $f_1,f_2$ in $S^\ast$ of length at most $1$ such that $f \star_i v' = f_1 \star_2 f_2^-$. Then, we have $t_2^E(u' \star_i e) = s_2^S(f_2)$ and by assumption the branching $(f_2 , (u' \star_i e)^-)$ is confluent modulo $E$, so there exists positive $3$-cells $g$ and $h$ in $S^\ell$ and a $3$-cell $e'$ in $E^\ell$ as follows:
\[
\xymatrix @R=2em @C=2em {
u\star_i v 
  \ar[r] ^-{f\star_i v} 
  \ar[d] _-{u\star_i e}
& 
u'\star_i v \ar@{.>} [d] ^-{u' \star_i e} \ar@{.>} [rr] ^-{g} & & w \ar@{.>} [dd] ^-{e'} 
\\
u \star_i v' \ar@{.>} [r] _-{f \star_i v'} \ar[d] _-{\rotatebox{90}{=}} & u' \star_i v' \ar@{.>} [r] ^-{f_2} & u'' \ar@{.>} [d]^-{\rotatebox{90}{=}} & \\
 u \star_i v' \ar [rr] _-{f_1} & & u'' \ar@{.>} [r] _-{h} & w'  
} \]
which proves the confluence modulo of $(f,g)$.
\item If $(f,e)$ is a local additive branching modulo $E$ of the form 
\[
\xymatrix @R=1.5em @C=2em {
u + v 
  \ar[r] ^-{f + v} 
  \ar[d] _-{u + e}
& 
u' + v
\\
u + v' &
} \]
where $f$ is a positive $3$-cell in $\So$ and $e$ is a $3$-cell in $E^\ell$ of length $1$ (the other form of such branching being treated similarly), there exist $3$-cells $f+v'$ and $u'+e$ in $S^\ell$ and in $E^\ell$ respectively as in the following diagram
\[
\xymatrix @R=1.5em @C=2em {
u + v 
  \ar[r] ^-{f + v} 
  \ar[d] _-{u + e}
& 
u' + v \ar@{.>} [d] ^-{u'+e} 
\\
u + v' \ar@{.>} [r] _-{f+ v'} & u' + v' 
} \]
However, the $3$-cell $f + v'$ in $\So$ is not necessarily positive, for instance if $u \in \text{Supp}(v')$ but by Lemma \ref{L:CellDecomposition}, there exist positive $3$-cells $f_1$ and $f_2$ in $S^\ast$ of length at most $1$ such that $f + v' = f_1 \star_2 f_2^-$. We then prove the confluence modulo of the branching $(f,e)$ by a similar argument than above.
\item If $(f,e)$ is an overlapping modulo, the proof is similar to the proof for property $\mathbf{a_0)}$. 
\end{enumerate}
\end{proof}

\subsection{Confluence modulo by decreasingness}
\label{SS:ConfluenceByDecreasingness}
Van Oostrom introduced the notion of decreasingness in \cite{VOO94} in order to weaken the termination assumption needed to prove confluence from local confluence. In this section, we extend this construction to the context of rewriting modulo.

\subsubsection{Labelled linear~$(3,2)$-polygraphs}
A \emph{well-founded labelled linear~$(3,2)$-polygraph} is a data $(P,X,<, \psi)$ made of:
\begin{enumerate}[{\bf i)}]
	\item a linear~$(3,2)$-polygraph $P$;
	\item a set $X$;
	\item a well-founded order $<$ on $X$;
	\item a map $\psi$ which associates to each rewriting step $f$ of $P$ an element $\psi(f)$ of $X$ called the label of $f$.
\end{enumerate}

The map $\psi$ is called a \emph{well-founded labelling} of $P$. Given a rewriting sequence $f = f_1 \star_1 \ldots \star_1 f_k$, we denote by $L^X(f)$ the set $\{ \psi (f_1) ,\ldots , \psi (f_k) \}$.

\subsubsection{Multisets and multiset ordering}
\label{SSS:MultisetOrdering}
Recall that a \emph{multiset} is a collection in which elements are allowed to occur more than once or even infinitely often, contrary to an usual set. It is called \emph{finite} when every element appears a finite number of times. These multisets are equipped with three operations: union $\cup$, intersection $\cap$ and difference $-$. 

Given a well-founded set of labels $(X, <)$, we denote by $\multiinf{x}$ the multiset $\{ y \in X \; | \; y < x \}$ for any $x$ in $X$, and by $\multiinf{M}$ the multiset 
\[ \bigcup\limits_{x \in M} \multiinf{x} \] for any multiset $M$ over $X$.  The order $<$ extend to a partial order $\omult$ on the multisets over $X$ defined by $M \omult N$ if there exists multisets $M_1$, $M_2$ and $M_3$ such that
\begin{enumerate}[{\bf i)}]
	\item $M = M_1 \cup M_2$, $N = M_1 \cup M_3$ and $M_3$ is not empty,
	\item $M_2 \subseteq \multiinf{M_3}$, that is for every $x_2$ in $M_2$, there exists $x_3$ in $M_3$ such that $x_2 < x_3$.
\end{enumerate}

Following \cite{DM79}, if $<$ is well-founded, then so is $\omult$. Let us recall the following lemma from \cite[Lemma A.3.10]{VOO94} establishing the properties of the operations on multisets, that we will use to establish the proof of confluence by decreasingness:
\begin{lemma}
	\label{L:OperationsMultisets}
	For any multisets $M$, $N$ and $S$, the following properties hold:
	\begin{enumerate}[{\bf i)}]
		\item $\cup$ is commutative, associative and admits $\emptyset$ as unit element,
		\item $\cup$ is distributive over $\cap$,
		\begin{multicols}{2}
			\item $S \cap ( M \cup N) = (S \cap M ) \cup ( S \cap N ) $,
			\item $M \cap (N - S) = (M \cap N) - (M \cap S)$
			\item $(M \cap N) - S = (M  -S) \cap (N-S)$,
			\item $(S \cup M) - N = (S - N) \cup (M - N)$,
			\item $(M \cup N) - S = (M - S) \cup (N-S)$,
			\item $(M- N) - S = M - (N \cup S)$,
			\item $M = (M \cap N) \cup (M-N)$,
			\item $(M - N) \cap S = (M \cap S) - N$.
		\end{multicols}
	\end{enumerate}
\end{lemma}

\subsubsection{Lexicographic maximum measure}
\label{SSS:LexMaxMeasure}
Let $(P,X,<,\psi)$ be a well-founded labelled linear~$(3,2)$-polygraph. Let $x = x_1 \dots x_n$ and $x' = x'_1 \dots x'_m$ be two $1$-cells in the free monoid $X^\ast$. We denote by $x^{(x')}$ the $1$-cell $\cl{x_1} \dots \cl{x_n}$ where each $\cl{w_i}$ is defined as 
\begin{itemize}
	\item[-] $1$ if $x_k < x'_j$ for some $1 \leq  \leq m$;
	\item[-] $x_k$ otherwise.
\end{itemize}

Following \cite{VOO94}, we consider the measure $| \cdot |$ from $X^\ast$ to the set of multisets over $X$ and defined as follows:
\begin{enumerate}[{\bf i)}]
	\item for any $x$ in $X$, the multiset $|x|$ is the singleton $\{ x \}$.
	\item for any $i$ in $X$ and any $1$-cell $x$ of $X^\ast$, $| i x| = |i | \cup | x^{(i)} |$.
\end{enumerate}
This measure is extended to the set of finite rewriting sequences of $P$ by setting for every rewriting sequence $f_1 \star_1 \ldots \star_1 f_n$:
\[ | f_1 \star_1 \ldots \star_1 f_n | = | k_1 \ldots k_n | \] where each $f_i$ is labelled by $k_i$ and $k_1 \ldots k_n$ is a product in the monoid $X^\ast$. Finally, the measure $| \cdot |$ is extended to the set of finite branchings $(f,g)$ of $P$ be setting$| (f,g) | = |f] \cup |g| $.

Recall from \cite[Lemma 3.2]{VOO94} that for any $1$-cells $x_1$ and $x_2$ in $X^\ast$, we have 
\[ |x_1 x_2 | = |x_1 | \cup |x_2^{(x_1)} | \]
and as a consequence, for any rewriting sequences $f$ and $g$ of $P$, the following 
relations hold:
\[ | f \star_2 g || =|f| \cup | k_1 \ldots k_m ^{(l_1 \ldots l_n)} | \] 
where $f= f_1 \star_1 \ldots \star_1 f_n$ (resp. $g= g_1 \star_1 \ldots \star_1 g_m$) and each $f_i$ (resp. $g_j$) is labelled by
$l_i$ (resp. $k_j$).

\subsubsection{Well-founded labelling modulo}
Given a linear~$(3,2)$-polygraph modulo $(R,E,S)$, a \emph{well-founded labelling modulo of $S$} is a well-founded labelling $\psi$ of $R$ extended to $\ERE$ by setting $\psi (e) = 1$ the trivial word in $X^\ast$ for any $e$ in $E$.
The lexicographic maximum measure defined in \ref{SSS:LexMaxMeasure} then extends to the rewriting steps of $S$ as follows:
\[ | e_1 \star_1 f \star_1 e_2 | = | f | \]
for any $3$-cells $e_1$ and $e_2$ in $E^\ell$ and rewriting step $f$ of $R$. It then extends to the rewriting sequences of $S$ and $\ERE$, and to the finite branchings $(f,e,g)$ of $S$ modulo $E$.

\subsubsection{Decreasingness modulo}
Following \cite[Definition 3.3]{VOO94}, we introduce a notion of decreasingness for a diagram of confluence modulo. Let $(R,E,S)$ be a linear~$(3,2)$-polygraph modulo equipped with a well-founded labelling modulo $(X,<,\psi)$ of $S$. A local branching $(f,g)$ (resp. $(f,e)$) of $S$ modulo $E$ is decreasing modulo $E$ if there exists confluence diagrams of the following form
\[
\raisebox{0.55cm}{
	\xymatrix @R=2em @C=2em {
		{}
		\ar[r] ^-{f}
		\ar[d] _-{\fleq}
		&
		{} 
		\ar@{.>}[r] ^-{f'} 
		& 
		{} \ar@{.>}[r] ^-{g''} 
		& {} \ar@{.>}[r] ^-{h_1} 
		& {}
		\ar@{.>}[d] ^-{e'}
		\\
		{}
		\ar [r] _-{g}
		&
		{}
		\ar@{.>}[r] _-{g'} 
		& {} \ar@{.>}[r] _-{f''} 
		& {} \ar@{.>}[r] _-{h_2}
		&
		{}
	}} , \qquad (resp. 
	\raisebox{0.55cm}{
	\xymatrix @R=2em @C=2em {
		{}
		\ar[r] ^-{f}
		\ar[d] _-{e}
		&
		{} 
		\ar@{.>}[r] ^-{f'} 
		& 
		 {} \ar@{.>}[r] ^-{h_1} 
		& {}
		\ar@{.>}[d] ^-{e'}
		\\
		{}
		\ar@{.>} [rrr] _-{h_2}
		&
		{}
		& {} 
		&
		{}
	}} )
	\]
	such that the following properties hold:
	\begin{enumerate}[{\bf i)}]
		\item $k < \psi(f)$ for all $k$ in $L^X(f')$.
		\item $k < \psi(g)$ for all $k$ in $L^X (g')$.
		\item $f''$ is an identity or a rewriting step labelled by $\psi(f)$.
		\item $g''$ is an identity or a rewriting step labelled by $\psi(g)$.
		\item $k < \psi(f)$ or $k < \psi(g)$ for all $k$ in $L^X(h_1) \cup L^X(h_2)$ (resp. $k \leq \psi(f)$ for any $k$ in $L^X(h_2)$ and $k' < \psi (f)$ for any $k'$ in $L^X(h_1)$).
	\end{enumerate}

\begin{remark}
Note that the definition of decreasingness for a local branching $(f,g)$ where $f$ and $g$ are positive $3$-cells in $\So$ is the same than Van Oostrom's definition. This definition is enlarged for a local branching $(f,e)$ where $f$ is a positive $3$-cell in $\So$ and $E$ is a $3$-cell in $E^\ell$ of length $1$ with the large inequality $k \leq \psi (f)$ in order to make sure that critical branchings of the form $(f,e)$ are decreasing with respect to the quasi-normal form labelling $\psiqnf$ defined in \ref{SSS:QuasiNormalFormLabelling} when rewriting with a linear~$(3,2)$-polygraph modulo $(R,E,S)$ such that $\ER \subseteq S$. Indeed, recall from \cite[Section 3.1]{DMpp18} that in this case these critical branchings are trivially confluent as follows:
\begin{equation*}
\label{E:JKCoherenceTrivialForER}
\xymatrix @R=2em @C=2em{
u 
  \ar[r] ^-{f} 
  \ar@{->} [d] _-{e}
& 
v \ar[d] ^-{\rotatebox{90}{=}}
\\
v'
  \ar[r] _-{e^- \cdot f} 
&
v
} 
\end{equation*}
In that case, $h_2 := e^- \cdot f$ has the same label than $f$ for $\psiqnf$, but we require that this confluence diagram is decreasing.
\end{remark}
	
	Such a diagram is called a decreasing confluence diagram of the branching modulo $(f,e,g)$. A linear~$(3,2)$-polygrah modulo $(R,E,S)$ is \emph{decreasing} is there exists a well-founded labelling $(X,<,\psi)$ of $R$ making all the local branchings $(f,e,g)$ of $S$ modulo $E$ decreasing. It was proven in \cite[Theorem 4.3.3]{AL16}, following the original proof by Van Oostrom for an abstract rewriting system \cite{VOO94}, that any decreasing left-monomial linear~$(3,2)$-polygraph $P$ is confluent. We adapt these proofs to establish the following result:
	
	\begin{theorem}
	\label{T:ConfluenceByDecreasingness}
		Let $(R,E,S)$ be a left-monomial linear~$(3,2)$-polygraph modulo. If $(R,E,S)$ is decreasing, then $S$ is confluent modulo $E$.
	\end{theorem}
	
	Let us at first prove the following two lemmas:
	\begin{lemma}
	\label{L:PastingProperty1}
		Let $(R,E,S,X,<, \psi)$ be a decreasing labelled linear~$(3,2)$-polygraph modulo. For every diagram of the following form
		\[
		\raisebox{0.55cm}{
			\xymatrix @R=2em @C=2em {
				{} \ar@1 [r] ^-{f_1} \ar@1 [d] _-{\fleq} & {} \ar@1 [rr] ^-{f_2} \ar@1 [d] ^-{\fleq} & & {}  \\
				{} \ar@1[r] |-{f_1} \ar@1 [d] _-{e_1} & {} \ar@1 [r] |-{f'_1} & {} \ar@1[d] ^-{e'_1} & {} \\
				{} \ar@1 [r] _-{g_1} & {} \ar@1 [r] _-{g'_1} & {} & {}}}  \]
		such that the confluence modulo $(f_1 \star_2 f'_1, g_1 \star_2 g'_1)$ is decreasing, the inequality \[ | (f'_1, f_2) | \omulteq | (g_1 ,f_1 \star_2 f_2) | \] holds.
	\end{lemma}
	
	\begin{proof}
	By Lemma \ref{L:OperationsMultisets} $\mathbf{ix)}$, we get the following inequality:
		\[ | (f'_1 ,f_2) |  = |(f'_1,f_2)| \cap \multiinf{ |f_1| } \cup |(f'_1, f_2) | - \multiinf{ |f_1|}. \]
		Since $\multiinf{ |f_1| }  \omult |f_1|$, we get that 
\[ | (f'_1 ,f_2) | \omult |f_1| \cup |((f'_1)^{(f_1)}, f_2^{(f_1)})| = | f_1 \star_2 f'_1 | \cup |f_2^{(f_1)} |. \]
Finally, we get from the decreasingness assumption that 
\[ | f_1 \star_2 f'_1 | \cup |f_2^{(f_1)} | \omulteq |(f_1,e_1,g_1)| | \cup |f_2^{(f_1)} | = |(g_1, f_1 \star_2 f_2) |. \]
	\end{proof}

	\begin{lemma}
	\label{L:PastingProperty2}
		Let $(R,E,S,X,<, \psi)$ be a decreasing labelled linear~$(3,2)$-polygraph modulo. For every diagram of the following form
		\[ 
		\raisebox{0.55cm}
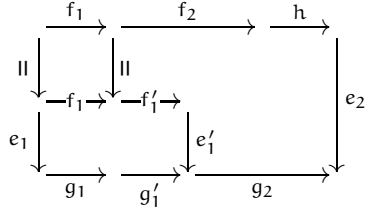
{
			\xymatrix @R=2em @C=2em {
				{} \ar@1 [r] ^-{f_1} \ar@1 [d] _-{\fleq} & {} \ar@1 [rr] ^-{f_2} \ar@1 [d] ^-{\fleq} & & {} \ar@1 [r] ^-{h} & {} \ar@1 [dd]^{e_2} \\
				{} \ar@1[r] |-{f_1} \ar@1 [d] _-{e_1} & {} \ar@1 [r] |-{f'_1} & {} \ar@1[d] ^-{e'_1} & {} & & \\
				{} \ar@1 [r] _-{g_1} & {} \ar@1 [r] _-{g'_1} & {} \ar@1 [rr] _-{g_2} & & {}
			}}  \]
			such that the confluence $(f'_1,e'_1,g'_1)$ and $(f_2 \star_2 h, e_2, g_2)$ are decreasing, i.e. the following inequalities hold:
			\begin{enumerate}[{\bf a)}]
				\item $| g_1 \star_2 g'_1 | \omulteq |(f_1,e_1,g_1) |$ and $|f_1 \star_2 f'_1 | \omulteq |(f_1,e_1,g_1)|$,
				\item $|f'_1 \star_2 e'_1 \star_2 g_2 | \omulteq |(f'_1,f_2) |$ and $|f_2 \star_2 h | \omulteq |(f'_1 , f_2) |$
			\end{enumerate}
			Then the following inequalities hold:
			\[ | g_1 \star_2 g'_1 \star_2 g_2 | \omulteq | (f_1 \star_2 f_2, e_1 , g_1 )| \qquad \text{and} \qquad | f_1\star_2 f_2 \star_2h | \omulteq | (f_1 \star_2 f_2, e_1 , g_1 )| \]
		\end{lemma}
		
\begin{proof}
To shorten the notations in this proof, we will denote the $2$-cell $f \star_2 g$ by simply $fg$. For the second inequality, we get that
	\begin{align*}
		| f_1 f_2 h | & = | f_1 f_2 | \cup |h ^{(f_1 f_2)} | = | f_1 f_2 | \cup |h ^{(f_1) (f_2)} | \\
		& \omulteq | f_1 f_2 | \cup |(f'_1) ^{( f_1)} |  \\
		\end{align*}
	since $|h ^{(f_2)} | \omulteq |f'_1|$ and $|f_1 f'_1| \omulteq |f_1 | \cup |g_1|$ respectively by properties $\mathbf{b)}$ and $\mathbf{a)}$. For the first inequality, we have by Lemma \ref{L:OperationsMultisets} $\mathbf{ix)}$ that 
\[ 	|g_1 g'_1 g_2 |   = |g_1 g'_1 | \cup |g_2^{(g_1 g'_1)} | = | g_1 g'_1 | \cup \left[ \left( |g_2^{(g_1 g'_1)} | \cap \multiinf{f_1} \right) \cup \left( |g_2^{(g_1 g'_1)} | - \multiinf{f_1} \right) \right]. \]
We deduce from \cite[Claim in Lemma 3.5]{VOO94} the following two inequalities, that we do not detail here:
\[ 	|g_1 g'_1 g_2 |  \omulteq |g_1| \cup |f_1| \cup |g_2 ^{(g_1 g'_1) (f_1)} | \omulteq |g_1 | \cup |f_1 | \cup | g_2^{(f'_1)(f_1)} |. \]
Since $|g_2^{(f'_1)} | \omulteq |f_2|$ by $\mathbf{b)}$, we finally get that 
\[ 	|g_1 g'_1 g_2 | \omulteq |g_1 | \cup |f_1| \cup |f_2 ^{(f_1)} |  = |g_1 | \cup |f_1 f_2 | = |(f_1 f_2, e_1, g_1)|. \]	
\end{proof}

Before proving Theorem \ref{T:ConfluenceByDecreasingness}, let us also establish the following preliminary lemma:
\begin{lemma}
\label{L:PreliminaryDecreasingness}
Let $(R,E,S,X,<,\psi)$ be a decreasing labelled linear~$(3,2)$-polygraph modulo. For any branching $(f,e,g)$ of $S$ modulo $E$ with $f$ and $g$ positive $3$-cells in $\So$ and $e$ a $3$-cell in $E^\ell$ of length $1$, there exist a confluence $(f',e',g')$ of this branching such that
\[ | f \star_2 f' | \omulteq |(f,e,g)| \quad \text{and} \quad |g \star_2 g' | \omulteq |(f,e,g)| \]
\end{lemma}

\begin{proof}
Let us denote by $(X,<, \psi)$ the well-founded labelling on $S$ making it decreasing. We consider such a branching $(f,e,g)$ of $S$ modulo $E$, and we prove this result by well-founded induction, assuming  that it is true for any branching $(f'' , e'' , g'' )$ of $S$ modulo $E$ such that $ |(f'',e'',g'') | \omult |(f,e,g)|$.

The local branching $(f,e)$ of $S$ modulo $E$ being decreasing by assumption, there exist positive $3$-cells $f',f'_1$ and $h_2$ in $S^\ell$ such that $k \leq \psi((f)$ for any $k$ in $L^X(h_2)$. Let us fix a decomposition $h_2 = h_2^1 \star_2 h_2^2$ where $h_2$ is a positive $3$-cell in $\So$. Then $(h_1^1,g_1)$ is a local branching of $S$ modulo $E$ and by decreasingness, there exist a decreasing confluence of this local branching, as depicted in the following diagram:
\[
		\raisebox{0.55cm}{
			\xymatrix @R=2em @C=2em {
\ar@1 [r] ^-{f_1} \ar [d] _-{e_1} & \ar [r] ^-{f'} & \ar [r] ^-{f'_1} & \ar [d] ^-{e'_1} \\
\ar [d] _-{\fleq} \ar [r] |-{h_2^1} & \ar [d] ^-{\fleq} \ar [rr] ^-{h_2^2} & &  \\
\ar [d] _-{\fleq} \ar [r] |-{h_2^1} & \ar [rr] ^-{k_1} & & \ar [d] ^-{e_2} \\
\ar [r] _-{g_1} & \ar [rr] _{g'1} & &  
 }} \]
By decreasingness of $(f,e)$, we have that $|h_2^2 | \omulteq |f_1|$ and by decreasingness of $(h_2^1,g)$, we have that $|k_1| \omult [g_1|$ so that $|(f,e,g)| \omult |(h_2^2,k_1)|$ and by induction, this branching admits a confluence $(h_3,e_3,k_2)$ satisfying 
\[ |h_2^2 \star_2 h_3 | \omulteq |(h_2^2,k_1)| \quad \text{and} \quad |k_1 \star_3 k_2| \omulteq |(h_2^2,k_1)| \]
We can now repeat the same process on the branchings $((e'_1)^-, h_3)$ and $(e_2,k_2)$ to obtain a confluence modulo of these branchings as follows:
   \[
		\raisebox{0.55cm}{
			\xymatrix @R=2em @C=2em {
\ar@1 [r] ^-{f_1} \ar [d] _-{e_1} & \ar [r] ^-{f'} & \ar [r] ^-{f'_1} & \ar [d] ^-{e'_1} \ar [rr] ^-{f_2} & & \ar [d] ^-{e''_1} \\
\ar [d] _-{\fleq} \ar [r] |-{h_2^1} & \ar [d] ^-{\fleq} \ar [rr] ^-{h_2^2} & & \ar [r] |-{h_3} & \ar [r] |-{h_4} \ar [d] ^-{e_3} &  \\
\ar [d] _-{\fleq} \ar [r] |-{h_2^1} & \ar [rr] ^-{k_1} & & \ar [d] ^-{e_2} \ar [r] |-{k_2} & \ar [r] |-{k_3} & \ar [d] ^-{e'_2} \\
\ar [r] _-{g_1} & \ar [rr] _{g'1} & & \ar [rr] _-{g_2} & &  
 }} \]
One can repeat this process, however it terminates in finitely many steps, otherwise this would lead to infinite sequences $(h_n)_{n \in \N}$ and $(k_n)_{n \in \N}$ satsifying 
\[ |f| \omulteq |h_2| \omult |h_3| \omulteq |h_4| \omult |h_5| \ldots , \qquad  |g| \omult |k_1| \omult |k_2| \ldots \]
yielding two infinite strictly decreasing sequences for $\omult$, which is impossible since by assumption, $<$ is well-founded and then so is $\omult$ as explained in section \ref{SSS:MultisetOrdering}.
\end{proof}

Let us now prove Theorem \ref{T:ConfluenceByDecreasingness}:
\begin{proof}
Let us denote by $(X,<, \psi)$ the well-founded labelling on $S$ making it decreasing. We consider a branching $(f,e,g)$ of $S$ modulo $E$ such that $f$ and $g$ are positive $3$-cells of $S^\ell$. We prove by well-founded induction on the labels that $(f,e,g)$ can be completed into a confluence modulo diagram with positive $3$-cells $f'$, $g'$ in $S^\ell$ and a $3$-cell $e'$ in $E^\ell$ such that
\begin{equation}
\label{E:BranchingsDecreasing}
 | f \star_2 f' | \omulteq |(f,e,g) |, \quad \text{and} \quad |g \star_2 g' | \omulteq |(f,e,g)| 
 \end{equation}
We assume that for any branching $(f'' , e'' , g'' )$ of $S$ modulo $E$ such that $ |(f'',e'',g'') | \omult |(f,e,g)|$, there exists a decreasing confluence modulo of the branching $(f'',e'',g'')$.
Let us choose decompositions $f=f_1 \star_2f_2$ and $g= g_1 \star_2 g_2$ where $f_1$, $g_1$ belong to $\So$ and $f_2$ and $g_2$ are in $S^\ell$. By Lemma \ref{L:PreliminaryDecreasingness}, the branching $(f_1,e,g_1)$ admits a confluence modulo $(f'_1,e_1,g'_1)$ satsifying the conditions of (\ref{E:BranchingsDecreasing}), as depicted on the following diagram:
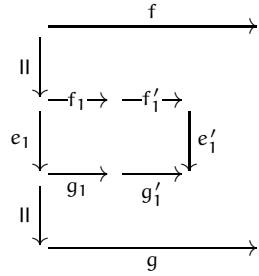
\[
		\raisebox{0.55cm}{
			\xymatrix @R=2em @C=2em {
				{} \ar@1 [rrr] ^-{f} \ar@1 [d] _-{\fleq} & {} & {} & {}  \\
				{} \ar@1[r] |-{f_1} \ar@1 [d] _-{e_1} & {} \ar@1 [r] |-{f'_1} & {} \ar@1[d] ^-{e'_1} & {} \\
				{} \ar@1 [r] _-{g_1} \ar@1 [d]_-{\fleq} & {} \ar@1 [r] _-{g'_1} & {} & {} \\
	{} \ar@1 [rrr] _-{g} & & &{} }} \]
Using Lemma \ref{L:PastingProperty1}, we get that $|f_2| \cup |f'_1 | \omult |(f,e,g)|$ and $|g_2| \cup |g'_1| \omult |(f,e,g)|$ so that by induction on the branchings $(f_2,f'_1)$ and $(g'_1,g_2)$, there exist positive $3$-cells $f_3,f'_2,g_3,g'_2$ in $S^\ell$ satisfying the conditions of (\ref{E:BranchingsDecreasing}) and $3$-cells $e_2$, $e'_2$ in $E^\ell$ as in the following diagram:
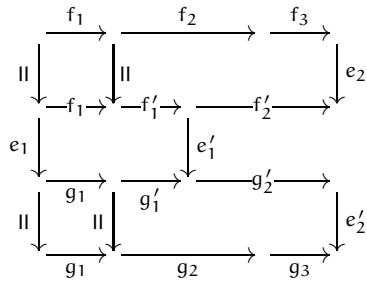
\[
		\raisebox{0.55cm}{
			\xymatrix @R=2em @C=2em {
				{} \ar@1 [r] ^-{f_1} \ar@1 [d] _-{\fleq} & {} \ar@1 [d] ^-{\fleq} \ar@1[rr] ^-{f_2} & {} & {} \ar@1[r] ^-{f_3} & {} \ar@1[d] ^-{e_2} \\
				{} \ar@1[r] |-{f_1} \ar@1 [d] _-{e_1} & {} \ar@1 [r] |-{f'_1} & {} \ar@1[d] ^-{e'_1} \ar@1[rr] |-{f'_2} & {} & {}  \\
				{} \ar@1 [r] _-{g_1} \ar@1 [d]_-{\fleq} & {} \ar@1 [r] _-{g'_1} \ar@1 [d] _-{\fleq} & {} \ar@1 [rr] |-{g'_2} & {} & {}   \ar@1[d] ^-{e'_2} \\
	{} \ar@1 [r] _-{g_1} & \ar@1[rr] _-{g_2} & &{} \ar@1 [r] _-{g_3} & {} }} \]
Now, either there is a $2$-cell $e''': t_2(e_2) \fl s_2(e'_2)$  in $E^\ell$, and the confluence diagram obtained satisfy the conditions of (\ref{E:BranchingsDecreasing}) using Lemma \ref{L:PastingProperty2} on the top part of the diagram and decreasingness of the confluence modulo $(g'_2,e'_2,g_3)$. Otherwise, the branching $(f'_2, e'_1,g'_2)$ is a branching of $S$ modulo $E$ whose label is strictly smaller than $|(f,e,g)|$ with respect to $\omult$ by construction. Applying induction on this branching, there exists a confluence modulo $(f'_3,e_3,g'_3)$ of this branching satisfying the conditions of (\ref{E:BranchingsDecreasing}). Then, we may still apply induction on the branchings $(e_2,f'_3)$ and $(e'_2, g'_3)$ of $S$ modulo $E$, whose respective multisets $|f'_3|$ and $|g'_3|$ are strictly smaller than $|(f,e,g)|$ with respect to $\omult$ by construction. We get the following situation:
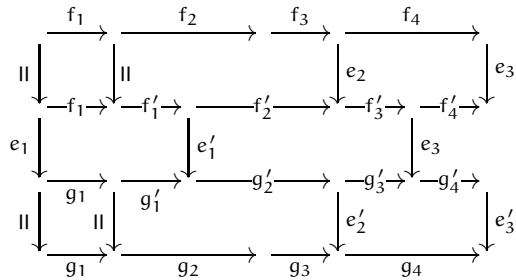
\[
		\raisebox{0.55cm}{
			\xymatrix @R=2em @C=2em {
				{} \ar@1 [r] ^-{f_1} \ar@1 [d] _-{\fleq} & {} \ar@1 [d] ^-{\fleq} \ar@1[rr] ^-{f_2} & {} & {} \ar@1[r] ^-{f_3} & {} \ar@1[d] ^-{e_2} \ar@1[rr] ^-{f_4} & {} & {} \ar@1 [d]^{e_3} \\
				{} \ar@1[r] |-{f_1} \ar@1 [d] _-{e_1} & {} \ar@1 [r] |-{f'_1} & {} \ar@1[d] ^-{e'_1} \ar@1[rr] |-{f'_2} & {} & {} \ar@1 [r] |-{f'_3} & {} \ar@1[d] ^-{e_3} \ar@1 [r] |-{f'_4} & {}  \\
				{} \ar@1 [r] _-{g_1} \ar@1 [d]_-{\fleq} & {} \ar@1 [r] _-{g'_1} \ar@1 [d] _-{\fleq} & {} \ar@1 [rr] |-{g'_2} & {} & {}   \ar@1[d] ^-{e'_2} \ar@1 [r] |-{g'_3} & {}  \ar@1 [r] |-{g'_4} & {} \ar@1 [d] ^-{e'_3}  \\
	{} \ar@1 [r] _-{g_1} & \ar@1[rr] _-{g_2} & &{} \ar@1 [r] _-{g_3} & {} \ar@1 [rr] _-{g_4} & {} & {} }} \]
This process can be repeated, however it terminates in finitely many steps to reach a confluence modulo of the branching $(f,e,g)$, using a similar argument than in the proof of Lemma \ref{L:PreliminaryDecreasingness}. This confluence modulo satisty the properties of (\ref{E:BranchingsDecreasing}) from successive use of Lemmas \ref{L:PastingProperty1} and \ref{L:PastingProperty2}.
\end{proof}

\subsection{Quasi-termination of polygraphs modulo}
\label{SS:QuasiTermination}
In this section, we recall some results on quasi-terminating linear~$(3,2)$-polygraphs from \cite{AM17} and \cite{DM19}.

\subsubsection{Quasi-termination and exponentiation freedom}
A linear~$(3,2)$-polygraph $P$ is \emph{quasi-terminating} if for each sequence $(u_n)_{n \in \N}$ of $2$-cells such that for each $n$ in $\N$ there is a rewriting step from $u_n$ to $u_{n+1}$, the sequence $(u_{n})_{n \in \N}$ contains an infinite number of occurences of the same $2$-cell. A $2$-cell $u$ of $P$ is called a \textit{quasi-normal form} if for any rewriting step from $u$ to another $2$-cell $v$, there exist a rewriting sequence from $v$ to $u$. A quasi-normal form of a $2$-cell $u$ is a quasi-normal form $\cl{u}$ such that there exists a rewriting sequence from $u$ to $\cl{u}$. A linear~$(3,2)$-polygraph $P$ is said \emph{exponentiation free} is for any $2$-cell $u$, there does not exist a $3$-cell $\alpha$ in $P$ such that $$ 
\text{$u \overset{\alpha}{\fl} \lambda u + h$ with $\lambda \in \K \backslash \{ 0 \}$ and $h \ne 0$}. $$ 

\subsubsection{Labelling to the quasi-normal form}
\label{SSS:QuasiNormalFormLabelling}
Given a quasi-terminating linear~$(3,2)$-polygraph $P$, any $2$-cell $u$ in $P_2^\ast$ admits at least quasi normal form. For such a $2$-cell $u$, we fix a choice of a quasi normal form denoted by $\cl{u}$. Then we get a quasi-normal form map $s: P_2^\ell \fl P_2^\ell$ sending a $2$-cell $u$ in $P_2^\ast$ on $\cl{u}$. The \emph{labelling to the quasi-normal form}, labellling QNF for short associates to the map $s$ the labelling $\psiqnf: P_{\text{stp}} \fl \mathbb{N}$ defined by
\[ \psiqnf (f) = d(t_1(f), \cl{t_1(f)}) \]
where $d(t_1(f), \cl{t_1(f)})$ represent the minimal number of rewriting steps needed to reach the quasi normal form $\cl{t_1(f)}$ from $t_1(f)$.

\subsubsection{Proving confluence modulo under quasi-termination}
Alleaume established in \cite{ALPhD} that if $P$ is a quasi-terminating and exponentiation free linear~$(3,2)$-polygraph, then it is locally confluent if and only if all its critical branchings are confluent. The proof of this result consists in two different stepsFirst of all, one proves that the local additive branchings are confluent under quasi-termination and exponentiation freedom using Lemma \ref{L:CellDecomposition}. Then, one proves that the overlapping local branchings are confluent in a similar fashion than in Theorem \ref{T:CriticalBranchingLemmaModulo}.
This result is adapted to rewriting modulo in \cite{DM19}, where it was also established that one can prove decreasingness of a linear~$(3,2)$-polygraph modulo using decreasing confluence modulo of critical branchings modulo:
\begin{proposition}
\label{P:DecreasingnessFromCriticalPairs}
Let $(R,E,S)$ be a left-monomial linear~$(3,2)$-polygraph modulo such that $\ERE$ is quasi-terminating and $S$ is exponentiation free. The following properties hold:
\begin{enumerate}[{\bf i)}]
\item $S$ is locally confluent modulo $E$ if and only if its critical branchings of the form $\mathbf{a_0)}$ and $\mathbf{b_0)}$ of Theorem \ref{T:CriticalBranchingLemmaModulo} are confluent modulo $E$.
\item If all these critical branchings of $S$ modulo $E$ are decreasing with respect to  the labelling to the quasi-normal form $\psiqnf$, then $S$ is decreasing.
\end{enumerate}
\end{proposition}

\subsection{Linear bases by confluence modulo}
\label{SS:LinearBasesModulo}
We give a method to compute a hom-basis for a linear~$(2,2)$-category $\mathcal{C}$ from a presentation of $\mathcal{C}$ by a linear~$(3,2)$-polygraph $P$ admitting a convergent subpolygraph $E$ such that the polygraph with set of $3$-cells $R_3 = P_3 \backslash E_3$ is confluent modulo $E$, and $\ERE$ is terminating, or quasi-terminating.

\subsubsection{Splitting of a polygraph}
Given a linear~$(3,2)$-polygraph $P$, recall that a \emph{subpolygraph} of $P$ is a linear~$(3,2)$-polygraph $P'$ such that $P'_i \subseteq P_i$ for any $0 \leq i \leq 3$. A \emph{splitting} of $P$ is a pair $(E,R)$ of linear~$(3,2)$-polygraphs such that:
\begin{enumerate}[{\bf i)}]
\item $E$ is a subpolygraph of $P$ such that $E_{\leq 1} = P_{\leq 1}$,
\item $R$ is a linear~$(3,2)$-polygraph such that $R_{\leq 2} = P_{\leq 2}$ and $P_3 = R_3 \coprod E_3$.
\end{enumerate}
Such a splitting is called  \emph{convergent} if we require that $E$ is convergent. Note that any linear~$(3,2)$-polygraph $P$ admits a convergent splitting given by $(P_0,1,P_2,\emptyset)$ and $(P_0,P_1,P_2,P_3)$. It is not unique in general. The data of a convergent splitting of a linear~$(3,2)$-polygraph $P$ gives two distinct linear~$(3,2)$-polygraphs $R= (P_0,P_1,P_2,R_3)$ and $E= (P_0,P_1,E_2,E_3)$ satisfying $R_{\leq 1} = E_{\leq 1}$ and $E_2 \subseteq P_2$, so that we can construct a linear~$(3,2)$-polygraph modulo from $R$ and $E$. Note that when $P$ is left-monomial, if $(E,R)$ is a splitting of $P$, then both $E$ and $R$ are left-monomial.

\subsubsection{Normal forms modulo}
Let us consider a linear~$(3,2)$-polygraph $P$ presenting a linear~$(2,2)$-category $\mathcal{C}$, $(E,R)$ a convergent splitting of $P$ and $(R,E,S)$ a normalizing linear~$(3,2)$-polygraph modulo such that $S$ is confluent modulo $E$.

$S$ being normalizing, each $2$-cell $u$ of $R_2^\ell$ admits at least one normal form with respect to $E$, and all these normal forms are congruent with respect to $E$. We fix such a normal form that we denote by $\widehat{u}$, with the convention that if $u$ is already a normal form with respect to $E$, then $\widehat{u} = u$. By convergence of $E$, any $2$-cell $u$ of $R_2^\ell$ admits a unique normal form with respect to $E$, that we denote by $\widetilde{u}$. Note that when $S$ is confluent modulo $E$, the element $\tilda{u}$ does not depend on the chosen normal form $\widehat{u}$ for $u$ with respect to $S$, since two normal forms of $u$ being equivalent with respect to $E$, they have the same normal form with respect to $E$. A \emph{normal form for $(R,E,S)$} of a $2$-cell $u$ in $R_2^\ell$ is a $2$-cell $v$ such that $v$ appears in the monomial decomposition of $\widetilde{w}$ where $w$ is a monomial in the support of $\widehat{u}$. Given a $2$-cell $u$ in $R_2^\ell$, we denote by $\text{NF}_{(R,E,S)}(u)$ the set of all normal forms of $u$ for $(R,E,S)$. Such a set is obtained by reducing $u$ into its chosen normal form with respect to $S$, then taking all the monomials appearing in the $E$-normal form of each element in $\text{Supp}(\widehat{u})$. Note that when $E$ is also right-monomial, the $E$-normal form of a monomial in normal form with respect to $S$ already is a monomial. In particular, this is the case when $E$ is the polygraph of isotopies described in \ref{SSS:PolygraphIsotopy}.

\begin{lemma}
\label{L:BasisLemma}
Let $P$ be a left-monomial linear $(3,2)$-polygraph, $(E,R)$ be a convergent splitting of $P$ and $(R,E,S)$ be a normalizing left-monomial linear~$(3,2)$-polygraph modulo such that $S$ is confluent modulo $E$, and let $\C$ be the category presented by $P$. Then, for any parallel $1$-cells $x$ and $y$ in $R_1^\ast$, the map $\gamma_{x,y} : R_2^\ell(x,y) \to \C(x,y)$ sending each $2$-cell to its congruence class in $\C$ has for kernel the subspace of $R_2^\ell$ made of $2$-cells $u$ such that $\tilda{u} = 0$.
\end{lemma}

\begin{proof}
Let us denote by $N$ the set $\{ u \in R_2^\ell \; ; \; \tilda{u} = 0 \}$. Then $N\subseteq \Ker(\gamma)$ since if $u \in N$, there exist positive $3$-cells $f$ in $E^\ell$ and $e$ in $ E^\ell$ such that \[ \xymatrix{ u \ar [r] ^-{f} & \hat{u} \ar[r] ^-{e} & \tilda{u} = 0 } \]
Thus by definition of $S$ there exist a zig-zag sequence of rewriting steps either of $R$ or $E$ between $u$ and $0$, so that $\cl{u} = 0$ in $\C$ and $u $ belongs to $ \Ker(\gamma)$.
Conversely, if $u$ belongs to $\Ker(\gamma)$, that is $\pi (u) = 0$ where $\pi : R_2^\ell \fl \C$ is the canonical projection, there is a zig-zag sequence of rewriting steps $(f_i)$ for $0 \leq i \leq n$ with $f_i$ being either a rewriting step of $R$ or a rewriting step of $E$ such that
\[ \xymatrix{ u \ar[r] ^-{f_1} & u_1 & u_2 \ar [l] _-{f_2} & \dots \ar [r] ^-{f_{n-2}} & u_{n-1} & u_n \ar [r] ^-{f_n} \ar [l] _-{f_{n-1}} & v } \]
$S$ being confluent modulo $E$, it is Church-Rosser modulo $E$ from \ref{SSS:JKConfluence}, and then by \ref{SSS:ChurchRosserModulo}, we get that there exist rewriting sequences $f: u \fl \hat{u}$ and $g: 0 \fl \hat{0}$ in $S^\ell$ and a $3$-cell $e: \hat{v} \fl \hat{0}$ in $E^\ell$. As $S$ is left-monomial, $0$ is a normal form with respect to $S$ so that $\hat{0} = 0$. Then $\hat{u}$ and $0$ are equivalent with respect to $E$ so that, by convergence of the linear~$(3,2)$-polygraph $E$, we get that $\tilda{u} = \widetilde{0}$, and similarly $\widetilde{0} = 0$ since $E$ is left-monomial and $0$ is a normal form with respect to $E$. This finishes the proof.
\end{proof}

We then obtain the following result:
\begin{theorem}
\label{T:BasisByConfluenceModulo}
Let $P$ be a linear~$(3,2)$-polygraph presenting a linear~$(2,2)$-category $\C$, $(E,R)$ a convergent splitting of $P$ and $(R,E,S)$ a linear~$(3,2)$-polygraph modulo such that 
\begin{enumerate}[{\bf i)}]
\item $S$ is normalizing,
\item $S$ is confluent modulo $E$,
\end{enumerate}
then the set of all normal forms for $(R,E,S)$ is a hom-basis of $\mathcal{C}$.
\end{theorem}

\begin{proof}
Let us denote by $B$ the set of $E$-normal forms of all monomials in normal forms with respect to $S$,
and let $B^{\text{Mon}}$ be the set of all normal forms for $(R,E,S)$. Note that by definition, $B^{\text{Mon}}$ is obtained by considering all the $2$-cells in the support of the elements of $B$. Since $S$ is left-monomial, each normal form in $R_2^\ell$ can be decomposed into a linear combination of monomials in normal form with respect to $S$, and by left-monomiality of $E$, we get that an element of $B$ is a linear combination of monomials in $B^{\text{Mon}}$, so that $B^{\Mon}$ is a basis of $B$. For any $1$-cells $p$ and $q$ of $\C$, the map $\gamma_{x,y} : R_2^\ell (p,q) \fl \C_2(p,q)$ is surjective by definition, each $2$-cell of $\C_2(p,q)$ having at least one representative in $R_2^\ell(p,q)$. Moreover, the restriction of $\gamma_{p,q}$ to the subvector space $B$ of $R_2^\ell$ has for kernel $B \cap \Ker(\gamma_{p,q})$, which is reduced to $\{ 0 \}$ by confluence modulo $E$ of $S$, using Lemma \ref{L:BasisLemma}. This proves that $(\gamma_{p,q})_{|B} $ is a bijection between $B$ and $\C_2(p,q)$, and so $B^{\text{Mon}}$ is a linear basis of $\C_2(p,q)$. 
\end{proof}

\subsubsection{Linear bases under quasi-termination}
\label{R:AdaptationQuasiTerm}
Note that both Lemma \ref{L:BasisLemma} and Theorem \ref{T:BasisByConfluenceModulo} have an adaptation in a non-normalizing but quasi-terminating setting. Indeed, instead of fixing a normal form $\widehat{u}$ with respect to $S$ for any $u$ in $R_2^\ell$, we fix a choice of a quasi-normal form $\cl{u}$ for $u$ satisying $\cl{u} = u$ if $u$ already is a quasi-normal form with respect to $S$. By confluence modulo, $u$ and $v$ are $2$-cells of $R_2^\ell$ such that there is a $3$-cell $e: u \fl v$ in $E^\ell$, then the $2$-cells $\cl{u}$ and $\cl{v}$ are equivalent modulo $E$. We then say that a \emph{quasi-normal form} for $(R,E,S)$ is a monomial appearing in the monomial decomposition of the $E$-normal form of a monomial in $\text{Supp}(\cl{u})$. With a similar proof than above, we obtain the following result:
\begin{theorem}
\label{T:BasisByConfluenceModulob}
Let $P$ be a linear~$(3,2)$-polygraph presenting a linear~$(2,2)$-category $\mathcal{C}$, $(E,R)$ a convergent splitting of $P$ and $(R,E,S)$ a linear~$(3,2)$-polygraph modulo such that
\begin{enumerate}[{\bf i)}]
	\item $S$ is quasi-terminating,
	\item $S$ is confluent modulo $E$,
\end{enumerate}
Then the set of quasi-normal forms form $(R,E,S)$ is a hom-basis of $\mathcal{C}$.
\end{theorem}

\section{Affine oriented Brauer category}
\label{S:AOB}
In this section, we illustrate the previous results by computing a hom-basis for the affine Oriented Brauer linear~$(2,2)$-category $\aob$. We describe a linear~$(3,2)$-polygraph $(E,R,\ER)$ for which we prove that $\ER$ is quasi-terminating and $\ER$ is confluent modulo. As a consequence, we prove that a choice of quasi-normal forms yields to the well-known basis obtained in \cite{BCNR14,AL16}.

\subsection{A presentation of $\mathcal{AOB}$}
\label{SS:PresentationAOB}
We recall from \cite{SAV18} the natural presentation of the affine oriented Brauer category from the degenerate affine Hecke monoidal category.

\subsubsection{The degenerate affine Hecke category}
Let $\deghecke$ be the linear~$(2,2)$-category with only one $0$-cell, one generating $1$-cell \raisebox{-3mm}{$\scalebox{0.8}{\upb{}}$} , two generating $2$-cells
\[ \raisebox{-6mm}{$\crossingup{}{}$} : \raisebox{-3mm}{$\scalebox{0.8}{\upb{}}$} \star_0 \raisebox{-3mm}{$\scalebox{0.8}{\upb{}}$} \fl \raisebox{-3mm}{$\scalebox{0.8}{\upb{}}$} \star_0 \raisebox{-3mm}{$\scalebox{0.8}{\upb{}}$} \quad \text{and} \quad \raisebox{2mm}{$\udottsl{}$} : \raisebox{-3mm}{$\scalebox{0.8}{\upb{}}$} \fl \raisebox{-3mm}{$\scalebox{0.8}{\upb{}}$}  \]
and three relations 
\[ \tcrossup{}{} \; \raisebox{8mm}{$=$} \;  \didupb{}{} \raisebox{5mm}{$,$} \qquad \qquad \ybgauchehaut{}{}{}  \; \raisebox{8mm}{$=$} \; \raisebox{1mm}{$\ybdroithaut{}{}{}$} \raisebox{5mm}{$,$} \qquad \udcrossul{}{} \; \raisebox{8mm}{$=$} \; \udcrossdr{}{} \; \raisebox{8mm}{$+$} \; \didup{}{} \raisebox{5mm}{$.$}  \]

Following \cite{SAV18}, $\text{End}_{\deghecke} \left( \raisebox{-3mm}{$\scalebox{0.8}{\upb{}}$}^{\otimes n} \right)$ is isomorphic to the degenerate affine Hecke algebra of degree $n$.

\subsubsection{The linear~$(2,2)$-category $\aob$}
 To define the affine oriented Brauer linear~$(2,2)$-category $\aob$, we add to this data an additional generating $1$-cell \raisebox{-3mm}{$\scalebox{0.8}{\dwb{}}$} that we require to be right dual to \raisebox{-3mm}{$\scalebox{0.8}{\upb{}}$}. Following \ref{SSS:AdjunctionsIn2Cat}, this requires the existence of unit and counit $2$-cells
\[ 
\cuprsl{} \; : \;  1 \fl \raisebox{-3mm}{$\scalebox{0.8}{\dwb{}}$} \star_0 \raisebox{-3mm}{$\scalebox{0.8}{\upb{}}$}, \qquad \text{and} \qquad \caprsl{}{} \; : \; \raisebox{-3mm}{$\scalebox{0.8}{\upb{}}$} \star_0 \raisebox{-3mm}{$\scalebox{0.8}{\dwb{}}$} \fl 1. \]
where $1$ denoted the identity $1$-cell on the only $0$-cell of $\deghecke$. These $2$-cells have to satisfy the adjunction relations 
\[ 
\mathord{
\begin{tikzpicture}[baseline = 0]
  \draw[->,thick,black] (0.3,0) to (0.3,.4);
	\draw[-,thick,black] (0.3,0) to[out=-90, in=0] (0.1,-0.4);
	\draw[-,thick,black] (0.1,-0.4) to[out = 180, in = -90] (-0.1,0);
	\draw[-,thick,black] (-0.1,0) to[out=90, in=0] (-0.3,0.4);
	\draw[-,thick,black] (-0.3,0.4) to[out = 180, in =90] (-0.5,0);
  \draw[-,thick,black] (-0.5,0) to (-0.5,-.4);
\end{tikzpicture}
}  \; = \;
\mathord{\begin{tikzpicture}[baseline=0]
  \draw[->,thick,black] (0,-0.4) to (0,.4);
\end{tikzpicture}
} \; \raisebox{-3mm}{$,$} \qquad
\mathord{
\begin{tikzpicture}[baseline = 0]
  \draw[->,thick,black] (0.3,0) to (0.3,-.4);
	\draw[-,thick,black] (0.3,0) to[out=90, in=0] (0.1,0.4);
	\draw[-,thick,black] (0.1,0.4) to[out = 180, in = 90] (-0.1,0);
	\draw[-,thick,black] (-0.1,0) to[out=-90, in=0] (-0.3,-0.4);
	\draw[-,thick,black] (-0.3,-0.4) to[out = 180, in =-90] (-0.5,0);
  \draw[-,thick,black] (-0.5,0) to (-0.5,.4);
\end{tikzpicture}
}
\; = \;
\mathord{\begin{tikzpicture}[baseline=0]
  \draw[<-,thick,black] (0,-0.4) to (0,.4);
\end{tikzpicture}
} \; \raisebox{-3mm}{$.$} 
\]
We also add an additional $2$-cell defined by a right-crossing as follows:
\begin{equation*} \label{E:RightCrossing}
\mathord{
\begin{tikzpicture}[baseline = 0]
	\draw[<-,thick,black] (0.28,-.3) to (-0.28,.4);
	\draw[->,thick,black] (-0.28,-.3) to (0.28,.4);
\end{tikzpicture}
}
:=
\trightcrosssl{}{} 
\end{equation*} 
that we require to be invertible, namely there exists a two-sided inverse to this $2$-cell, that we will denote by \; $\mathord{
\begin{tikzpicture}[baseline = 0]
	\draw[->,thick,black] (0.28,-.3) to (-0.28,.4);
	\draw[<-,thick,black] (-0.28,-.3) to (0.28,.4);
\end{tikzpicture}
}$. 
The resulting category $\aob$ is called the \emph{affine oriented Brauer category}. It was proved to be a pivotal linear~$(2,2)$-category in \cite{BRU18}, with \raisebox{-3mm}{$\scalebox{0.8}{\dwb{}}$} also being the left dual of \raisebox{-3mm}{$\scalebox{0.8}{\upb{}}$} and the unit and counit $2$-cells being defined as follows:
\[ \raisebox{-3mm}{$\cuplsl{}$} \; \raisebox{-3mm}{$=$}  \; \fishdlsl{} \qquad
\raisebox{-3mm}{$\caplsl{}$} \; \raisebox{-3mm}{$=$}  \; \raisebox{-2mm}{$\fishulsl{}$} \]

The left crossing $2$-cell is then proved to be equal to 
\[ \tleftcrosssl{}{} \]
The inverse condition is then given by the following two relations:
\[ 
\tdcrosslrsl{}{} \; \raisebox{-3mm}{$=$} \;  \mathord{
\begin{tikzpicture}[baseline = 0]
	\draw[<-,thick,black] (0.08,-1.2) to (0.08,.4);
	\draw[->,thick,black] (-0.38,-1.2) to (-0.38,.4);
\end{tikzpicture}} \qquad \qquad 
\tdcrossrlsl{}{} \; \raisebox{-3mm}{$=$} \;  \mathord{
\begin{tikzpicture}[baseline = 0]
	\draw[->,thick,black] (0.08,-1.2) to (0.08,.4);
	\draw[<-,thick,black] (-0.38,-1.2) to (-0.38,.4);
\end{tikzpicture}}
\]

\subsubsection{The linear~$(3,2)$-polygraph $\aob$}
\label{SSS:LinearPolAOB}
Let $\cl{\aob}$ be the linear~$(3,2)$-polygraph having:
\begin{enumerate}[{\bf i)}]
\item one $0$-cell,
\item two biadjoint generating $1$-cells \raisebox{-3mm}{$\scalebox{0.8}{\upb{}}$} and \raisebox{-3mm}{$\scalebox{0.8}{\dwb{}}$},
\item $8$ generating $2$-cells:
\begin{equation}
\label{E:GenCellsAOB} \raisebox{-6mm}{$\crossingup{}{}$} \quad \raisebox{-6mm}{$\crossingdn{}{}$} \quad \raisebox{2mm}{$\udottsl{}$} \quad \raisebox{2mm}{$\ddottsl{}$}
\quad \caplsl{} \quad \cuplsl{} \quad \caprsl{} \quad \cuprsl{} 
\end{equation}
\item the following families of $3$-cells:
\begin{enumerate}[{\bf a)}]
\item Isotopy $3$-cells:
\begin{align}
\label{E:IsotopyCells}
& \mathord{
\begin{tikzpicture}[baseline = 0]
  \draw[->,thick,black] (0.3,0) to (0.3,.4);
	\draw[-,thick,black] (0.3,0) to[out=-90, in=0] (0.1,-0.4);
	\draw[-,thick,black] (0.1,-0.4) to[out = 180, in = -90] (-0.1,0);
	\draw[-,thick,black] (-0.1,0) to[out=90, in=0] (-0.3,0.4);
	\node at (-0.13,0.12) {$\bullet$};
	\node at (0.07,0.2) {$\alpha$};
	\draw[-,thick,black] (-0.3,0.4) to[out = 180, in =90] (-0.5,0);
  \draw[-,thick,black] (-0.5,0) to (-0.5,-.4);
\end{tikzpicture}
} \; \overset{i_1^\alpha}{\Rrightarrow} \;
\mathord{\begin{tikzpicture}[baseline=0]
  \draw[->,thick,black] (0,-0.4) to (0,.4);
  		\node at (0,0) {$\bullet$};
	\node at (0.2,0.1) {$\alpha$};
\end{tikzpicture}
}, \quad
\mathord{
\begin{tikzpicture}[baseline = 0]
  \draw[->,thick,black] (0.3,0) to (0.3,-.4);
	\draw[-,thick,black] (0.3,0) to[out=90, in=0] (0.1,0.4);
	\draw[-,thick,black] (0.1,0.4) to[out = 180, in = 90] (-0.1,0);
	\draw[-,thick,black] (-0.1,0) to[out=-90, in=0] (-0.3,-0.4);
		\node at (-0.09,0.12) {$\bullet$};
	\node at (-0.28,0.2) {$\alpha$};
	\draw[-,thick,black] (-0.3,-0.4) to[out = 180, in =-90] (-0.5,0);
  \draw[-,thick,black] (-0.5,0) to (-0.5,.4);
\end{tikzpicture}
} \;
\overset{i_3^\alpha}{\Rrightarrow} \;
\mathord{\begin{tikzpicture}[baseline=0]
  \draw[<-,thick,black] (0,-0.4) to (0,.4);
    		\node at (0,0) {$\bullet$};
	\node at (0.2,0.1) {$\alpha$};
\end{tikzpicture}
}, \quad 
\mathord{
\begin{tikzpicture}[baseline = 0]
  \draw[-,thick,black] (0.3,0) to (0.3,-.4);
	\draw[-,thick,black] (0.3,0) to[out=90, in=0] (0.1,0.4);
	\draw[-,thick,black] (0.1,0.4) to[out = 180, in = 90] (-0.1,0);
	\draw[-,thick,black] (-0.1,0) to[out=-90, in=0] (-0.3,-0.4);
		\node at (-0.09,0.12) {$\bullet$};
	\node at (-0.28,0.2) {$\alpha$};
	\draw[-,thick,black] (-0.3,-0.4) to[out = 180, in =-90] (-0.5,0);
  \draw[->,thick,black] (-0.5,0) to (-0.5,.4);
\end{tikzpicture}
} \;
\overset{i_4^\alpha}{\Rrightarrow} \;
\mathord{\begin{tikzpicture}[baseline=0]
  \draw[->,thick,black] (0,-0.4) to (0,.4);
    		\node at (0,0) {$\bullet$};
	\node at (0.2,0.1) {$\alpha$};
\end{tikzpicture}
}, \quad
\mathord{
\begin{tikzpicture}[baseline = 0]
  \draw[-,thick,black] (0.3,0) to (0.3,.4);
	\draw[-,thick,black] (0.3,0) to[out=-90, in=0] (0.1,-0.4);
	\draw[-,thick,black] (0.1,-0.4) to[out = 180, in = -90] (-0.1,0);
	\draw[-,thick,black] (-0.1,0) to[out=90, in=0] (-0.3,0.4);
		\node at (-0.13,0.12) {$\bullet$};
	\node at (0.07,0.2) {$\alpha$};
	\draw[-,thick,black] (-0.3,0.4) to[out = 180, in =90] (-0.5,0);
  \draw[->,thick,black] (-0.5,0) to (-0.5,-.4);
\end{tikzpicture}
} \;
\overset{i_2^\alpha}{\Rrightarrow} \;
\mathord{\begin{tikzpicture}[baseline=0]
  \draw[<-,thick,black] (0,-0.4) to (0,.4);
    		\node at (0,0) {$\bullet$};
	\node at (0.2,0.1) {$\alpha$};
\end{tikzpicture}
}, \; \text{for any $\alpha \in \{ 0,1 \}$} \\
 & \cuprdlsl{}{}  \overset{i_1^2}{\Rrightarrow} \cuprdrsl{}{} \; \caprdlsl{}{} \overset{i_3^2}{\Rrightarrow} \caprdrsl{}{} \; 
\cupldlsl{}{} \overset{i_2^2}{\Rrightarrow} \cupldrsl{}{} \; \capldlsl{}{} \overset{i_4^2}{\Rrightarrow} \capldrsl{}{} 
\end{align}
\item degenerate affine Hecke $3$-cells:
\[ \tcrossup{}{} \; \raisebox{8mm}{$\overset{\alpha_+}{\tfl}$} \;  \didupb{}{} \raisebox{5mm}{$,$}  \quad \ybgauchehaut{}{}{}  \; \raisebox{8mm}{$\overset{\beta_+}{\Rrightarrow}$} \; \raisebox{1mm}{$\ybdroithaut{}{}{}$} \raisebox{5mm}{$,$} \quad \udcrossul{}{} \; \raisebox{8mm}{$\overset{\gamma_{l,+}}{\tfl}$} \; \udcrossdr{}{} \; \raisebox{8mm}{$+$} \; \didup{}{}  \raisebox{5mm}{$,$} \qquad \udcrossur{}{} \; \raisebox{8mm}{$\overset{\gamma_{r,+}}{\tfl}$} \; \udcrossdl{}{} \; \raisebox{8mm}{$-$} \; \didup{}{} \raisebox{5mm}{$.$}  \]
and the corresponding $3$-cells with downward orientations respectively denoted by $\alpha_-$,$\beta_-$, $\gamma_{l,-}$ and $\gamma_{r,-}$.
\item Invertibility $3$-cells:
\[ 
\tdcrossrlsl{}{} \; \raisebox{-3mm}{$\overset{E}{\tfl}$} \;  \mathord{
\begin{tikzpicture}[baseline = 0]
	\draw[->,thick,black] (0.08,-1.2) to (0.08,.4);
	\draw[<-,thick,black] (-0.38,-1.2) to (-0.38,.4);
\end{tikzpicture}}  \qquad \qquad
\tdcrosslrsl{}{} \; \raisebox{-3mm}{$\overset{F}{\tfl}$} \;  \mathord{
\begin{tikzpicture}[baseline = 0]
	\draw[<-,thick,black] (0.08,-1.2) to (0.08,.4);
	\draw[->,thick,black] (-0.38,-1.2) to (-0.38,.4);
\end{tikzpicture}}
\]
\item $3$-cells defining the caps and cups:
\[ \tfishdrsl{} \overset{A}{\tfl} \cuprsl{} , \quad \tfishulsl{} \overset{B}{\tfl} \caplsl{}, \quad \tfishursl{} \overset{C}{\tfl} \caprsl{}, \quad \tfishdlsl{} \overset{D}{\tfl} \cuplsl{} \]
\item sliding $3$-cells $s_n^0$ and $s_n^1$ 
 and ordering $3$-cells $o_n$ defined by induction in \cite{BCNR14}, and oriented in the same way than in \cite{AL16}.
\end{enumerate}
\end{enumerate}

We easily prove following \cite{SAV18} that this linear~$(3,2)$-polygraph is a presentation of $\aob$. To study this linear~$(3,2)$-polygraph modulo, we consider its convergent subpolygraph $E$ defined by
$E_i = \cl{\aob}_i$ for $i=0,1$, $E_2$ contains the last six generating $2$-cells in \ref{E:GenCellsAOB} and $E_3$ contains exactly the isotopy $3$-cells (\ref{E:IsotopyCells}). Following \ref{SSS:PolygraphIsotopy}, $E$ is convergent. We denote by $R$ the linear~$(3,2)$-polygraph having the same $i$-cells than $\cl{\aob}$ for $i=0,1,2$ and such that $R_3 = \cl{\aob}_3 \backslash E_3$. From the data of $E$ and $R$, we can then consider the linear~$(3,2)$-polygraph $(R,E,\ER)$, and prove the following result:
\begin{theorem}
\label{T:QuasiTermConfModAOB}
Let $(R,E)$ be the splitting of $\aob$ defined above, then $\ER$ is quasi-terminating and $R$ is confluent modulo $E$.
\end{theorem}

\subsection{Quasi termination of $\ER$}
\label{SS:Quasi-TermAOB}
To prove quasi-termination of the linear~$(3,2)$-polygraph $\ER$ is quasi-terminating, we will proceed in two steps: at first we will prove that the linear~$(3,2)$-polygraph $R$ minus the sliding $3$-cells is terminating using derivations as in \ref{SSS:TerminationByDerivation}. Then, using a notion of quasi-ordering and a suited notion of polynomial interpretation on $\aob_2^\ell$, we will describe in the same fashion than in \cite{AL16} a procedure proving that every $2$-cell in $\aob$ can be rewritten in a finite number of steps into a monomial on which the only $3$-cells that can be applied are the cells creating cycles.

\subsubsection{Termination without sliding $3$-cells} 
In order to prove that the linear~$(3,2)$-polygraph $\ER$ is quasi-terminating, let us at first state the following lemma:
\begin{lemma}
\label{L:TerminationWithoutSliding}
The linear~$(3,2)$-polygraph $R' = R \; \backslash \{ s_n^0, s_n^1 \}_{n \in \N}$ is terminating.
\end{lemma}

\begin{proof}
Let us proceed in three steps, using the derivation method given in \ref{SS:TerminationDerivation}. We at first consider a derivation $d$  defined by $d(u) = || u ||_{\{ \raisebox{-3mm}{$\smcrossingup{}{}$} , \raisebox{-3mm}{$\smcrossingdn{}{}$} \}}$ into the trivial modulo $M_{\ast, \ast, \Z}$, counting the number of crossing generators in a given $2$-cell. We have that $d ( s_2(\omega)  ) > d( \omega_i )$ for any $3$-cell $\omega$ in $\{  A,B,C,D,E,F, \alpha \}$ and any $\omega_i$ in $\text{Supp}(t( \omega))$. As a consequence, one gets that if the linear~$(3,2)$-polygraph $R''$ defined as $R'$ minus each of these $3$-cell terminates, then so does $R'$. Indeed, otherwise there would exist an infinite reduction sequence $(f_n)_{n \in \N}$ in $R'$ and thus, an infinite decreasing sequence $(d(f_n))_{n \in \N}$ of natural numbers. Moreover, this sequence would be strictly decreasing at each step that is generated by any of these $3$-cells and thus, after some natural number $p$, this sequence would be generated by the other $3$-cells only. This would yield an infinite reduction sequence $(f_n)_{n \geq p}$ in $R''$, which is impossible by assumption.

It remains to prove that the linear~$(3,2)$-polygraph $(R_0,R_1,R_2, \{ \beta_{\pm}, \gamma_{l,\pm}, \gamma_{r, \pm}, o_n \}_{n \in \N})$ terminates. We can still reduce this problem to the termination of the rules $\beta_{\pm}, \gamma_{l,\pm}$ and $\gamma_{r, \pm}$ by considering a derivation $d'$ with values in the trivial modulo $M_{\ast, \ast, \Z}$ counting the number of clockwise oriented bubbles. To prove this, let us consider $X$ the $2$-functor $X: \aob_2^* \to \textbf{Ord}$ on generating $2$-cells by: 
$$X ( \identsm{} ) (i)=i \hspace{1cm} X ( \identdotsm{}  ) (i)= i+1 \hspace{1cm} X ( \raisebox{-5mm}{$\scalebox{0.6}{\crossing{}{}}$} ) (i,j) = (j+1,i ) \quad \forall i,j \in \N$$
for both orientations of strands, and we consider the $\aob_2^\ast$-module $M_{X,*,\Z}$ and define the derivation $d: \aob_2^* \fl M_{X,*,\Z}$ on the generating $2$-cells by \[ d ( \raisebox{-4mm}{$\scalebox{0.6}{\crossing{}{}}$} ) (i,j) = i \hspace{1cm} d ( \identsm{} ) (i) = 0 = d ( \identdotsm{} ) (i) \]
One then checks that the following inequalities hold:
\[  d ( \raisebox{-5mm}{$\scalebox{0.6}{\ybg{}{}{}}$} ) (i,j,k) = 2i+j+1 > 2i+j =  d ( \raisebox{-5mm}{$\scalebox{0.6}{\ybd{}{}{}}$} )   (i,j,k) \]
\[ d ( \raisebox{-5mm}{$\scalebox{0.7}{\dcrossul{}{} }$} ) (i,j) = i > 0 =  d (\raisebox{-5mm}{$\scalebox{0.7}{\dcrossdr{}{}}$} ) (i,j), \quad
 d ( \raisebox{-5mm}{$\scalebox{0.7}{\dcrossur{}{}}$} ) (i,j) =  i > 0 = d ( \raisebox{-5mm}{$\scalebox{0.7}{\dcrossdl{}{} }$} ) (i,j)  \]

so that the $2$-functor $X$ and the derivation $d$ satisfy the conditions ${\bf i)}$, ${\bf ii)}$ and ${\bf iii)}$ of \ref{SSS:TerminationByDerivation}, and thus the corresponding linear~$(3,2)$-polygraph is terminating.
\end{proof}

However, as explained in \cite{AL16}, the addition of the sliding $3$-cells create rewriting cycles, so that $R$ is not terminating. Nethertheless, we will prove that it is quasi-terminating.

\subsubsection{Quasi-orderings}
Following \cite{DER87}, a quasi-ordered set is a set $A$ equipped with a transitive and reflexive binary relation $\qord$ on elements of $A$. For example, for any abstract rewriting system $(A,\fl_R)$, the derivability relation $\fl_R^\ast$ is a quasi-ordering on the set $A$. Given a quasi-ordering $\qord$ on a set $A$, we define the associated equivalence relation $\approx$ as both $\qord$ and $\lesssim$ and the strict partial ordering $>$ as $\qord$ but not $\lesssim$. Such a quasi-ordering is said \emph{total} if for any $a$,$b$ in $A$, we have either $a \qord b$ or $b \qord a$. The strict part $>$ of a quasi-ordering is \emph{well-founded} if and only if all infinite quasi-descending sequences 
\[ a_1 \qord a_2 \qord \dots \]
of elements of $A$ contains a pair $s_j \lesssim s_k$ for $j < k$. A quasi-ordering defined on a set of $2$-cells of a linear~$(2,2)$-category $\mathcal{C}$ is said \emph{monotonic} if 
\[ ( u \qord v ) \; \Rightarrow  (C[u] \qord C[v]) \]
for any context $C$ of $\mathcal{C}$. From \cite{DER87}, if $\qord$ is monotonic then $\approx$ is a congruence. Many termination and quasi-termination proofs in the literature are made using well-founded quasi-orderings defined by monotonic polynomial interpretations, \cite{LAN79}. In the case of linear~$(2,2)$-categories, these polynomial interpretations will be given by weight functions.

\subsubsection{Weight functions}
\label{SSS:WeightFunctions}
Let $\mathcal{C}$ a linear~$(2,2)$-category. A \emph{weight function} on $\mathcal{C}$ is a function $\tau$ from $\mathcal{C}_2$ to $\N$ such that
\begin{enumerate}[{\bf i)}]
\item $\tau (u \star_i v) = \tau (u) + \tau (v)$ for $i=0,1$ for any $i$-composable $2$-cells $u$ and $v$,
\item for each $2$-cell $u$ in $\mathcal{C}_2$,
\[ \tau (u) = \text{max} \{ \tau(u_i) \; | \; u_i \in \text{Supp}(u) \} \]
\end{enumerate}
Note that when $\mathcal{C}$ is presented by a linear~$(3,2)$-polygraph $P$, such a weight function is uniquely and entirely determined by its values on the generating $2$-cells of $P_2$. This enables to define a quasi-ordering $\qord$ on $\cl{\aob}_2^\ell$ by $ u \qord v$ if $\tau(u) \geq \tau(v)$.

\subsubsection{Quasi-reduced monomials}
\label{SSS:Quasi-Red}
Recall from \cite{AL16} that a monomial in $\aob$ is \emph{quasi-reduced} if it can be rewritten by only one of the $3$-cells derived from ordering and sliding $3$-cells in $\ER$ on the following subdiagrams:
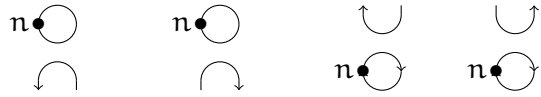
\begin{gather*}
\begin{array}{c}
\tikz[scale=0.5]{
\draw[color=black] (1,0.5) arc (0:180:0.5);
\draw[->,color=black] (0,0.5) -- (0,0.25);
\draw[color=black] (1,0.5) -- (1,0.25);
\draw[color=black] (0,2) arc (180:0:0.5);
\draw[color=black] (1,2) arc (360:180:0.5);
\node at (0,2) {$\bullet$};
\node at (-0.5,2) {$n$};
} \end{array} \qquad \begin{array}{c}
\tikz[scale=0.5]{
\draw[color=black] (1,0.5) arc (0:180:0.5);
\draw[color=black] (0,0.5) -- (0,0.25);
\draw[->,color=black] (1,0.5) -- (1,0.25);
\draw[color=black] (0,2) arc (180:0:0.5);
\draw[color=black] (1,2) arc (360:180:0.5);
\node at (0,2) {$\bullet$};
\node at (-0.5,2) {$n$};
} \end{array} \qquad
\begin{array}{c}
\tikz[scale=0.5]{
\draw[color=black] (1,3.5) arc (360:180:0.5);
\draw[color=black, ->] (0,3.5) -- (0,3.75);
\draw[color=black] (1,3.5) -- (1,3.75);
\draw[color=black, ->] (0,2) arc (180:0:0.5);
\draw[color=black, ->] (1,2) arc (360:180:0.5);
\node at (0,2) {$\bullet$};
\node at (-0.5,2) {$n$};
} \end{array} \quad \begin{array}{c}
\tikz[scale=0.5]{
\draw[color=black] (1,3.5) arc (360:180:0.5);
\draw[color=black] (0,3.5) -- (0,3.75);
\draw[color=black, ->] (1,3.5) -- (1,3.75);
\draw[color=black, ->] (0,2) arc (180:0:0.5);
\draw[color=black, ->] (1,2) arc (360:180:0.5);
\node at (0,2) {$\bullet$};
\node at (-0.5,2) {$n$};
} \end{array}
\end{gather*}
for any $n$ in $\N$. We call a 2-cell of $\aob_2^\ell$ quasi-reduced if all monomials in its monomial decomposition are quasi-reduced.

\subsubsection{Quasi-termination of $\ER$}
Following \cite{AL16}, we define a weight function on $\cl{\aob_2^\ell}$ by its following values on generating $2$-cells:
\vspace{-0.5cm}
\[ \tau (\: \raisebox{-1mm}{$\caplsl{}$} ) = \tau (\caprsl{} \: ) = \tau (  \cuprsl{} \: ) = \tau (\: \raisebox{-1mm}{$\cuplsl{}$} ) = 0, \; \tau ( \raisebox{2mm}{$\udottsl{}$} ) = \tau ( \raisebox{2mm}{$\ddottsl{}$} ) = 0, \;
\tau ( \raisebox{-3mm}{$\scalebox{0.6}{\crossingup{}{}}$} ) = \tau ( \raisebox{-3mm}{$\scalebox{0.6}{\crossingdn{}{}}$} ) = 3.
 \]
 Note that for any $3$-cell $\alpha$ in $E_3$, we have $\tau (s_2 (\alpha)) = \tau (t_2 (\alpha))$ so that the isotopy $3$-cells preserve this weight function. Then, starting with a monomial $u$ of $\cl{\aob}_2^\ell$:
 \begin{itemize}
 \item[-] While $u$ can be rewritten with respect to $\ER$ into a $2$-cell $u'$ such that
 $\tau(u') < \tau(u)$, then assign $u$ to $u'$.
\item[-] While $u$ can be rewritten with respect to $\ER$ into a $2$-cell $u'$ without any of the $3$-cells depicted in \ref{SSS:Quasi-Red}, then assign $u$ to $u'$.
 \end{itemize}
 From Lemma \ref{L:TerminationWithoutSliding} and well-foundedness of the quasi-ordering $\qord$ defined in \ref{SSS:WeightFunctions}, this procedure terminates and returns a linear combination of monomials in $\cl{\aob}_2^\ell$ which are quasi-reduced.

\subsection{Confluence modulo}
\label{SS:ConfluenceModuloAOB}
In this subsection, we prove that the linear~$(3,2)$-polygraph modulo $\ER$ is confluent modulo $E$ using Theorem \ref{T:ConfluenceByDecreasingness} and Proposition \ref{P:DecreasingnessFromCriticalPairs}.

\subsubsection{Critical branchings modulo} Let us at first enumerate the list of all critical branchings modulo that we have to prove decreasing with respect to $\psiqnf$. First of all, there are $6$ regular critical branchings implying the degenerate affine Hecke $3$-cells:
\[ (\alpha_{\pm}, \quad \alpha_{\pm}), \quad (\alpha_{\pm},\beta_{\pm}), \quad (\beta_{\pm},\alpha_{\pm}), \quad (\alpha_{\pm}, \gamma_{\eta, \pm})_{\eta \in \{ l,r \}}, \quad (\beta_{\pm}\gamma_{\eta, \pm})_{\eta \in \{ l,r \}}, \quad (\gamma_{l,\pm}, \gamma_{r,\pm}). \]

The first three families are proved confluent modulo in the same way that the polygraph of permutations is proved confluent in \cite{GM09}. The remaining critical branchings are decreasingly confluent as follows:
               
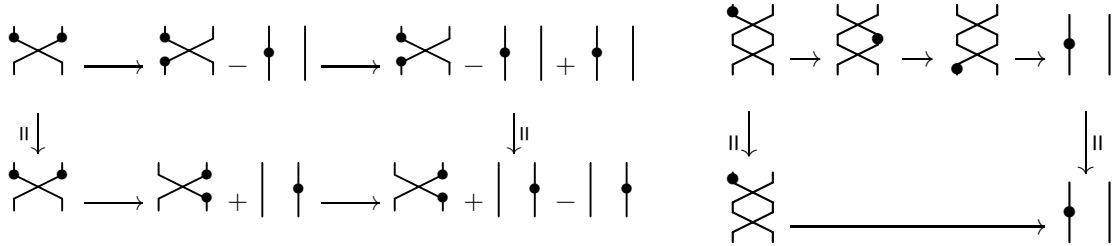
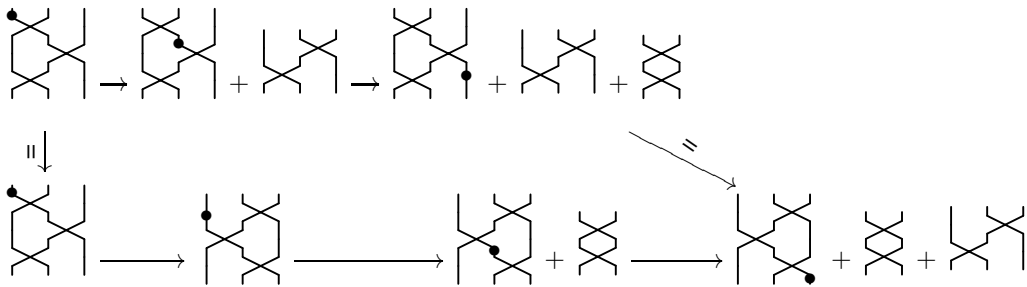
\[ \scalebox{0.9}{\xymatrix@R=1.5em{ \raisebox{-5mm}{$\ducross{}{}$} \ar@1 [r] \ar@1 [d] _-{\fleq} & \raisebox{-5mm}{$\dcrossudl{}{}$} - \raisebox{-5mm}{$\didl{}{}$} \ar@1[r] & \raisebox{-5mm}{$\dcrossudl{}{}$} - \raisebox{-5mm}{$\didl{}{}$} + \raisebox{-5mm}{$\didl{}{}$} \ar@1 [d] ^-{\fleq} \\
\raisebox{-5mm}{$\ducross{}{}$} \ar [r] &  \raisebox{-5mm}{$\dcrossudr{}{}$} + \raisebox{-5mm}{$\didr{}{}$} \ar@1[r] &  \raisebox{-5mm}{$\dcrossudr{}{}$} + \raisebox{-5mm}{$\didr{}{}$} - \raisebox{-5mm}{$\didr{}{}$} }} \qquad \xymatrix@R=1.5em@C=1em{ \raisebox{-5mm}{$\tcrossul{}{}$} \ar@1 [r] \ar@1 [d] _{\fleq} & \raisebox{-5mm}{$\tcrossmr{}{}$} \ar@1[r] &  \raisebox{-5mm}{$\tcrossdl{}{}$} \ar@1[r] & \raisebox{-5mm}{$\didl{}{}$} \ar@1 [d] ^-{\fleq}  \\
              \raisebox{-5mm}{$\tcrossul{}{}$}  \ar@1[rrr] &   &    &  \raisebox{-5mm}{$\didl{}{}$} }  \]

\[ 
\scalebox{0.9}{\xymatrix@R=1.5em@C=1em{ \raisebox{-5mm}{$\ybgul{}{}{}$}  \ar@1 [r] \ar@1 [d] _-{\fleq} & \raisebox{-5mm}{$\ybgmmm{}{}{}$} + \raisebox{-5mm}{$\cdcg{}{}{}$}  \ar@1 [r] & \raisebox{-5mm}{$\ybgdr{}{}{}$} + \raisebox{-5mm}{$\cdcg{}{}{}$} + \raisebox{-5mm}{$\tcross{}{}$} \ar@1 [dr] ^-{\rotatebox{145}{=}} \\
\raisebox{-5mm}{$\ybgul{}{}{}$} \ar@1 [r] & \raisebox{-5mm}{$\ybdul{}{}{}$} \ar@1 [r] & \raisebox{-5mm}{$\ybdmmm{}{}{}$} + \raisebox{-5mm}{$\tcross{}{}$} \ar@1 [r] & \raisebox{-5mm}{$\ybddr{}{}{}$} + \raisebox{-5mm}{$\tcross{}{}$} + \raisebox{-5mm}{$\cdcg{}{}{}$} }}
  \]
for both orientations of strands. In the last two cases, we proceed similarly if the dot is placed on another strand. Following the study of the $3$-polygraphs of permutations in \cite{GM09}, there also are two families of indexed critical branchings given by 
\begin{align} \label{indexyb}
\indexybsl{k}
\end{align}
where the normal forms $k$ that we can plug in (\ref{indexyb}) are of the form
\begin{enumerate}[{\bf i)}]
\item $\identdots{}{n}$ for every $n \in N$, which is just an identity if $n=0$.
\item \raisebox{-6mm}{$\scalebox{1}{\dcrossdldot{}{}{n}}$}  \: for all $n \in \N$.
\end{enumerate}
These indexed critical branchings are confluent modulo $E$, and the proof of their confluence is not drawn here but can be found in a forecoming paper, in the more general combinatorics of KLR algebras. The critical branchings modulo implying the sliding and ordering $3$-cells are proved confluent modulo $E$ in a similar fashion than in \cite{AL16}.
We then give the exhaustive list of all critical branchings modulo implying the $3$-cells $A$,$B$,$C$,$D$,$E$ and $F$. First of all, these branchings overlap with degenerate affine Hecke relations to give the following sources of critical branchings modulo:
\[ (A,C), \quad (B,D), \quad  (B,F), \quad  (E,D), \quad (C,E), \quad (E,F), \quad (F,E), \]
\[ (A, \gamma_{l,+}), \quad (B, i_4^2 , \gamma_{l,+}), \quad (D, \gamma_{r,+}), \quad (E , \gamma_{l,+}), \quad (\gamma_{r,+} , i_3^2 , C), \quad (F , \gamma_{r,+}), \]
\[ (\alpha_+ , i_1^0 \star_2 i_4^0 , F), \quad (\gamma_{r,+}, i_1^0 \star_2 i_4^0 \star_2 (i_3^2 \star_2 i_1^2)^-, F), \quad (\beta_+, i_1^0 \star_2 i_4^0, F), \]
\[ (\alpha_+, i_1^0 \star_2 i_4^0, E), \quad (\gamma_{l,+}, i_1^0 \star_2 i_4^0 \star_2 (i_2^2 \star_2 i_4^2)^- , E). \]

Some of these branchings are proved decreasingly confluent with respect to $\psiqnf$ by the confluence modulo diagrams below. The remaining one are obtained by symmetries of the diagrams and are thus not drawn.
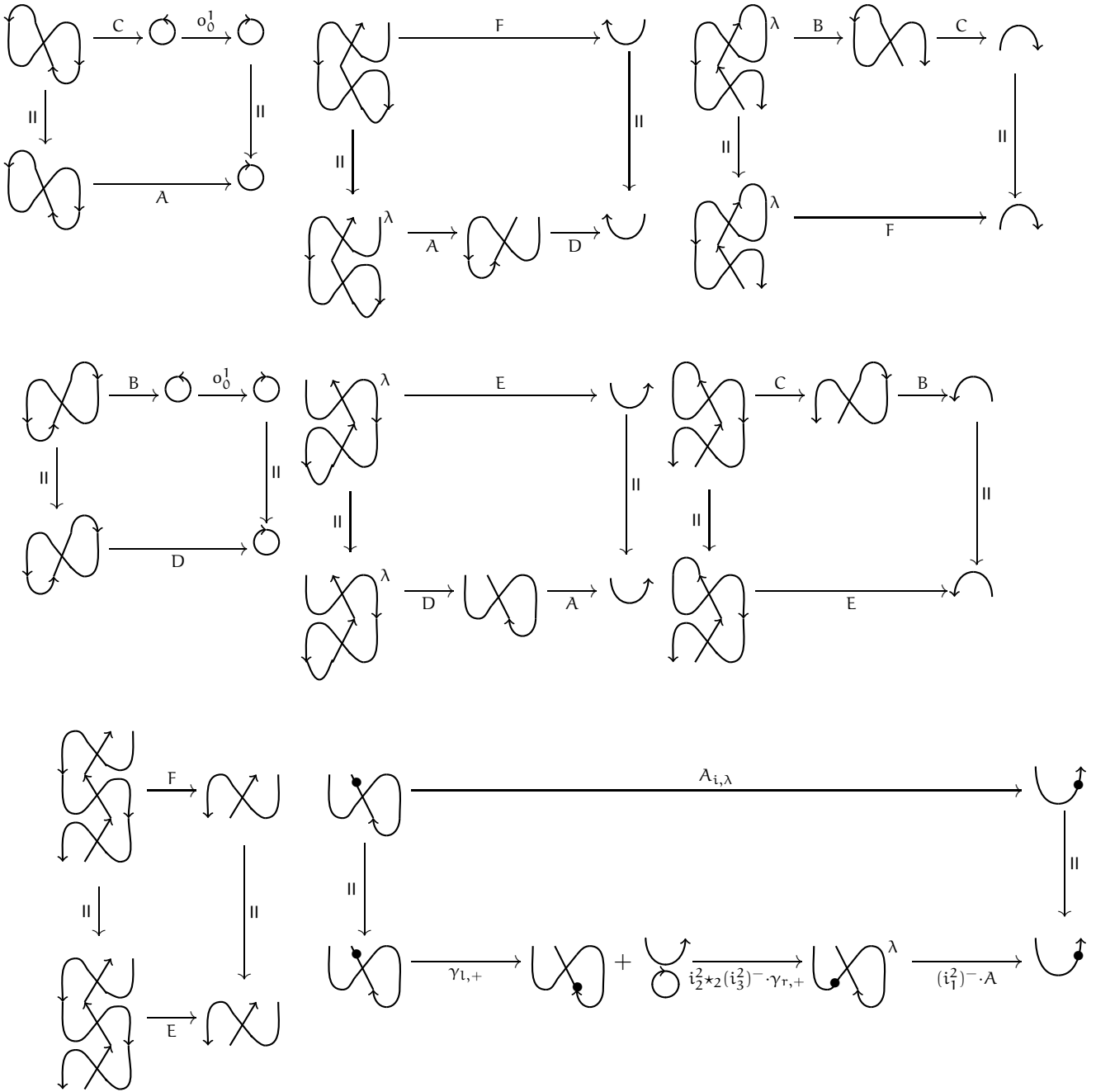
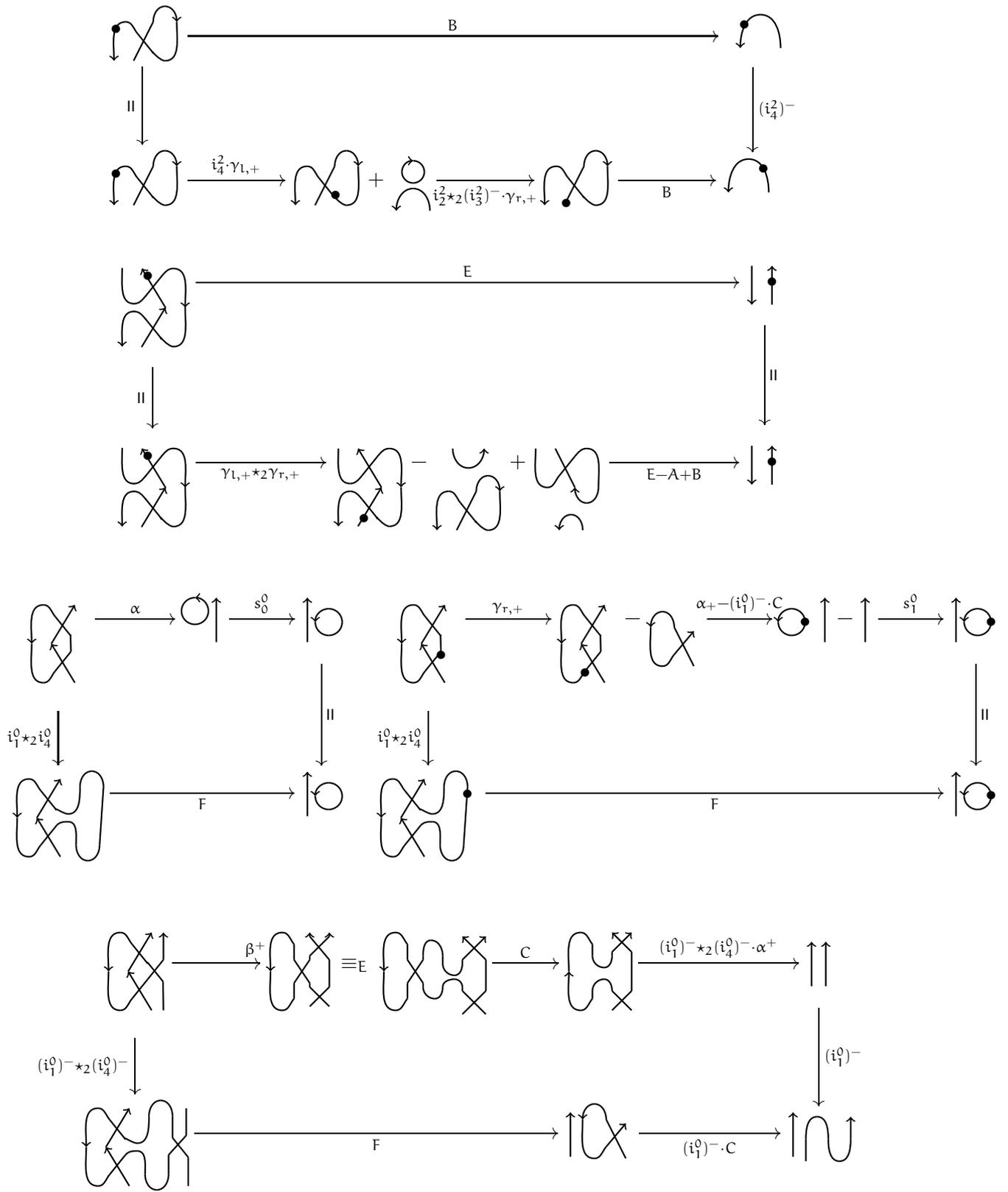
\[ \hspace{-1cm} 
\xymatrix@C=2em{ \bcritacsl{} \ar@1 [d] _-{\fleq}
  \ar@1 [r] ^-{C}
 & \negbubsdsl{} \ar@1[r] ^-{o_0^1} & \posbubsdsl{} \ar@1 [d] ^-{\fleq} \\
  \bcritacsl{} \ar@1 [rr] _-{A} & & \posbubsdsl{} }
  \quad
\xymatrix@R=2.5em@C=2em{ \bcritafsl{} 
\ar@1 [rr] ^-{F}
\ar@1 [d] _-{\fleq} & & \cuplsl{} \ar@1[d] ^-{\fleq} \\
\bcritaf{} \ar@1 [r] _-{A} & \tfishdlsl{} \ar@1 [r] _-{D} & \cuplsl{} }  
\quad 
 \xymatrix@R=2.5em@C=2em{
\bcritbfsl{} \ar@1 [r] ^-{B} \ar @1 [d] _-{\fleq} & \tfishursl{}  \ar @1[r] ^-{C} & \caprsl{} \ar @1 [d] _-{\fleq} \\
\bcritbfsl{} \ar@1 [rr] _-{F} & &  \caprsl{} }
\] \[
\hspace{-1.5cm} 
 \xymatrix@R=2.5em@C=2em{
\bcritbdsl{} \ar@1 [r] ^-{B}
            \ar@1 [d] _-{\fleq} & \negbubsdsl{} \ar@1[r] ^-{o_0^1} & \posbubsdsl{} \ar@1 [d] ^-{\fleq} \\
\bcritbdsl{} \ar@1 [rr] _-{D} & & \posbubsdsl{} } 
\;
\xymatrix@R=2.5em@C=2em{
\bcritdesl{} \ar@1 [rr] ^-{E}
            \ar@1 [d] _-{\fleq} & & \cuprsl{} \ar [d] ^-{\fleq} \\
\bcritdesl{} \ar@1 [r] _-{D} & \tfishdrsl{} \ar@1 [r] _-{A} & \cuprsl{} } \;
\xymatrix@R=2.5em@C=2em{
\bcritcesl{} \ar@1 [r] ^-{C} 
            \ar @1 [d] _-{\fleq} & \tfishulsl{}  
            \ar @1 [r] ^-{B} & \caplsl{} \ar@1 [d] ^-{\fleq} \\
\bcritcesl{} \ar@1 [rr] _-{E}  & &   \caplsl{} } \] \[
\xymatrix@R=2em@C=2em{
\raisebox{7mm}{$\bcritfesl{}{}$} \ar @1 [r] ^-{F}
               \ar@1 [d] _-{\fleq} &  \tleftcrosssl{}{} \ar@1 [d] ^-{\fleq} \\
               \raisebox{7mm}{$\bcritfesl{}{}$} \ar @1 [r] _-{E} &  \tleftcrosssl{}{} }
\quad 
\xymatrix@C=4.5em@R=3em{ \cbadotsl{} 
  \ar@1[rrr] ^-{A_{i,\lambda}}
  \ar@1 [d] _-{\rotatebox{90}{=}} & & & \cuprdrsl{}{} \ar@1 [d] ^-{\fleq}  \\
  \cbadotsl{} \ar@1 [r] _-{\gamma_{l,+}} & \cbadotbsl{} + \ruleawdotsl{} \ar@1[r] _-{i_2^2 \star_2 (i_3^2)^- \cdot \gamma_{r,+}} & \cbadotcsl{}  \ar@1[r] _-{(i_1^2)^- \cdot A} & \cuprdrsl{}{} 
  } \]
\[
\xymatrix@C=4.5em@R=3em{ \cbbdotsl{} 
  \ar@1[rrr] ^-{B}
  \ar@1 [d] _-{\rotatebox{90}{=}} & & & \capldlsl{}{} \ar@1 [d] ^-{(i_4^2)^-} \\
 \cbbdotsl{} \ar@1 [r] ^-{i_4^2 \cdot \gamma_{l,+}} & \cbbdotbsl{} + \rulebwdotsl{} \ar@1 [r] _-{i_2^2 \star_2 (i_3^2)^- \cdot \gamma_{r,+}} & \cbbdotcsl{} \ar@1[r] _-{B} & \capldrsl{}{} }   \]

 \[ \xymatrix@C=6em@R=2.5em{ 
\tdcrossrldtestsl{}{} \ar@1 [rr] ^-{E}
                       \ar@1 [d] _-{\fleq}   & &  \diddownupdlsl{}{} \ar @1 [d] ^-{\fleq} \\
\tdcrossrldtestsl{}{} \ar@1 [r] _-{\gamma_{l,+} \star_2 \gamma_{r,+}} & \tdcrossrlddsl{}{} - \cuprbsl{} + \acaplsl{} \ar@1 [r] _-{E-A+B} &  \diddownupdlsl{}{} } \]

\[ \xymatrix@R=2.5em@C=3em{
 \isocbasl{}{} \ar@1[r] ^-{\alpha} \ar@1 [d] _-{i_1^0 \star_2 i_4^0} & \begin{tikzpicture}[baseline = 0, scale = 1.2]
  \draw[-,thick,black] (0,0.4) to[out=180,in=90] (-.2,0.2);
  \draw[->,thick,black] (0.2,0.2) to[out=90,in=0] (0,.4);
 \draw[-,thick,black] (-.2,0.2) to[out=-90,in=180] (0,0);
  \draw[-,thick,black] (0,0) to[out=0,in=-90] (0.2,0.2);
 \end{tikzpicture} \identusl{} \ar@1 [r] ^-{s_0^0} & \identusl{} \raisebox{-2mm}{$\begin{tikzpicture}[baseline = 0, scale=1.2]
  \draw[->,thick,black] (0,0.4) to[out=180,in=90] (-.2,0.2);
  \draw[-,thick,black] (0.2,0.2) to[out=90,in=0] (0,.4);
 \draw[-,thick,black] (-.2,0.2) to[out=-90,in=180] (0,0);
  \draw[-,thick,black] (0,0) to[out=0,in=-90] (0.2,0.2);
 \end{tikzpicture}$} \ar@1 [d] ^-{\fleq} \\
 \isocbbsl{}{} \ar@1 [rr] _-{F} & & \identusl{} \raisebox{-2mm}{$\begin{tikzpicture}[baseline = 0, scale=1.2]
  \draw[->,thick,black] (0,0.4) to[out=180,in=90] (-.2,0.2);
  \draw[-,thick,black] (0.2,0.2) to[out=90,in=0] (0,.4);
 \draw[-,thick,black] (-.2,0.2) to[out=-90,in=180] (0,0);
  \draw[-,thick,black] (0,0) to[out=0,in=-90] (0.2,0.2);
 \end{tikzpicture}$} } 
 \quad
 \xymatrix@R=2.5em@C=3em{\isocbadotsl{}{} \ar@1 [d] _-{i_1^0 \star_2 i_4^0} \ar@1 [r] ^-{\gamma_{r,+}} & \isocbadotbbsl{}{} - \raisebox{-4mm}{$\isocbadotbcsl{}$} \ar@1 [r] ^-{\alpha_+ - (i_1^0)^- \cdot C } &   \raisebox{-2mm}{$\begin{tikzpicture}[baseline = 0, scale=1.2]
  \draw[->,thick,black] (0,0.4) to[out=180,in=90] (-.2,0.2);
  \draw[-,thick,black] (0.2,0.2) to[out=90,in=0] (0,.4);
 \draw[-,thick,black] (-.2,0.2) to[out=-90,in=180] (0,0);
  \draw[-,thick,black] (0,0) to[out=0,in=-90] (0.2,0.2);
  \node at (0.2,0.2) {$\bullet$};
 \end{tikzpicture}$} \identusl{} - \identusl{} \ar@1 [r] ^-{s_1 ^0} & \identusl{} \raisebox{-2mm}{$\begin{tikzpicture}[baseline = 0, scale=1.2]
  \draw[->,thick,black] (0,0.4) to[out=180,in=90] (-.2,0.2);
  \draw[-,thick,black] (0.2,0.2) to[out=90,in=0] (0,.4);
 \draw[-,thick,black] (-.2,0.2) to[out=-90,in=180] (0,0);
  \draw[-,thick,black] (0,0) to[out=0,in=-90] (0.2,0.2);
  \node at (0.2,0.2) {$\bullet$};
 \end{tikzpicture}$} \ar@1 [d] ^-{\fleq} \\ \isocbbdotsl{}{} \ar@1 [rrr] _-{F}  & & & \identusl{} \raisebox{-2mm}{$\begin{tikzpicture}[baseline = 0, scale=1.2]
  \draw[->,thick,black] (0,0.4) to[out=180,in=90] (-.2,0.2);
  \draw[-,thick,black] (0.2,0.2) to[out=90,in=0] (0,.4);
 \draw[-,thick,black] (-.2,0.2) to[out=-90,in=180] (0,0);
  \draw[-,thick,black] (0,0) to[out=0,in=-90] (0.2,0.2);
  \node at (0.2,0.2) {$\bullet$};
 \end{tikzpicture}$} } \]
 
\[ \xymatrix@R=2.5em@C=3em{
\raisebox{3mm}{$\isocbacrsl{}{}$} \ar [d] _-{(i_1^0)^- \star_2 (i_4^0)^-}
\ar [r] ^{\beta^+} & \raisebox{3mm}{$\isocbacrbsl{}{}$} \: \equiv_E \: \raisebox{3mm}{$\isocbbcrbsl{}{}$} \ar [r] ^-{C}  & \raisebox{3mm}{$\isocbbcrcsl{}{}$} \ar [rr] ^-{(i_1^0)^- \star_2 (i_4^0)^- \cdot \alpha^+} &  & \identusl{} \identusl{} \ar [d] ^-{(i_1^0)^-} \\
\raisebox{2mm}{$\isocbbcrsl{}{}$} \ar [rr] _-{F} & & \identusl{} \isocbadotbcsl{} \ar [rr] _-{(i_1^0)^- \cdot C} & & \identusl{} \: \begin{tikzpicture}[baseline = 0]
  \draw[->,thick,black] (0.3,0) to (0.3,.4);
	\draw[-,thick,black] (0.3,0) to[out=-90, in=0] (0.1,-0.4);
	\draw[-,thick,black] (0.1,-0.4) to[out = 180, in = -90] (-0.1,0);
	\draw[-,thick,black] (-0.1,0) to[out=90, in=0] (-0.3,0.4);
	\draw[-,thick,black] (-0.3,0.4) to[out = 180, in =90] (-0.5,0);
  \draw[-,thick,black] (-0.5,0) to (-0.5,-.4);
\end{tikzpicture}
}
\]
 
\subsubsection{Normally ordered Brauer diagrams}
A \emph{dotted oriented Brauer diagram} is a planar string diagram built from $\star_0$ and $\star_1$-compositions of the above generating $2$-cells in which every edge is oriented and is either a bubble or have a boundary point as source and target, each edge is decorated with an arbitrary number of dots not allowed to pass through the crossings. Such a diagram is said \emph{normally ordered} if all its bubbles are clockwise oriented and located in the leftmost region, and if all dots are either on a bubble or a segment pointing toward a boundary (or in the opposite direction). In a similar fashion than \cite[Lemma 5.2.6]{AL16}, we prove that each $2$-cell of $\aob_2^\ell$ can be rewritten with respect to $\ER$ into a linear combination of diagrams whose normal forms with respect to $E$ are normally ordered dotted oriented Brauer diagrams. As a consequence, we get from \ref{T:BasisByConfluenceModulob} that the set of such diagrams with $1$-source $u$ and $1$-target $v$ form a basis of the $\K$-vector space $\aob_2 (u,v)$, and we recover the result from \cite{BCNR14,AL16}.

%

\begin{small}
\renewcommand{\refname}{\Large\textsc{References}}
\bibliographystyle{plain}
\bibliography{Bibliography}

\begin{thebibliography}{10}

\bibitem{AL16}
Cl\'ement {Alleaume}.
\newblock {Rewriting in higher dimensional linear categories and application to
  the affine oriented Brauer category}.
\newblock {\em arXiv:1603.02592}, March 2016.

\bibitem{ALPhD}
Clement Alleaume.
\newblock {\em {Higher-dimensional linear rewriting and coherence in
  categorification and representation theory}}.
\newblock Theses, {Universit{\'e} de Lyon}, June 2018.

\bibitem{AM17}
Cl{\'e}ment Alleaume and Philippe Malbos.
\newblock {Coherence of string rewriting systems by decreasingness}.
\newblock preprint, arXiv:1612.09193, 2016.

\bibitem{Brauer37}
Richard Brauer.
\newblock On algebras which are connected with the semisimple continuous
  groups.
\newblock {\em Annals of Mathematics}, 38(4):857--872, 1937.

\bibitem{BRU18}
Jonathan Brundan.
\newblock On the definition of heisenberg category.
\newblock {\em Algebraic Combinatorics}, 1(4):523--544, 2018.

\bibitem{BUR93}
Albert Burroni.
\newblock Higher-dimensional word problems with applications to equational
  logic.
\newblock {\em Theoret. Comput. Sci.}, 115(1):43--62, 1993.
\newblock 4th Summer Conference on Category Theory and Computer Science (Paris,
  1991).

\bibitem{CKM14}
Sabin Cautis, Joel Kamnitzer, and Scott Morrison.
\newblock Webs and quantum skew howe duality.
\newblock {\em Mathematische Annalen}, 360(1):351--390, Oct 2014.

\bibitem{CKS00}
J.~Robin~B. Cockett, Jurgen Koslowski, and Robert A.~G. Seely.
\newblock Introduction to linear bicategories.
\newblock {\em Math. Structures Comput. Sci.}, 10(2):165--203, 2000.
\newblock The Lambek Festschrift: mathematical structures in computer science
  (Montreal, QC, 1997).

\bibitem{DER87}
Nachum Dershowitz.
\newblock Termination of rewriting.
\newblock {\em J. Symb. Comput.}, 3(1-2):69--115, February 1987.

\bibitem{DM79}
Nachum Dershowitz and Zohar Manna.
\newblock Proving termination with multiset orderings.
\newblock In {\em Automata, languages and programming ({S}ixth {C}olloq.,
  {G}raz, 1979)}, volume~71 of {\em Lecture Notes in Comput. Sci.}, pages
  188--202. Springer, Berlin-New York, 1979.

\bibitem{DMpp18}
Benjamin Dupont and Philippe Malbos.
\newblock {Coherent confluence modulo relations and double groupoids}.
\newblock preprint arXiv:1810.08184, Hal-01898868, submitted, 2018.

\bibitem{DM19}
Benjamin Dupont and Philippe Malbos.
\newblock {Algebraic polygraphs modulo and linear rewriting}.
\newblock in preparation, 2019.

\bibitem{ELI15}
Ben Elias.
\newblock {Light ladders and clasp conjectures}.
\newblock preprint, arXiv:1510.06840, 2015.

\bibitem{GUI06}
Yves Guiraud.
\newblock Termination orders for three-dimensional rewriting.
\newblock {\em J. Pure Appl. Algebra}, 207(2):341--371, 2006.

\bibitem{GHM17}
Yves Guiraud, Eric Hoffbeck, and Philippe Malbos.
\newblock {Convergent presentations and polygraphic resolutions of associative
  algebras}.
\newblock {\em arxiv:1406.0815, to appear in Math. Z.}, 2019.

\bibitem{GM09}
Yves Guiraud and Philippe Malbos.
\newblock Higher-dimensional categories with finite derivation type.
\newblock {\em Theory Appl. Categ.}, 22:No. 18, 420--478, 2009.

\bibitem{GM12}
Yves Guiraud and Philippe Malbos.
\newblock Higher-dimensional normalisation strategies for acyclicity.
\newblock {\em Adv. Math.}, 231(3-4):2294--2351, 2012.

\bibitem{GM18}
Yves Guiraud and Philippe Malbos.
\newblock Polygraphs of finite derivation type.
\newblock {\em Math. Structures Comput. Sci.}, 28(2):155--201, 2018.

\bibitem{Huet80}
G{\'e}rard Huet.
\newblock Confluent reductions: abstract properties and applications to term
  rewriting systems.
\newblock {\em J. Assoc. Comput. Mach.}, 27(4):797--821, 1980.

\bibitem{BCNR14}
{Jonathan Brunda, Jonathan Comes, David Nash and Andrew Reynolds}.
\newblock {A basis theorem for the affine oriented Brauer category and its
  cyclotomic quotients}.
\newblock {\em arXiv:1404.6574}, April 2014.

\bibitem{Jones99}
V.~F.~R. {Jones}.
\newblock {Planar algebras, I}.
\newblock {\em arXiv:math/9909027}, September 1999.

\bibitem{JouannaudKirchner84}
Jean-Pierre Jouannaud and Helene Kirchner.
\newblock Completion of a set of rules modulo a set of equations.
\newblock In {\em Proceedings of the 11th ACM SIGACT-SIGPLAN Symposium on
  Principles of Programming Languages}, POPL '84, pages 83--92, New York, NY,
  USA, 1984. ACM.

\bibitem{KL1}
M.~{Khovanov} and A.~D. {Lauda}.
\newblock {A diagrammatic approach to categorification of quantum groups I}.
\newblock {\em arXiv:0803.4121}, March 2008.

\bibitem{KL3}
M.~{Khovanov} and A.~D. {Lauda}.
\newblock {A diagrammatic approach to categorification of quantum groups III}.
\newblock {\em arXiv:0807.3250}, July 2008.

\bibitem{KHO14}
Mikhail Khovanov.
\newblock Heisenberg algebra and a graphical calculus.
\newblock {\em Fundamenta Mathematicae}, 225(0):169--210, 2014.

\bibitem{LAF03}
Yves Lafont.
\newblock Towards an algebraic theory of boolean circuits.
\newblock {\em Journal of Pure and Applied Algebra}, 184(2):257 -- 310, 2003.

\bibitem{LAN79}
Dallas Lankford.
\newblock On proving term rewriting systems are noetherian.
\newblock 1979.

\bibitem{LAU12}
Aaron~D. Lauda.
\newblock An introduction to diagrammatic algebra and categorified quantum
  {${\rm sl}(2)$}.
\newblock {\em Bull. Inst. Math. Acad. Sin. (N.S.)}, 7(2):165--270, 2012.

\bibitem{ROU08}
R.~{Rouquier}.
\newblock {2-Kac-Moody algebras}.
\newblock {\em arXiv:0812.5023}, December 2008.

\bibitem{BirmanWenzl89}
Joan S.~Birman and Hans Wenzl.
\newblock Braids, link polynomials and a new algebra.
\newblock 313:249--249, 05 1989.

\bibitem{SAV18}
Alistair Savage.
\newblock String diagrams and categorification.
\newblock preprint arXiv:1806.06873, 2018.

\bibitem{STR86}
Ross Street.
\newblock Limits indexed by category-valued {$2$}-functors.
\newblock {\em J. Pure Appl. Algebra}, 8(2):149--181, 1976.

\bibitem{STR87}
Ross Street.
\newblock The algebra of oriented simplexes.
\newblock {\em J. Pure Appl. Algebra}, 49(3):283--335, 1987.

\bibitem{TemperleyLieb71}
H.~N.~V. Temperley and E.~H. Lieb.
\newblock Relations between the 'percolation' and 'colouring' problem and other
  graph-theoretical problems associated with regular planar lattices: Some
  exact results for the 'percolation' problem.
\newblock {\em Proceedings of the Royal Society of London. Series A,
  Mathematical and Physical Sciences}, 322(1549):251--280, 1971.

\bibitem{VOO94}
Vincent van Oostrom.
\newblock Confluence by decreasing diagrams.
\newblock {\em Theor. Comput. Sci.}, 126(2):259--280, April 1994.

\end{thebibliography}
\end{small}
\end{document}